\documentclass[preprint]{imsart}

\usepackage[english]{babel}
\usepackage[T1]{fontenc}
\usepackage[utf8]{inputenc}
\usepackage{lmodern}

\usepackage{calc}
\usepackage[a4paper,portrait,left=3.cm,right=3.cm,top=3.cm,bottom=3.cm]{geometry}

\usepackage{amsmath}
\usepackage{amsfonts}
\usepackage{amssymb}
\usepackage{amsthm}
\usepackage{stmaryrd}
\usepackage{euscript}
\usepackage{xspace}
\usepackage{bm}
\usepackage{enumerate}
\usepackage{hyperref}
\usepackage{graphicx}
\usepackage{subcaption}
\usepackage[normalem]{ulem}

\usepackage[numbers]{natbib}

\arxiv{1508.00229}
\startlocaldefs

\newtheorem{lemma}{Lemma}[section]
\newtheorem{theorem}{Theorem}
\newtheorem{proposition}[theorem]{Proposition}
\newtheorem{corollary}[theorem]{Corollary}

\theoremstyle{definition}

\numberwithin{equation}{section}

\DeclareMathOperator{\supp}{supp}

\DeclareMathOperator{\tr}{tr}

\newcommand{\mb}{\textbf{m}\xspace}

\newcommand{\e}{\ensuremath{\mathrm{e}}\xspace}

\newcommand{\R}{\ensuremath{\mathbb{R}}\xspace}
\newcommand{\Q}{\ensuremath{\mathbb{Q}}\xspace}
\newcommand{\N}{\ensuremath{\mathbb{N}}\xspace}

\newcommand{\Ibb}{\ensuremath{\mathbb{I}}\xspace}

\newcommand{\Tbb}{\ensuremath{\mathbb{T}}\xspace}
\newcommand{\Ubb}{\ensuremath{\mathbb{U}}\xspace}
\newcommand{\Vbb}{\ensuremath{\mathbb{V}}\xspace}

\newcommand{\Nb}{\ensuremath{\mathbf{N}}\xspace}

\newcommand{\Ai}{\ensuremath{\mathcal A}\xspace}
\newcommand{\Bi}{\ensuremath{\mathcal B}\xspace}

\newcommand{\Di}{\ensuremath{\mathcal D}\xspace}
\newcommand{\Ei}{\ensuremath{\mathcal E}\xspace}

\newcommand{\Gi}{\ensuremath{\mathcal G}\xspace}
\newcommand{\Hi}{\ensuremath{\mathcal H}\xspace}
\newcommand{\Ii}{\ensuremath{\mathcal I}\xspace}
\newcommand{\Ji}{\ensuremath{\mathcal J}\xspace}
\newcommand{\Li}{\ensuremath{\mathcal L}\xspace}

\newcommand{\Ni}{\ensuremath{\mathcal N}\xspace}
\newcommand{\Oi}{\ensuremath{\mathcal O}\xspace}
\newcommand{\Pii}{\ensuremath{\mathcal P}\xspace}

\newcommand{\Ti}{\ensuremath{\mathcal T}\xspace}
\newcommand{\Vi}{\ensuremath{\mathcal V}\xspace}

\newcommand{\dimH}{\ensuremath{{\dim}_{\text{\normalfont\tiny H}}}\xspace}
\newcommand{\dimP}{\ensuremath{{\dim}_{\text{\normalfont\tiny P}}}\xspace}

\newcommand{\dimBu}{\ensuremath{\overline{\dim}_{\text{\normalfont\tiny B}}}\xspace}

\newcommand{\vset}{\ensuremath{{\emptyset}}\xspace}

\newcommand{\pth}[1]{(#1)}
\newcommand{\pthb}[1]{\bigl(#1\bigr)}
\newcommand{\pthB}[1]{\Bigl(#1\Bigr)}
\newcommand{\pthbb}[1]{\biggl(#1\biggr)}

\newcommand{\bkt}[1]{[#1]}
\newcommand{\bktb}[1]{\bigl[#1\bigr]}
\newcommand{\bktB}[1]{\Bigl[#1\Bigr]}
\newcommand{\bktbb}[1]{\biggl[#1\biggr]}

\newcommand{\brc}[1]{\{#1\}}
\newcommand{\brcb}[1]{\bigl\{#1\bigr\}}
\newcommand{\brcB}[1]{\Bigl\{#1\Bigr\}}
\newcommand{\brcbb}[1]{\biggl\{#1\biggr\}}

\newcommand{\bk}[1]{\langle#1\rangle}

\newcommand{\dt}{\ensuremath{\mathrm d}\xspace} 
\newcommand{\eqdef}{:=}

\newcommand{\ivoo}[1]{\ensuremath{(#1)}}
\newcommand{\ivoob}[1]{\ensuremath{\bigl(#1\bigr)}}

\newcommand{\ivof}[1]{\ensuremath{(#1]}}
\newcommand{\ivofb}[1]{\ensuremath{\bigl(#1\bigr]}}

\newcommand{\ivfo}[1]{\ensuremath{[#1)}}
\newcommand{\ivfob}[1]{\ensuremath{\bigl[#1\bigr)}}

\newcommand{\ivff}[1]{\ensuremath{[#1]}}
\newcommand{\ivffb}[1]{\ensuremath{\bigl[#1\bigr]}}
\newcommand{\ivffB}[1]{\ensuremath{\Bigl[#1\Bigr]}}

\newcommand{\iivff}[1]{\ensuremath{\llbracket#1\rrbracket}}

\newcommand{\abs}[1]{\lvert#1\rvert}

\newcommand{\ceil}[1]{\lceil#1\rceil}
\newcommand{\ceilb}[1]{\bigl\lceil#1\bigr\rceil}
\newcommand{\ceilB}[1]{\Bigl\lceil#1\Bigr\rceil}

\renewcommand{\Pr}{\ensuremath{\mathbb P}\xspace}

\newcommand{\pr}[2][]{\mathbb{P}#1\pth{#2}}
\newcommand{\prb}[2][]{\mathbb{P}#1\pthb{\hspace{1pt}#2\hspace{1pt}}}
\newcommand{\prB}[2][]{\mathbb{P}#1\pthB{#2}}

\newcommand{\esp}[2][]{\mathbb{E}#1\bkt{#2}}
\newcommand{\espb}[2][]{\mathbb{E}#1\bktb{\hspace{1pt}#2\hspace{1pt}}}
\newcommand{\espB}[2][]{\mathbb{E}#1\bktB{#2}}

\newcommand{\pthc}[2]{\pth{\hspace{1pt}#1\hspace{1.5pt}|\hspace{1.5pt}#2\hspace{1pt}}}
\newcommand{\pthcb}[2]{\pthb{\hspace{1pt}#1\bigm|#2\hspace{1pt}}}
\newcommand{\pthcB}[2]{\pthB{#1\Bigm|#2}}
\newcommand{\pthcbb}[2]{\pthbb{#1\biggm|#2}}

\newcommand{\indi}{\ensuremath{\mathbf{1}}\xspace}
\newcommand{\eps}{\varepsilon}

\newcommand{\littleo}{\text{o}}
\newcommand{\bigo}{\text{O}}

\newcommand{\vsp}{\vspace{.15cm}}

\endlocaldefs

\begin{document}

\begin{frontmatter}

\title{Uniform multifractal structure of stable trees}
\runtitle{Uniform multifractal structure of stable trees}

\begin{aug}
\author{\fnms{Paul} \snm{Balan\c{c}a}%
\thanksref{t2}%
\ead[label=e1]{paul.balanca@gmail.com}%
\ead[label=u1,url]{balancap.github.io}%
}%
\thankstext{t2}{Research supported by the Israel Science Foundation grant 1325/14 and the French Embassy in Israel.}%
\runauthor{Paul Balança}%
\affiliation{Technion}%
\address{
Faculty of Industrial Engineering and Management\\%
Technion Israel Institute of Technology\\%
Haifa 32000, Israël\\[1ex]%
\printead{e1}\\\printead{u1}%
}
\end{aug}

\begin{abstract}
In this work, we investigate the spectrum of singularities of random stable trees with parameter $\gamma\in(1,2)$. We consider for that purpose the scaling exponents derived from two natural measures on stable trees: the local time $\ell^a$ and the mass measure $\mb$, providing as well a purely geometrical interpretation of the latter exponent.
We first characterise the uniform component of the multifractal spectrum which exists at every level $a>0$ of stable trees and corresponds to large masses with scaling index $h\in\ivff{\tfrac{1+\gamma}{\gamma},\tfrac{\gamma}{\gamma-1}}$ for the mass measure (or equivalently $h\in\ivff{\tfrac{1}{\gamma},\tfrac{1}{\gamma-1}}$ for the local time). In addition, we investigate the distribution of vertices appearing at random levels with exceptionally large masses of index $h\in\ivfo{0,\tfrac{1+\gamma}{\gamma}}$. Finally, we discuss more precisely the order of the largest mass existing on any subset $\Ti(F)$ of a stable tree, characterising the former with the packing dimension of the set $F$.
\end{abstract}



\end{frontmatter}


\section{Introduction}

Continuous random trees have been a dynamic research topic in probability in recent years. Following the seminal work of \citet{Aldous-1991,Aldous-1991a} who defined the now celebrated \emph{Continuous Random Tree} (CRT), \citet{Duquesne.LeGall-2002,Duquesne.LeGall-2005} have introduced and developed the theory of (sub)critical Lévy trees, including stable trees, to provide continuous analogues to discrete Galton--Watson trees. Their definition of Lévy trees was later extended by \citet{Duquesne.Winkel-2007} to the supercritical case using a different approach. As pointed out in the initial work of \citet{LeGall.LeJan-1998}, these Lévy trees encode the complete genealogy of continuous state branching processes (CSBPs), and as a consequence, their law is characterised by CSBPs branching mechanism:
\begin{align*}
  \forall \lambda\geq 0;\quad \psi(\lambda) = \alpha\lambda + \beta\lambda^2 + \int_{(0,\infty)} \pi(\dt r) \pthb{ \e^{-\lambda r} - 1 +\lambda r },
\end{align*}
where $\alpha\geq0$, $\beta\geq 0$ and $\int_{\R_+} \pthb{r\wedge r^2} \pi(\dt r)<\infty$. Continuous random trees have proved to be a major research field in probability theory and are deeply connected to several other major topics such as superprocesses \cite{LeGall-1999,Duquesne.LeGall-2005}, fragmentation processes \cite{Haas.Miermont-2004,Goldschmidt.Haas-2010,Abraham.Delmas-2008} and planar maps \cite{LeGall.Miermont-2011,LeGall.Miermont-2012} to name but a few; consequently leading to a significant recent literature on the subject.

Random stable trees are particular instances of Lévy trees whose branching mechanism is given by $\psi(\lambda) = c\,\lambda^\gamma$, $\gamma\in\ivof{1,2}$, the specific case $\gamma=2$ corresponding to the quadratic branching and the CRT. In the framework of continuous trees, or $\R$-trees, the latter are seen as random metric spaces $(\Ti,d)$ where for any two vertices $\sigma$ and $\sigma'$ in $\Ti$, there is a unique arc with endpoints $\sigma$ and $\sigma'$. In addition, this arc is isometric to a compact interval of the real line. We usually denote by $h(\Ti)$ the height of the tree and by $\rho(\Ti)$ the distinguished vertex called the root, if the former is a rooted $\R$-tree. In the rest of this work, we will designate by $\Nb(\dt\Ti)$ the law of random stable trees, $\Nb(\dt\Ti)$ hence being a distribution with infinite mass on $\R$-trees. As presented by \citet{Duquesne.LeGall-2005}, $\Nb(\dt\Ti)$-a.e. at any level $a>0$ can be constructed a finite mass measure $\ell^a(\dt\sigma)$ called the \emph{local time} and carried by the level set
\begin{align*}
  \Ti(a) \eqdef \brcb{\sigma\in\Ti : d(\rho(\Ti),\sigma) = a}.
\end{align*}
Informally, $\ell^a(\dt\sigma)$ represents the mass distribution of the population of generation $a$ in tree. As observed by \citet[Th. 1.4.1]{Duquesne.LeGall-2002}, the local time happens to be closely related to the law of CSBPs through the so-called Ray--Knight theorem. In addition, the mapping $a\mapsto\ell^a$ is weakly càdlàg and its atoms correspond to vertices $\sigma\in\Ti$ with infinite multiplicity. One can derived from the local time a natural measure on the full tree. Namely, we define the \emph{mass measure} $\mb(\dt\sigma)$ on $\Ti$ by:
\begin{align*}
  \textbf{m}(\dt\sigma) = \int_0^\infty \ell^a(\dt\sigma)\,\dt a.
\end{align*}
As a consequence of the càdlàguity of the local time, the measure $\mb(\dt\sigma)$ is diffuse. Moreover, $\mb$ is carried by the set of leaves of $\Ti$, i.e. the set of the vertices $\sigma$ such that $\Ti\setminus\brc{\sigma}$ remains connected, and, on the contrary to the local time, it remains invariant under re-rooting of the tree.
\vsp

Several geometric properties of stable trees have already been discussed in the literature: \citet{Haas.Miermont-2004} and \citet{Duquesne.LeGall-2005} have presented the Hausdorff and packing dimensions of the tree and the level sets. Namely, for every $a>0$,
\begin{align*}
  \dimH \Ti(a) = \frac{1}{\gamma-1}\quad \Nb_a\pth{\dt\Ti}\text{-a.e.}\quad \text{and}\quad \dimH \Ti=\frac{\gamma}{\gamma-1}\quad \Nb(\dt\Ti)\text{-a.e.}
\end{align*}
where $\Nb_a\pth{\dt\Ti}\eqdef\Nb\pthc{\dt\Ti}{h(\Ti)>a}$. The question of the existence of exact Hausdorff and packing measures have been investigated on stable and Lévy trees (including the CRT) by \citet{Duquesne.LeGall-2006}, \citet{Duquesne.Reichmann.ea-2010} and \citet{Duquesne-2012}. In particular, Theorem 1.2 in the latter provides an asymptotic estimate on the small balls of the mass measure at almost every vertex:
\begin{align}  \label{eq:tree_small_balls0}
  \Nb\text{-a.e. for $\mb$ almost all }\sigma;\quad \liminf_{r\rightarrow 0} \frac{\mb(B(\sigma,r))}{g_\gamma(r)} = \gamma-1\quad\text{where } g_\gamma(r)\eqdef \frac{r^{\tfrac{\gamma}{\gamma-1}}}{\pth{\log\log 1/r}^{\tfrac{\gamma}{\gamma-1}}}.
\end{align}
This result has been extended recently by \citet{Duquesne.Wang-2014} who described the exceptionally small balls of the mass measure on the full tree. Namely,
\begin{align}  \label{eq:tree_small_balls1}
  \liminf_{r\rightarrow 0} \frac{1}{f_\gamma(r)} \inf_{\sigma\in\Ti} \mb\pthb{B(\sigma,r)} \geq k_\gamma \quad\text{where } f_\gamma(r) \eqdef \frac{ r^{\tfrac{\gamma}{\gamma-1}} }{ \pth{\log1/r}^{\tfrac{\gamma}{\gamma-1}} }
\end{align}
and $k_\gamma$ is a positive constant only depending on $\gamma$.\vsp

The two previous results give a very precise picture of the small balls asymptotic of the mass measure. In this work, we are interested in investigating an alternative behaviour: the appearance of exceptional large masses on the tree, either described by the local time or the mass measure. As presented in \cite{Duquesne.Wang-2014}, fluctuations of small masses of the mass measure are of a logarithmic order. On the other hand, we will see in the rest of this work that fluctuations of large masses are of bigger order, implying in particular that the interesting and relevant quantity to analyse the former is the pointwise scaling exponent of the local time or the mass measure. More precisely, the latter idea refers to the common notion of \emph{pointwise Hölder exponent} in the fractal geometry literature. It has been introduced to characterise the asymptotic local mass of measures, or the local fluctuations of functions (see for instance \cite{Meyer-1998} for a broad review on the subject). Namely, in the case of the local time and the mass measure, the pointwise exponents of $\ell^a(\dt\sigma)$ and $\mb(\dt\sigma)$ are formally defined at every vertex $\sigma\in\Ti$ by
\begin{align}
  \alpha_\ell(\sigma,\Ti) \eqdef \liminf_{r\rightarrow 0}  \frac{\log \ell^a(B(\sigma,r))}{\log r}\quad\text{and}\quad \alpha_\mb(\sigma,\Ti) \eqdef \liminf_{r\rightarrow 0}  \frac{\log \mb(B(\sigma,r))}{\log r},
\end{align}
where $a=d(\rho(\Ti),\sigma)$.

It is clear from this definition that these pointwise exponents aim to capture the scale of the largest balls as the radius tends to zero. In our opinion, it remains interesting to study the behaviour of both the local time and the mass measure as the two concepts bring different, and complementary, information about the heterogeneous mass distribution on stable trees. From a pure tree point of view, the mass measure $\mb$ may seem more relevant as it remains invariant under re-rooting of the tree (and thus, the exponent $\alpha_\mb(\sigma,\Ti)$ as well) and provides a description of the local mass in the full neighbourhood of a vertex.
Nevertheless, the asymptotic behaviour of the local time is also interesting if one looks at stable trees as the encoding of the genealogy of a continuous state branching process (CSBP). In this case, the pointwise exponent $\alpha_\ell(\sigma,\Ti)$ gives a precise picture of the local mass of a population at a given generation. Moreover, this description also has direct applications on stochastic models such as superprocesses which are derived from continuous Lévy trees (we refer \cite{Duquesne.LeGall-2005} for a more complete construction).\vsp

In addition to the study of the scaling exponents derived from the local time and the mass measure, we are also interested in giving a more geometrical interpretation of the non-homogeneous structure of stable trees by studying the local branching in the neighbourhood of a vertex. \citet{Duquesne.LeGall-2005} already gave an insight of the latter by studying a common graph quantity: the multiplicity $n(\sigma)$ of a vertex. Namely, at every $\sigma\in\Ti$, $n(\sigma)$ is defined by
\begin{align}  \label{eq:multiplicty_trees}
  n(\sigma) \eqdef \#\brcb{ \text{connected components in } \Ti\setminus\brc{\sigma} }.
\end{align}
In the case of stable of trees $\gamma\in\ivoo{1,2}$, they proved that at every $\sigma\in\Ti$, $n(\sigma)\in\brc{1,2,\infty}$, the last case corresponding to the countable atoms of the local time $a\mapsto\ell^a$. Moreover, it happens that almost all vertices are leaves, i.e. $\textbf{m}(\dt\sigma)$-a.e., $n(\sigma) = 1$. The multiplicity describes precisely the branching behaviour at a vertex itself, but gives little information on its neighbourhood, and in particular the existence of large sub-trees branching closely to $\sigma$. To capture a more local information, we extend and modify slightly the original definition of the multiplicity by defining for any $\delta>0$ and at any $\sigma\in\Ti$:
\begin{align}  \label{eq:branching_exponent0}
  n_b(\sigma,\delta) \eqdef \#\brcb{ \text{connected components diameter $>\delta$ in } \Ti\setminus \overline{B}(\sigma,\delta) }.
\end{align}
The quantity $n_b(\sigma,\delta)$ therefore encapsulates information on the local branching behaviour around $\sigma$ at scale $\delta$. Then, similarly to the local time and the mass measure, we may define a \emph{branching exponent} which capture the asymptotic order of $n_b(\sigma,\delta)$ in the neighbourhood of a vertex: for every $\sigma\in\Ti$,
\begin{align}  \label{eq:branching_exponent1}
  \alpha_b(\sigma,\Ti) \eqdef \limsup_{r\rightarrow 0} \frac{\log n_b(\sigma,r)}{\log 1/r}.
\end{align}
The branching exponent $\alpha_b(\sigma,\Ti)$ complements the usual multiplicity $n(\sigma)$ by providing a classification of vertices in the tree which takes into account the local branching structure.\vsp

As we aim to study the fluctuations of the scaling exponents previously introduced, we may first look more closely at how these are related in the case of stable trees.
\begin{proposition}  \label{prop:scaling_exponents}
  Suppose $\gamma\in\ivoo{1,2}$. The local time $\alpha_\ell(\sigma,\Ti)$ and the mass measure $\alpha_\mb(\sigma,\Ti)$ exponents are related as following: $\Nb\pth{ \dt \Ti }$-a.e.
  \begin{align}  \label{eq:ltime_mass_exponents}
    \forall \sigma\in\Ti;\quad \alpha_\ell(\sigma,\Ti) \leq \frac{1}{\gamma-1} \quad\Longrightarrow\quad \alpha_\mb(\sigma,\Ti) \leq \alpha_\ell(\sigma,\Ti) + 1.
  \end{align}
  Moreover, the mass measure $\alpha_\mb(\sigma,\Ti)$ and the branching $\alpha_b(\sigma,\Ti)$ exponents are equivalent on stable trees: $\Nb\pth{ \dt \Ti }$-a.e.
  \begin{align}  \label{eq:branching_mass_exponents}
    \forall \sigma\in\Ti;\quad \alpha_\mb(\sigma,\Ti) = \frac{\gamma}{\gamma-1} - \alpha_b(\sigma,\Ti).
  \end{align}
\end{proposition}
An interesting consequence of Equation~\eqref{eq:branching_mass_exponents} is to provide a purely geometrical interpretation of the mass measure exponent and motivate furthermore its study as a form of extension of the multiplicity. Given the previous equality, it is natural to wonder if the connection \eqref{eq:ltime_mass_exponents} between mass measure and local time scaling exponents may also be stronger. In fact, as a consequence of the branching property of stable trees (see Section~\ref{sec:notations}), the previous result can not be improved since with positive probability, one may have $\mb(B(\sigma,\delta))\geq \delta^h$ and $\ell^a(B(\sigma,\delta))\asymp \delta^{\tfrac{1}{\gamma-1}}$. Consequently, even though the previous two exponents do not behave completely independently, it remains relevant and interesting to obtain a complete classification of both.

The analysis of the fluctuations of pointwise-like Hölder exponents is usually called \emph{multifractal analysis}. This research topic has attracted attention in probability theory for now several years on many different subjects: (fractional) Lévy processes \cite{Jaffard-1999,Durand.Jaffard-2012,Balanca-2014a,Neuman-2014}, spatial Brownian motion \cite{Dembo.Peres.ea-2000}, Galton--Watson trees \cite{Moerters.Shieh-2002,Moerters.Shieh-2008}, beta-coalescents \cite{Berestycki.Berestycki.ea-2007} and superprocesses \cite{Perkins.Taylor-1998,Mytnik.Wachtel-2015} among them. This formalism happens to be relevant when a scaling exponent tends to fluctuate erratically and one can not determine an almost sure behaviour at every time (or vertex in our case). To characterise nevertheless the variations of such exponents, multifractal analysis is interested in the fractal structure of the iso-Hölder sets. Namely, in the case of stable trees, we define for any index $h\geq 0$:
\begin{align}
  E_\ell(h,\Ti) = \brcb{\sigma\in\Ti : \alpha_\ell(\sigma,\Ti) = h} \quad\text{and}\quad E_\mb(h,\Ti) = \brcb{\sigma\in\Ti : \alpha_\mb(\sigma,\Ti) = h}.
\end{align}
The \emph{multifractal spectrum}, or \emph{spectrum of singularities}, is then commonly defined as the function $h\mapsto \dimH E_\star(h,\Ti)$. In the case of stable trees, it therefore provides an insight of the heterogeneous mass distribution by indicating the relative proportion of vertices with a given scaling exponent.

It is known from simple calculations that the typical mass of a ball $B(\sigma,r)$ is of order $r^{\tfrac{1}{\gamma-1}}$ for the local time and $r^{\tfrac{\gamma}{\gamma-1}}$ for the mass measure. As a consequence, we expect that most vertices, in a fractal dimension sense, have local time and mass measure exponents respectively equal to $\tfrac{1}{\gamma-1}$ and $\tfrac{\gamma}{\gamma-1}$, and that larger masses may only appear at exceptional vertices. In the case of the CRT $\gamma=2$, \citet{Duquesne.Wang-2014} have proved the existence of positive constants $k,K$ such that $\Nb$-a.e.
\begin{align}
  k \leq \liminf_{r\rightarrow 0} \frac{1}{f_2(r)} \sup_{\sigma\in\Ti} \mb\pthb{B(\sigma,r)} \leq \limsup_{r\rightarrow 0} \frac{1}{f_2(r)} \sup_{\sigma\in\Ti} \mb\pthb{B(\sigma,r)} \leq K,
\end{align}
inducing in particular that for every $\Nb$-a.e. $\sigma\in\Ti$, $\alpha_\mb(\sigma,\Ti) = 2$. 

The behaviour of scaling exponents on stable trees with $\gamma\in\ivoo{1,2}$ happens to be more complex and interesting.
A first result on the subject has been obtained by \citet{Berestycki.Berestycki.ea-2007} who determined the spectrum of singularities of the local time of stable trees at fixed level $a$. Namely, for any index $h\in\ivff{\tfrac{1}{\gamma},\tfrac{1}{\gamma-1}}$ and every level $a>0$
\begin{align}  \label{eq:spectrum_fix}
  \dimH \,\pthb{ E_\ell(h,\Ti)\cap \Ti(a) } = \gamma h - 1 \quad\Nb_a\text{-a.e.}
\end{align}
More broadly speaking, this question is also closely related to the literature investigating the multifractal aspect of superprocesses (we refer to the work of \citet{Perkins.Taylor-1998} on super-Brownian motion and the recent paper of \citet{Mytnik.Wachtel-2015} studying the density of stable sBm). Compared to these articles, one important aspect of our work is to present a uniform description of the mass distribution on random stable trees, i.e. simultaneously for all levels and all scaling indices, and therefore describe as well exceptional behaviours appearing at random levels.

Before stating our main results, we introduce a few additional notations: for every nonempty set $F$ in the interval $\ivoo{0,\infty}$, we denote by $\Ti(F)$ the following subset of the tree:
\begin{align*}
  \Ti(F) = \bigcup_{a\in F} \Ti(a).
\end{align*}
We will say that $F$ is \emph{regular} if its Hausdorff and packing dimensions coincide. Moreover, $F$ is said to satisfy a \emph{strong Frostman's lemma} if there exists a probability measure $\mu_F$ on $F$ (i.e. $\supp F\subseteq F$) such that for every $\eps>0$,
\begin{align}  \label{eq:strong_frostman}
  \exists r_0>0,\quad  \forall x\in F,\ \forall r\in\ivoo{r,r_0};\quad \mu_F\pthb{B(x,r)} \leq r^{\dimH F - \eps}.
\end{align}
The previous assumption is stronger than the result of the celebrated Frostman's lemma (see \cite[Prop. 4.11]{Falconer-2003}), and thus not automatically verified by any $F$. A large class of fractal sets nevertheless satisfies condition~\eqref{eq:strong_frostman}, and in particular, it encompasses any Borel set $F$ such that $\Hi^\varphi(F)\in\ivoo{0,\infty}$ for a gauge function $\varphi$ verifying: $\liminf_{r\rightarrow 0} \tfrac{\log \varphi(r)}{\log r} = \dimH F$.
Finally, we also use in the rest of the article the classic convention $\dimH F < 0$ if and only if the set $F$ is empty.\vsp

We begin by presenting the uniform component of the multifractal spectrum of the local time and the mass measure.
\begin{theorem}  \label{th:levy_tree_upper_spectrum1}
  Suppose $\gamma\in\ivoo{1,2}$. $\Nb\pth{ \dt \Ti }$-a.e. for any nonempty regular set $F\subset\ivoo{0,h(\Ti)}$, the spectrum of singularities of the local time on $\Ti(F)$ is equal to
  \begin{align}
    \forall h\in\ivffb{\tfrac{1}{\gamma}, \tfrac{1}{\gamma-1}};\quad \dimH \,\pthb{ E_\ell(h,\Ti)\cap \Ti(F) } = \gamma h - 1 + \dimH F.
  \end{align}
  Moreover, the mass measure multifractal spectrum is equal to
  \begin{align}
    \forall h\in\ivffb{\tfrac{1+\gamma}{\gamma}, \tfrac{\gamma}{\gamma-1}};\quad \dimH \,\pthb{ E_\mb(h,\Ti) \cap \Ti(F) } = \gamma (h-1) - 1 + \dimH F.
  \end{align}
  In particular, $E_\ell(\tfrac{1}{\gamma},\Ti)\cap\Ti(a)$ and $E_\mb(\tfrac{1+\gamma}{\gamma},\Ti)\cap\Ti(a)$ are non-empty for every $a\in\ivoo{0,h(\Ti)}$. Finally, for any $h>\tfrac{\gamma}{\gamma-1}$, $E_\mb(h,\Ti)=\vset$.
\end{theorem}
We may observe that even though the local time and mass measure exponent do not coincide in general, Theorem~\ref{th:levy_tree_upper_spectrum1} nevertheless proves that they share a similar spectrum of singularities.
Concerning the local time itself, Theorem~\ref{th:levy_tree_upper_spectrum1}  clearly extends the spectrum of singularities \eqref{eq:spectrum_fix} presented by \citet{Berestycki.Berestycki.ea-2007}, describing a result which stands for any level $a$ and index $h$. Note that the previous article relies heavily on the work of \citet{Moerters.Shieh-2002} and \citet{Moerters.Shieh-2008} who have investigated the multifractal structure of the branching measure on the boundary of supercritical Galton--Watson trees. Due to the technical requirements of a uniform statement, the proof of Theorem~\ref{th:levy_tree_upper_spectrum1}  makes use of different arguments and techniques.

On the other hand, the spectrum of singularities of $\mb$ completes the study initiated by \citet{Duquesne-2012,Duquesne.Wang-2014} on the asymptotic behaviour of the mass measure by providing a description of the distribution of exceptional large masses.

Note that as a corollary, Theorem~\ref{th:levy_tree_upper_spectrum1} provides a uniform characterisation of the Hausdorff dimension images sets of stable trees, extending the result of \citet[Th. 5.5]{Duquesne.LeGall-2005}.
\begin{corollary}  \label{cor:hdim_level_sets}
  Suppose $\gamma\in\ivoo{1,2}$. $\Nb(\dt\Ti)$-a.e.
  \begin{align}  \label{eq:cor_dimH_images}
    \text{for any Borel set } F\subset\ivoo{0,h(\Ti)};\quad \dimH \Ti(F) = \dimH F + \frac{1}{\gamma-1}.
  \end{align}
\end{corollary}

It seems natural to wonder if the formulas presented Theorem~\ref{th:levy_tree_upper_spectrum1} can be extended to the case $h<\tfrac{1}{\gamma}$ (resp. $h<\tfrac{1+\gamma}{\gamma}$) for the local time (resp. the mass measure). Unfortunately, as a consequence of \cite{Berestycki.Berestycki.ea-2007} and Fubini's theorem, we already know that $\Nb$-a.e.
\begin{align}
  \text{for almost all }a>0,\ \forall h<\tfrac{1}{\gamma};\quad E_\ell(h,\Ti)\cap \Ti(a) = \vset.
\end{align}
Therefore, a uniform statement such as Theorem~\ref{th:levy_tree_upper_spectrum1} can not be obtained once $h<\tfrac{1}{\gamma}$, the same remark existing as well for the mass measure scaling exponent. Nevertheless, a weaker form of result still holds for a fixed fractal set $F$, as presented in the next theorem.
\begin{theorem}  \label{th:levy_tree_upper_spectrum2}
  Suppose $\gamma\in\ivoo{1,2}$ and $F\subset\ivoo{0,\infty}$ is a Borel set such that for every $a>0$, $F\cap\ivoo{0,a}$ is regular and satisfies the strong Frostman's lemma \eqref{eq:strong_frostman}.

  Then, $\Nb\pth{ \dt \Ti }$-a.e., the spectrum of singularities of the local time on $\Ti(F)$ is equal to
  \begin{align}
    \forall h\in\ivofb{\tfrac{1-\dimH F_{|\Ti}}{\gamma}, \tfrac{1}{\gamma-1}};\quad \dimH \,\pthb{ E_\ell(h,\Ti)\cap \Ti(F) } = \gamma h - 1 + \dimH F_{|\Ti}.
  \end{align}
  where $F_{|\Ti}\eqdef F\cap\ivoo{0,h(\Ti)}$.
  In addition, the multifractal spectrum of the mass measure is equal to $\Nb\pth{ \dt \Ti }$-a.e.
  \begin{align}
    \forall h\in\ivofb{\tfrac{\gamma+1-\dimH F_{|\Ti}}{\gamma}, \tfrac{\gamma}{\gamma-1}};\quad \dimH \,\pthb{ E_\mb(h,\Ti) \cap \Ti(F) } = \gamma (h-1) - 1 + \dimH F_{|\Ti}.
  \end{align}
  Finally, $\Nb\pth{ \dt \Ti }$ for every level $a>0$, the sets $E_\ell(h,\Ti)\cap\Ti(a)$ and $E_\mb(h,\Ti)\cap\Ti(a)$ are either empty or have zero Hausdorff dimension.
\end{theorem}
We conjecture that Theorem~\ref{th:levy_tree_upper_spectrum2} remains valid without the \emph{strong Frostman's lemma} assumption (and in fact, the previous statement holds without on the collection $\brcb{\sigma\in\Ti(F) : \alpha_\star(\sigma,\Ti) \geq h}$). Unfortunately, we have not been able to drop this technical assumption in our proof of the lower bound. On the other hand, we may observe that extending the multifractal spectrum presented in Theorems~\ref{th:levy_tree_upper_spectrum1} and \ref{th:levy_tree_upper_spectrum2} to the general case where $F$ is not regular is known to be not trivial; the question remaining open on limsup random fractals (we refer to the works of \citet{Khoshnevisan.Peres.ea-2000,Zhang-2013} for partial answers on the subject).\vsp

Even though Theorems~\ref{th:levy_tree_upper_spectrum1} and \ref{th:levy_tree_upper_spectrum2} present a very similar spectrum, we may observe that they emphasize two very different configurations. In the first case, the uniformity of the result proves that large masses of order $h\in\ivffb{\tfrac{1}{\gamma},\tfrac{1}{\gamma-1}}$ (resp. $h\in\ivffb{\tfrac{1+\gamma}{\gamma},\tfrac{\gamma}{\gamma-1}}$) exist at every level with the same Hausdorff dimension, hence allowing to present a strong uniform statement. On the other hand, as a corollary of the proof of Theorem~\ref{th:levy_tree_upper_spectrum2}, we get: $\Nb\pth{ \dt \Ti }$ for any $h\in\ivoob{\tfrac{1-\dimH F_{|\Ti}}{\gamma}, \tfrac{1}{\gamma}}$:
\begin{align}
  \dimH \brcb{a\in F : E_\ell(h,\Ti)\cap\Ti(a)\neq\vset } = \gamma h - 1 + \dimH F_{|\Ti} < 1,
\end{align}
with an analogue statement holding as well on the mass measure.
In other words, large masses of order $h<\tfrac{1}{\gamma}$ (resp. $h<\tfrac{1+\gamma}{\gamma}$) appear only at exceptional random levels, making impossible to expect a uniform description. Even though a strong result of the form of Theorem~\ref{th:levy_tree_upper_spectrum1} can not be obtained, it remains an open question whether a description similar to the work of \citet{Kaufman-1989} on Brownian motion is tractable: for any regular set $F$, the spectrum of singularities holds for almost all $a>0$ on the set $\Ti(F+a)$.
\vsp


Finally, as a last result, we are interested in studying more precisely the smallest scaling exponent appearing on stable trees (or equivalently, the largest mass), improving the description provided in Theorem~\ref{th:levy_tree_upper_spectrum2} and giving a sufficient condition on the set $F$ under which the lower bound is realised.
\begin{theorem}  \label{th:levy_tree_upper_spectrum3}
  Suppose $\gamma\in\ivoo{1,2}$ and $F\subset\ivoo{0,\infty}$ is an analytic set. Then, $\Nb\pth{ \dt \Ti }$-a.e.,
  \begin{align}  \label{eq:stable_usp_hitting0}
    \inf_{\sigma\in \Ti(F)} \alpha_\ell(\sigma,\Ti)  = \frac{1-\dimP\, F_{|\Ti} }{\gamma} \quad\text{and}\quad \inf_{\sigma\in \Ti(F)} \alpha_\mb(\sigma,\Ti) = \frac{\gamma+1-\dimP \,F_{|\Ti} }{\gamma}.
  \end{align}
  where we recall that $F_{|\Ti}\eqdef F\cap\ivoo{0,h(\Ti)}$.

  In addition, if $F$ is an analytic set such that for every $a>0$, $F\cap\ivoo{0,a}$ is empty or has positive and finite packing measure, then $\Nb\pth{ \dt \Ti }$-a.e. the infimum is realized in $\Ti(F)$:
  \begin{align}
    E_\ell\pthB{\tfrac{1-\dimP\, F_{|\Ti} }{\gamma},\Ti}\cap\Ti(F)\neq\vset \quad\text{and}\quad E_\mb\pthB{\tfrac{\gamma+1-\dimP\, F_{|\Ti} }{\gamma},\Ti}\cap\Ti(F)\neq\vset.
  \end{align}
\end{theorem}
Note that by studying the specific question of the infimum of the scaling exponents on stable trees, we are able to weaken the assumption on the set $F$. In addition, since $\Nb\pth{h(\Ti) > a} > 0$ for any level $a>0$, Theorem~\ref{th:levy_tree_upper_spectrum3} induces estimates on the following hitting probabilities:
\begin{align}  \label{eq:stable_usp_hitting3}
  \forall h\geq 0;\quad \Nb\pthb{ E_\ell(h,\Ti)\cap \Ti(F) \neq \vset }
  \begin{cases}
     > 0 \quad&\text{if }\dimP F > 1-\gamma h, \\
    =0 &\text{if }\dimP F < 1-\gamma h.
  \end{cases}
\end{align}
An analogue statement holds as well on the mass measure.

Equation~\eqref{eq:stable_usp_hitting3} is consistent with the fact that at fixed level $a>0$, there is no vertex with large mass of order $h\in\ivfo{0,\tfrac{1}{\gamma}}$ (resp. $h\in\ivfo{0,\tfrac{1+\gamma}{\gamma}}$). In addition, we observe that the smallest Hölder index appearing on $\Ti(F)$ is properly characterised by the packing dimension of the set $F$. The appearance of the latter quantity could be expected since \citet{Khoshnevisan.Peres.ea-2000} have proved that the packing dimension is the proper notion to characterise hitting probabilities of a large class of limsup sets such as fast points of Gaussian processes. A related property has also been exhibited by \citet{Moerters-2001} on fast points of super-Brownian motion.\vsp

Finally, as a last result, we also present the packing dimension of the iso-Hölder sets.
\begin{proposition}  \label{prop:packing_dim_spectrum}
  Suppose $\gamma\in\ivoo{1,2}$. The packing dimension of iso-Hölder sets satisfies $\Nb\pth{ \dt \Ti }$-a.e. for all levels $a\in\ivoo{0,h(\Ti)}$ and any $h\in\ivffb{\tfrac{1}{\gamma}, \tfrac{1}{\gamma-1}}$,
  \begin{align}
    \dimP \,\pthb{ E_\ell(h,\Ti)\cap \Ti(a) } = \frac{1}{\gamma-1}\quad\text{and}\quad \dimP \,\pthb{ E_\mb(h+1,\Ti)\cap \Ti(a) } = \frac{1}{\gamma-1}.
  \end{align}
\end{proposition}
Note that the previous result could also be expected as it is a common property of random limsup sets to have full packing dimension. Interestingly, the equality still holds for the limit case $h=\tfrac{1}{\gamma}$ (resp. $h=\tfrac{1+\gamma}{\gamma}$) despite the zero Hausdorff dimension. As a corollary of the previous statement, we extend uniformly the packing dimension of level sets obtained by \citet{Duquesne.LeGall-2005}: $\Nb(\dt\Ti)$-a.e.
\begin{align}  \label{eq:cor_dimP_images}
  \forall a\in\ivoo{0,h(\Ti)};\quad \dimP \Ti(a) = \frac{1}{\gamma-1}.
\end{align}

The rest of the paper is organised as following: we start by recalling important notations and results on stable trees in Section~\ref{sec:notations} and we present several technical lemmas on the tail asymptotic of CSBPs, the local time and the mass measure of stable trees in Section~\ref{sec:preliminaries}. The proof of our three main theorems is then divided into two parts in Section~\ref{sec:stable_trees_spectrum}: the relatively easy upper bound is first presented (Subsection~\ref{ssec:stable_trees_ub}), and then follows in Subsection~\ref{ssec:stable_trees_lb} the more delicate estimate of the lower bound of the multifractal spectrum.


\section{Stable trees: notations and main properties}  \label{sec:notations}

We begin by recalling a few common notations and various results on stable trees. As presented by \citet{Duquesne.LeGall-2002,Duquesne.LeGall-2005}, Lévy trees are encoded by excursions of a continuous non-negative process $(H_t)_{t\geq 0}$ named the \emph{height process}. Note that in the CRT case, $H$ is simply a reflected Brownian motion. We denote by $N(\dt H)$ the \emph{excursion measure}, i.e. meaning that $N$ is a measure on $C(\R_+,\R_+)$ (non-negative continuous functions) such that $H(0)=0$ and $H(t)=0$ for every $t\geq\zeta$ where $\zeta<\infty$ denotes the lifetime of an excursion $\zeta=\inf\brc{t>0 : H(t)=0}$. We refer to \cite{Duquesne.LeGall-2002} for a more complete presentation on the subject.

Under $N(\dt H)$, the excursion $(H_t)_{0\leq t\leq\zeta}$ is the depth-first exploration process of a continuous tree that is defined as a quotient metric space. For that purpose, we introduce the equivalence relation $s\sim_H t$ if and only if $d_H(s,t) = 0$, where $d_h$ denotes the following pseudo-distance:
\begin{align*}
  d_H\pthb{ s,t } = H_s+H_t - 2\min_{s\wedge t\leq u\leq s\vee t} H_u.
\end{align*}
Stable trees are then defined as the quotient metric space:
\begin{align*}
  (\Ti,d) \eqdef \pthb{\ivff{0,\zeta}/\sim_H, d_H}.
\end{align*}
We denote by $p_H:\ivff{0,\zeta}\rightarrow\Ti$ the canonical projection. The latter being continuous, $(\Ti,d)$ is therefore a random connected compact metric space. More precisely, as presented in \cite[Th. 2.1]{Duquesne.LeGall-2005}, $(\Ti,d)$ is an $\R$-tree, i.e. a metric space such that for every $\sigma,\sigma'\in\Ti$
\begin{enumerate}[\it(i)]
  \item There is a unique isometry $f_{\sigma,\sigma'}$ from $\ivffb{0,d(\sigma,\sigma')}$ into $\Ti$ such $f_{\sigma,\sigma'}(0)=\sigma$ and $f_{\sigma,\sigma'}(d(\sigma,\sigma'))=\sigma'$. We set $\iivff{\sigma,\sigma'}=f_{\sigma,\sigma'}\pthb{\ivffb{0,d(\sigma,\sigma')}}$, that is the geodesic joining $\sigma$ to $\sigma'$;\vsp
  \item If $g:\ivff{0,1}\rightarrow\Ti$ is continuous injective, then $g([0,1]) = \iivff{g(0),g(1)}$.
\end{enumerate}
We refer to \cite{Dress.Moulton.ea-1996,Evans.Pitman.ea-2006,Evans-2010} for a more detailed overview on the topic of (random) \R-trees.

As previously outlined in the introduction, from the construction of a local time on the height process is deduced the existence a local time $\ell^a(\dt\sigma)$ on every level set $\Ti(a)$. \citet{Duquesne.LeGall-2006} have determined the joint law of the total mass of local time $\bk{\ell^a}\eqdef\bk{\ell^a,\indi}$ and the mass measure $\mb(B(\rho,a))$ under $\Nb$. Namely, setting for any $a,\lambda,\mu\in\ivfo{0,\infty}$
\begin{align}
  \kappa_a(\lambda,\mu) \eqdef \Nb\pthb{1 - \e^{-\mu\bk{\ell^a} - \lambda\mb(B(\rho,a))}},
\end{align}
the map $a\mapsto\kappa_a(\lambda,\mu)$ is the unique solution of the ordinary differential equation:
\begin{align}  \label{eq:diff_kappa}
  \frac{\partial\kappa_a}{\partial a}(\lambda,\mu) = \lambda - \kappa_a(\lambda,\mu)^\gamma\quad\text{and}\quad \kappa_0(\lambda,\mu) = \mu.
\end{align}
Note that if $\mu=\lambda^{1/\gamma}$, $\kappa_a(\lambda,\mu)=\lambda^{1/\gamma}$. Additionally, if $\mu\neq\lambda^{1/\gamma}$, a change of variable shows that $\kappa_a(\lambda,\mu)$ solves the integral equation
\begin{align}  \label{eq:int_kappa}
  \int_{\mu}^{\kappa_a(\lambda,\mu)} \frac{\dt u}{\lambda - u^\gamma} = a.
\end{align}
We may deduce from Equation~\eqref{eq:int_kappa} a couple of properties on $\kappa$: $\kappa_{a+b}(\lambda,\mu) = \kappa_a(\lambda,\kappa_b(\lambda,\mu))$ and the scaling property $c^{\tfrac{1}{\gamma-1}}\kappa_a\pthb{c^{-\tfrac{\gamma}{\gamma-1}}\lambda,c^{-\tfrac{\gamma}{\gamma-1}}\mu} = \kappa_{a/c}(\lambda,\mu)$. The explicit expression of $\kappa$ seems difficult to obtain in general. Nevertheless, considering the single law of $\bk{\ell^a}$, one may deduce from \eqref{eq:int_kappa} the equality:
\begin{align}
  u_a(\mu) \eqdef \kappa_a(0,\mu)=\Nb\pthb{1 - \e^{-\mu \bk{\ell^a}}} = \pthb{(\gamma-1)a + \mu^{1-\gamma} }^{-\tfrac{1}{\gamma-1}},
\end{align}
This result has been originally described in the foundational work of \citet[Th 1.4.1]{Duquesne.LeGall-2002} as a Ray--Knight theorem, connecting the law of local time of Lévy trees to the Laplace transform of continuous state branching processes (CSBPs). Informally, one may see the local time $\bk{\ell^a}$ of a tree as a CSBP starting from a single individual.

We may also recall that $\Nb(\dt\Ti)$-a.e., the local time $a\mapsto\ell^a$ is càdlàg for the weak topology and $\bk{\ell^a} > 0$ if and only if $h(\Ti)>a$, where the latter denotes the total height of the tree: $h(\Ti) = \sup\brc{d(\rho(\Ti),\sigma) : \sigma\in\Ti}$. As a simple consequence of \eqref{eq:int_kappa}, the measure of the former event is given by
\begin{align}
  \forall a\in\ivoo{0,\infty};\quad \Nb\pth{ \bk{\ell^a} > 0 } = \pthb{(\gamma-1)a}^{-\tfrac{1}{\gamma-1}} \eqdef v(a)\quad\text{where $v(\cdot)$ solves } \int_{v(a)}^\infty \frac{\dt u}{u^\gamma} = a.
\end{align}
For any $a>0$, the conditional probability measure $\Nb\pthc{\,\cdot\,}{\bk{\ell^a} > 0}$ is usually denoted $\Nb_a$. Based on the previous expressions, the law of $\bk{\ell^a}$ is explicit under $\Nb_a$:
\begin{align}  \label{eq:laplace_tr_local_time}
  \forall a,\mu\in\ivoo{0,\infty};\quad \Nb_a\pthb{\e^{-\mu \bk{\ell^a}}} = 1 - \pthbb{ 1 + \frac{1}{(\gamma-1)a \mu^{\gamma-1}} }^{-\tfrac{1}{\gamma-1}}.
\end{align}

Recall that for any $\sigma,\sigma'\in\Ti$, $\iivff{\sigma,\sigma'}$ stands for the unique geodesic between $\sigma$ and $\sigma'$. Then, we may define the subtree $\Ti_\sigma$ stemming from $\sigma\in\Ti$ as following:
\begin{align*}
  \forall \sigma\in\Ti;\quad \Ti_\sigma = \brcb{ \sigma'\in\Ti : \sigma\in\iivff{\rho(\Ti),\sigma'} }.
\end{align*}
For all $a,\delta\in\ivoo{0,\infty}$, we also introduce the subset $\Ti(a,\delta) = \brcb{\sigma\in\Ti(a) : h(\Ti_\sigma) > \delta}\subset\Ti(a)$. Since $\Ti$ is a compact space, $\Ti(a,\delta)$ is a finite subset of $\Ti(a)$. In addition, we denote by $Z(a,\delta) \eqdef \# \Ti(a,\delta)$ its cardinal and  by $\Tbb(a,\delta)$ the collection of subtrees rooted at level $a$ and higher than $\delta$:
\begin{align*}
  \Tbb(a,\delta) = \brc{ \Ti_\sigma : \sigma \in \Ti(a,\delta)}\subset \Tbb
\end{align*}
where $\Tbb$ stands for the set of all equivalence classes of rooted compact
\R-trees (two rooted R-trees are called equivalent if there is a root-preserving isometry mapping the two). Considering the limit $\delta\rightarrow 0$, we also define $\Ti(a,0)$ and $\Tbb(a,0)$ which stand for the collection of all subtrees rooted at level $a$. For any level $a>0$, we also designate by $\tr(a)$ the truncated tree above $a$: $\tr(a) = \brcb{\sigma\in\Ti : d(\rho(\Ti),\sigma) \leq a }$.

Note that even though we usually omit the dependency in the random term $\Ti$ in the previous notations, the latter will be added when not completely obvious, hence writing in this case $\Ti(a,\delta,\Ti)$, $Z(a,\delta,\Ti)$, $\Tbb(a,\delta,\Ti)$, $\dotsc$

\vsp

\noindent\textbf{Branching property.} One important feature of Lévy trees is the branching property presented by \citet{Duquesne.LeGall-2005}. For any $a\in\ivoo{0,\infty}$, define $\Gi_a$ the $\sigma$-field of $\tr(a)$ and $\Ni_a$ the following point measure
\begin{align}
  \Ni_a(\dt\sigma'\dt\Ti') = \sum_{\sigma\in\Ti(a,0)} \delta_{(\sigma,\Ti_\sigma)}.
\end{align}
Then, the branching property states that under $\Nb_a$ and given $\Gi_a$, $\Ni_a$ is a Poisson point process on $\Ti(a)\times\Tbb$ with intensity $\ell^a(\dt\sigma')\Nb(\dt\Ti')$. Note that \citet{Weill-2007} has conversely proved that the branching properly entirely characterised the law of Lévy trees.\vsp

\noindent\textbf{Re-rooting invariance.} Finally, we will also make use in this work of another key property of Lévy trees: the re-rooting invariance. The latter has been first observed on the CRT by \citet{Aldous-1991a}, and then extended to stable (and Lévy) trees in the works of \citet[Th. 11]{Haas.Pitman.ea-2009} and \citet[Prop. 4]{Duquesne.LeGall-2005}. In particular, the former proves that stable trees are the only class of continuous fragmentation trees satisfying this property.

Namely, if $\Ti^{[\sigma]}$ designates the tree re-rooted in the vertex $\sigma$, then the re-rooting invariance of Lévy trees states that
\begin{align}
  \text{the law of }\Ti^{[\sigma]} \text{ under the measure }\tfrac{\mb(\dt\sigma)}{\mb(\Ti)}\Nb(\dt\Ti) \text{ coincides with }\Nb(\dt\Ti).
\end{align}
In other words, the law of Lévy trees is invariant under uniform re-rooting. Note that the previous property is a direct consequence of the spinal decomposition of Lévy trees and that a stronger extension, not required for our application, is also presented in \citet[Th. 2.2]{Duquesne.LeGall-2009}.


\section{Preliminary technical lemmas}  \label{sec:preliminaries}

Before addressing the proof of Theorems~\ref{th:levy_tree_upper_spectrum1}, \ref{th:levy_tree_upper_spectrum2} and \ref{th:levy_tree_upper_spectrum3}, we will recall and extend a few technical results on the left and right tails of the distribution of stable-CSBPs, and the local time and mass measure of stable trees.

Let us begin with the left and right tails of the local time under $\Nb_1$, which are sufficient due to the self-similarity of stable trees.
\begin{lemma}  \label{lemma:local_time_tail}
  Suppose $\gamma\in\ivoo{1,2}$. The tails of local time $\langle\ell^1\rangle$ under $\Nb_1$ satisfy
  \begin{align}
    \Nb_1\pthb{ \langle\ell^1\rangle \leq x } \sim_{0+} \frac{x^{\gamma-1}}{(\gamma-1)^2\Gamma(\gamma)} \quad\text{ and }\quad \Nb_1\pthb{ \langle\ell^1\rangle \geq x } \sim_{+\infty} -\frac{x^{-\gamma}}{v(1)\Gamma(1-\gamma)},
  \end{align}
  observing that $\Gamma(1-\gamma)<0$ and recalling $v(1)=(\gamma-1)^{-\tfrac{1}{\gamma-1}}$.
\end{lemma}
\begin{proof}
  The first estimate is due to \citet[Lemma 2.5]{Duquesne.Reichmann.ea-2010}.
  To prove the second one, recall the Laplace transform of $\bk{\ell^1}$ is given by Equation~\eqref{eq:laplace_tr_local_time}. In order to apply a Tauberian theorem, one needs to obtain an asymptotic expansion as $\mu$ tends to zero. Namely,
  \begin{align*}
    \Nb_1\pthb{e^{-\mu\langle\ell^1\rangle}}
    &= 1 - (\gamma-1)^{1/(\gamma-1)}\mu \cdot \pthb{ 1 + (\gamma-1)\mu^{\gamma-1} }^{-\tfrac{1}{\gamma-1}} \\
    &= 1 - (\gamma-1)^{1/(\gamma-1)}\mu + (\gamma-1)^{1/(\gamma-1)}\mu^{\gamma} + \littleo(\mu^{\gamma}),
  \end{align*}
  as $\mu\rightarrow 0$. A Tauberian theorem presented in \cite[Th. 1.7.1]{Bingham.Goldie.ea-1987} then entails the desired result, recalling that $v(1) = (\gamma-1)^{-1/(\gamma-1)}$.
\end{proof}
Note that in the quadratic case $\gamma=2$, $\bk{\ell^1}$ has an exponential distribution under $\Nb_1$, inducing that the left tail still holds and the right one is exponential (see \cite{Duhalde-2014}).
Similarly, we also establish a right tail bounds on stable CSBPs in the following two lemmas.
\begin{lemma}  \label{lemma:csbp_lower_tail}
  Suppose $\gamma\in\ivof{1,2}$ and $X$ is a stable CSBP. There exists a positive constant $c_{\gamma,0}$ such that for all positive $x,\rho,\delta$ satisfying $\rho\leq c_{\gamma,0}x$,
  \[
    \prb[_x]{X_\delta \leq \rho } \leq \exp\pthb{ -c_{\gamma,0}\,x\delta^{-1/(\gamma-1)} }.
  \]
\end{lemma}
\begin{proof}
  Recall that the Laplace transform of a CSBP is given by $\esp[_x]{\e^{-\mu X_\delta}} = \exp\pthb{-x\, u_\delta(\mu)}$ where in the stable case, $u_\delta(\mu) = \pthb{(\gamma-1)\delta + \mu^{-(\gamma-1)}}^{-1/(\gamma-1)}$. The Markov inequality on the exponential moment yields
  \begin{align*}
    \prb[_x]{X_\delta \leq \rho }
    =  \prb[_x]{\exp(-\mu X_\delta) \geq \exp(- \mu\rho) }
    \leq \exp\pthb{-x\, u_\delta(\mu) + \mu\rho }.
  \end{align*}
  Setting $\mu = \delta^{-1/(\gamma-1)}$, $u_\delta(\mu) = (\gamma\delta)^{-1/(\gamma-1)}$ and therefore
  \begin{align*}
    \prb[_x]{X_\delta \leq \rho }
    \leq \exp\pthb{-\delta^{-1/(\gamma-1)}\pth{ x\gamma^{-1/(\gamma-1)} - \rho} }.
  \end{align*}
  By choosing $c_{\gamma,0} = \gamma^{-1/(\gamma-1)} / 2$, we obtain the expected inequality.
\end{proof}

Using the previous lemmas, we may present some tail estimates on the supremum and infimum of CSBPs.
\begin{lemma}  \label{lemma:csbp_sup_upper_tail}
  Suppose $\gamma\in\ivoo{1,2}$, $\kappa\geq 1$ and $X$ is a stable CSBP. Then, there exist two positive constants $c_1$ and $c_2$ only depending on $\kappa$ and $\gamma$ such that for all positive $x,\rho,\delta$
  \[
    \prB[_x]{ \sup_{\ivff{0,\kappa\delta}} X_u \geq \rho } \leq c_1 \,\prb[_x]{ X_{\kappa\delta} \geq c_{\gamma,0}\,\rho } \pthB{1 - \exp\pth{-c_2\,\rho\delta^{-1/(\gamma-1)} } }^{-1}.
  \]
  In particular, for all positive $x,\rho,\delta$ satisfying $\delta^{1/(\gamma-1)}\leq\rho$,
  \[
    \prB[_x]{ \sup_{\ivff{0,\kappa\delta}} X_u \geq \rho } \leq c_3 \,\prb[_x]{ X_{\kappa\delta} \geq c_{\gamma,0}\,\rho }.
  \]
  where the constant $c_3$ only depends on $\gamma$ and $\kappa$.
\end{lemma}
\begin{proof}
  The proof is rather classic in the Markov processes literature (see for instance \citet{Duquesne.Labbe-2014} for a similar property). Briefly, let $T$ be the stopping time $T=\inf\brc{u:X_u \geq \rho}$ and $\lambda>0$. Then,
  \begin{align*}
    \prB[_x]{ \sup_{\ivff{0,\kappa\delta}} X_u \geq \rho }
    &= \prb[_x]{ T\leq\kappa\delta }
    \leq \prb[_x]{ X_{\kappa\delta} \geq \lambda\rho } + \prb[_x]{ T\leq\kappa\delta , X_{\kappa\delta} < \lambda\rho }.
  \end{align*}
  Owing to the strong Markov property, the latter term is equal to $\esp[_x]{ \indi_{\brc{T\leq\kappa\delta}} \,p_{\kappa\delta-T}(X_T, \ivfo{0,\lambda\rho}) }$. Then, choosing $\lambda= c_{\gamma,0}^{-1}$, Lemma~\ref{lemma:csbp_lower_tail} entails
  \[
    p_{\kappa\delta-T}(X_T, \ivfo{0,\lambda\rho}) \leq \exp\pthb{ -c_{\gamma,0} X_T \,(\kappa\delta)^{-1/(\gamma-1)} },
  \]
  on the event $\brc{T\leq\kappa\delta}$. Therefore,
  \begin{align*}
    \espb[_x]{ \indi_{\brc{T\leq\kappa\delta}} \,p_{\kappa\delta-T}(X_T, \ivfo{0,\lambda\rho}) }
    \leq \prb[_x]{ T\leq\kappa\delta } \exp\pthb{ - c_{\gamma,0} \rho (\kappa\delta)^{-1/(\gamma-1)} }.
  \end{align*}
  Combining the last inequality and the first bound on $\pr[_x]{ T\leq \kappa\delta }$ yields the first part of the lemma. The second one is straightforward as $\exp\pthb{ - c_{\gamma,0} \rho (\kappa\delta)^{-1/(\gamma-1)} } < 1$ when $\rho \,\delta^{-1/(\gamma-1)}\geq 1$.
\end{proof}

The next lemma extends the bound presented by \citet{Duhalde-2014} on the tail of the infimum of a CSBP.
\begin{lemma}  \label{lemma:csbp_inf_lower_tail}
  Suppose $\gamma\in\ivof{1,2}$ and $X$ is stable CSBP. Then, for all positive $x,y,\delta$ satisfying $y\leq x$,
  \[
    \prB[_x]{ \inf_{u\in\ivff{0,\delta}} X_u \leq y } \leq \exp\brcB{ -v(\delta)\pthb{ x^{1-1/\gamma} + y^{1-1/\gamma} } \pthb{ x^{1-1/\gamma} - y^{1-1/\gamma} }^{1/(\gamma-1)} },
  \]
  where we recall that $v(\delta) = \pthb{ (\gamma-1)\delta }^{-1/(\gamma-1)}$.
\end{lemma}
\begin{proof}
  As pointed out by \citet[Prop. 4.1]{Bingham-1976}, under $\Pr_x$, the process
  \[
    \forall u\in\ivff{0,\delta};\quad M_u = \exp\pthB{- X_u\pthb{ \lambda^{-(\gamma-1)} - (\gamma-1)u }^{-1/(\gamma-1)} },
  \]
  is martingale, under the condition $\lambda\in\ivoo{0,v(\delta)}$. We may then observe that
  \[
    \brcB{ \inf_{u\in\ivff{0,\delta}} X_u \leq y } \subseteq \brcB{ \sup_{u\in\ivff{0,\delta}} M_u \geq \exp\pthB{- y\pthb{ \lambda^{-(\gamma-1)} - (\gamma-1)\delta }^{-1/(\gamma-1)} } }.
  \]
  As $\esp[_x]{M_u} = \exp(-\lambda x)$, the celebrated maximal inequality for submartingales entails
  \begin{align*}
    \prB[_x]{ \inf_{u\in\ivff{0,\delta}} X_u \leq y }
    \leq \exp\pthB{y\pthb{ \lambda^{-(\gamma-1)} - (\gamma-1)\delta }^{-1/(\gamma-1)} - x\lambda}.
  \end{align*}
  To optimize the bound on the variable $\lambda$, we define the function
  \[
    g(\mu) = y\pthb{\mu-(\gamma-1)\delta}^{-1/(\gamma-1)} - x\mu^{-1/(\gamma-1)}.
  \]
  Since $(\gamma-1)g'(\mu) = x\mu^{-\gamma/(\gamma-1)} - y\pth{\mu-(\gamma-1)\delta}^{-\gamma/(\gamma-1)}$, the minimum is attained for $\mu_0 = (\gamma-1)\delta \, y^{-(\gamma-1)/\gamma}\pth{y^{-(\gamma-1)/\gamma } - x^{-(\gamma-1)/\gamma} }^{-1}$. Elementary computations then show that
  \[
    g(\mu_0) = -v(\delta)\pthb{ x^{1-1/\gamma} + y^{1-1/\gamma} } \pthb{ x^{1-1/\gamma} - y^{1-1/\gamma} }^{1/(\gamma-1)},
  \]
  hence proving the lemma.
\end{proof}

Finally, let us describe the tail behaviours of the supremum and infimum of the local time on intervals of the form $\ivff{\delta\kappa,\delta/\kappa}$, for some fixed $\kappa\in\ivoo{0,1}$.
\begin{lemma}  \label{lemma:local_time_sup_tails}
  Suppose $\gamma\in\ivoo{1,2}$ and $\kappa\in\ivoo{0,1}$. There exist two positive constants $c_0$ and $c_1$ such that for any $h\in\ivffb{0,\tfrac{1}{\gamma-1}}$ and all $\delta>0$,
  \begin{align*}
    c_0\,\delta^{\gamma/(\gamma-1)-\gamma h} \Lambda(\delta)^{-\gamma} \leq \Nb_{\delta\kappa}\pthB{ \sup_{\ivff{\delta\kappa,\delta/\kappa}} \bk{\ell^u} \geq \Lambda(\delta)\,\delta^h  } \leq c_1\,\delta^{\gamma/(\gamma-1)-\gamma h}  \Lambda(\delta)^{-\gamma},
  \end{align*}
  where the function $\Lambda(\cdot)$ satisfies $\Lambda(\cdot)\geq 1$.
\end{lemma}
\begin{proof}
  Owing to the Ray--Knight theorem presented by \citet[Th. 1.4.1]{Duquesne.LeGall-2002} and Lemma~\ref{lemma:csbp_sup_upper_tail}
  \begin{align*}
    \Nb_{\delta\kappa}\pthB{ \sup_{\ivff{\delta\kappa,\delta/\kappa}} \bk{\ell^u} \geq \Lambda(\delta)\,\delta^h }
    &= \Nb_{\delta\kappa}\pthB{ \prB[_{\bk{\ell^{\delta\kappa}}}]{ \sup_{\ivff{0,\delta(\kappa)}} X_u \geq \Lambda(\delta)\,\delta^h } } \\
    &\leq c_1 \, \Nb_{\delta\kappa}\pthB{ \prb[_{\bk{\ell^{\delta\kappa}}}]{ X_{\delta(\kappa)} \geq c_0\,\Lambda(\delta)\,\delta^h } } \\
    &= c_1 \, \Nb_{\delta\kappa}\pthb{ \bk{\ell^{\delta/\kappa}} \geq c_0\,\Lambda(\delta)\,\delta^h }
    = c_2 \, \Nb_{\delta/\kappa}\pthb{ \bk{\ell^{\delta/\kappa}} \geq c_0\,\Lambda(\delta)\,\delta^h },
  \end{align*}
  where $X$ designates a stable CSBP and $\delta(\kappa)=\delta(1/\kappa-\kappa)$. Then, Lemma~\ref{lemma:local_time_tail} and the self-similarity of stable trees entail
  \begin{align*}
    \Nb_{\delta/\kappa}\pthb{ \bk{\ell^{\delta/\kappa}} \geq c_0\,\Lambda(\delta)\,\delta^h }
    &=  \Nb_1 \pthb{ (\delta/\kappa)^{1/(\gamma-1)} \bk{\ell^1} > c_0\,\Lambda(\delta)\,\delta^h } \\
    &\sim_{\delta\rightarrow 0} \, c_2 \, \delta^{\gamma/(\gamma-1)-\gamma h}  \Lambda(\delta)^{-\gamma}.
  \end{align*}
  The lower bound is a consequence of the same Lemma~\ref{lemma:local_time_tail} and the simple observation
  \[
    v(\delta) / v(\delta\kappa) \,\Nb_{\delta}\pthb{ \bk{\ell^{\delta}} \geq \Lambda(\delta)\,\delta^h  }  \leq \Nb_{\delta\kappa}\pthB{ \sup_{\ivff{\delta\kappa,\delta/\kappa}} \bk{\ell^u} \geq \Lambda(\delta)\,\delta^h }.
  \]
\end{proof}

\begin{lemma}  \label{lemma:local_time_inf_tail0}
  Suppose $\gamma\in\ivoo{1,2}$ and $\kappa\in\ivoo{0,1}$. There exist two positive constants $c_0$ and $c_1$ such that for any $h\in\ivffb{0,\tfrac{1}{\gamma-1}}$ and all $\delta>0$,
  \begin{align*}
    c_0\, \delta^{\gamma/(\gamma-1)-\gamma h} \Lambda(\delta)^{-\gamma} \leq \Nb_{\delta\kappa}\pthB{ \inf_{\ivff{\delta\kappa,\delta/\kappa}} \bk{\ell^u} \geq \Lambda(\delta)\,\delta^h } \leq c_1\, \delta^{\gamma/(\gamma-1)-\gamma h} \Lambda(\delta)^{-\gamma},
  \end{align*}
  where the function $\Lambda(\cdot)$ satisfies $\Lambda(\cdot)\geq 1$.
\end{lemma}
\begin{proof}
  The upper bound is a straightforward consequence of Lemma~\ref{lemma:local_time_tail} and the simple inequality
  \[
    \Nb_{\delta\kappa}\pthB{ \inf_{\ivff{\delta\kappa,\delta/\kappa}} \bk{\ell^u} \geq \Lambda(\delta)\,\delta^h } \leq v(\delta) / v(\delta\kappa) \,\Nb_\delta\pthb{ \bk{\ell^\delta} \geq \Lambda(\delta)\,\delta^h }.
  \]
  Then, to prove the lower bound, let us observe similarly to the proof of Lemma~\ref{lemma:local_time_sup_tails} that
  \begin{align*}
    \Nb_{\delta\kappa}\pthB{ \inf_{\ivff{\delta\kappa,\delta/\kappa}} \bk{\ell^u} \geq \Lambda(\delta)\,\delta^h }
    &= \Nb_{\delta\kappa}\pthB{ \prB[_{\bk{\ell^{\delta\kappa}}}]{ \inf_{\ivff{0,\delta(\kappa)}} X_u \geq \Lambda(\delta)\,\delta^h } } \\
    &\geq \Nb_{\delta\kappa}\pthB{ \prB[_{2\Lambda(\delta)\delta^h}]{ \inf_{\ivff{0,\delta(\kappa)}} X_u \geq \Lambda(\delta)\,\delta^h } \indi_{ \bk{\ell^{\delta\kappa}}\geq 2\Lambda(\delta)\delta^h} } \\
    &= \Nb_{\delta\kappa}\pthb{ \bk{\ell^{\delta\kappa}}\geq 2\Lambda(\delta)\delta^h } \prB[_{2\Lambda(\delta)\delta^h}]{ \inf_{\ivff{0,\delta(\kappa)}} X_u \geq \Lambda(\delta)\delta^h },
  \end{align*}
  Still using Lemma~\ref{lemma:local_time_tail}, we obtain a proper lower bound of the first term.
  Finally, Lemma~\ref{lemma:csbp_inf_lower_tail} leads to an estimate of the second one
  \begin{align*}
    \prB[_{2\Lambda(\delta)\delta^h}]{ \inf_{\ivff{0,\delta(\kappa)}} X_u \leq \Lambda(\delta)\delta^h } \leq \exp\pthb{ -c_2 \Lambda(\delta)\delta^{h-1/(\gamma-1)} }.
  \end{align*}
  for a constant $c_2$ independent of $\delta$ and $h$. The latter term is strictly smaller than $1$ independently of $\delta$, therefore concluding the proof.
\end{proof}

We also present in the next lemma a bound on the left tail of the infimum of the local time.
\begin{lemma}  \label{lemma:local_time_inf_tail1}
  Suppose $\gamma\in\ivof{1,2}$ and $\kappa\in\ivoo{0,1}$. There exist three positive constants $c_0$, $c_1$ and $c_2$ such that for all $\delta,\tau>0$ and any $v\in\ivff{\delta\kappa,\delta/\kappa}$
  \begin{align*}
    c_0\, \Lambda(\delta)^{\gamma-1} \leq \Nb_{v}\pthB{ \inf_{\ivff{v,v+\tau}} \bk{\ell^u} \leq \Lambda(\delta)\,\delta^{\tfrac{1}{\gamma-1}} } \leq c_1\, \Lambda(\delta)^{\gamma-1} + \exp\pthB{-c_2\,v(\tau) \Lambda(\delta)\,\delta^{\tfrac{1}{\gamma-1}} },
  \end{align*}
  where the function $\Lambda(\cdot)$ satisfies $\Lambda(\cdot)\leq 1$.
\end{lemma}
\begin{proof}
  The lower bound is a straightforward consequence of the inclusion
  \[
    \brcB{ \bk{\ell^{v}} \leq \Lambda(\delta)\,\delta^{\tfrac{1}{\gamma-1}} } \subset \brcB{ \inf_{\ivff{v,v+\tau}} \bk{\ell^u} \leq \Lambda(\delta)\,\delta^{\tfrac{1}{\gamma-1}} }.
  \]
  We proceed similarly to the proof of Lemma~\ref{lemma:local_time_inf_tail0} to obtain the upper bound.
  \begin{align*}
    \Nb_{v}\pthB{ \inf_{\ivff{v,v+\tau}} \bk{\ell^u} \leq \Lambda(\delta)\,\delta^{\tfrac{1}{\gamma-1}} }
    &= \Nb_{v}\pthB{ \prB[_{\bk{\ell^{v}}}]{ \inf_{\ivff{0,\tau}} X_u \leq \Lambda(\delta)\,\delta^{\tfrac{1}{\gamma-1}} } } \\
    &\leq \Nb_{v}\pthB{ \bk{\ell^{v}} \leq 2\Lambda(\delta)\,\delta^{\tfrac{1}{\gamma-1}} } \\
    &+ \Nb_{v}\pthB{ \prB[_{2\Lambda(\delta)\delta^{1/(\gamma-1)}}]{ \inf_{\ivff{0,\tau}} X_u \leq \Lambda(\delta)\,\delta^{\tfrac{1}{\gamma-1}} } }.
  \end{align*}
  The first term is upper-bounded using Lemma~\ref{lemma:local_time_tail}. For the second one, Lemma~\ref{lemma:csbp_inf_lower_tail} entails
  \begin{align*}
    \prB[_{2\Lambda(\delta)\delta^{1/(\gamma-1)}}]{ \inf_{\ivff{0,\tau}} X_u \leq \Lambda(\delta)\,\delta^{\tfrac{1}{\gamma-1}} }  \leq \exp\pthB{-c_0\,v(\tau) \Lambda(\delta)\,\delta^{\tfrac{1}{\gamma-1}} },
  \end{align*}
  therefore concluding the proof.
\end{proof}

As a last lemma on the tail behaviour of the local time, we present an estimate on the tail of the infimum following an hitting time of the local time.
\begin{lemma}  \label{lemma:local_time_sup_inf_tail}
  Suppose $\gamma\in\ivoo{1,2}$ and $\Lambda(\cdot)$ is a positive function. For any $\delta>0$, we define the stopping time
  \begin{align*}
    \Theta(\delta) \eqdef \inf\brcB{u\geq 0 : \bk{\ell^u} \geq 2\Lambda(\delta) }.
  \end{align*}
  Then, there exist three positive constants $c_0$, $c_1$ and $c_2$ such that for any $\delta,\tau>0$,
  \begin{align*}
    \Nb\pthB{ {\Theta(\delta) \leq \delta} \text{ and } \inf_{\Theta(\delta)+\ivff{0,\tau}} \bk{\ell^u} \leq \Lambda(\delta) }
    \leq c_0\, v(\delta) \pthB{1 - \exp\pthb{-c_1\,\Lambda(\delta) v(\delta) } }^{-1} \exp\pthB{ -c_2\,v(\tau) \Lambda(\delta) }.
  \end{align*}
\end{lemma}
\begin{proof}
  The proof is clearly inspired by the previous lemmas. To begin with, we remark that using the monotone convergence theorem:
  \begin{align*}
    \Nb\pthB{ \Theta(\delta) \leq \delta \cap \inf_{\Theta(\delta)+\ivff{0,\tau}} \bk{\ell^u} \leq \Lambda(\delta) } = \lim_{a\rightarrow 0} \Nb\pthB{ {\Theta(\delta) \in\ivof{a,\delta}} \cap \inf_{\Theta(\delta)+\ivff{0,\tau}} \bk{\ell^u} \leq \Lambda(\delta) }.
  \end{align*}
  Then, according to the Ray--Knight theorem, the right hand term $J(a)$ is equal to
  \begin{align*}
    J(a)
    &= v(a)\Nb_a\pthB{ {\Theta(\delta) \in\ivof{a,\delta}} \cap \inf_{\Theta(\delta)+\ivff{0,\tau}} \bk{\ell^u} \leq \Lambda(\delta) } \\
    &= v(a)\Nb_a\pthB{ \indi_{\brc{\Theta(\delta)>a}} \prB[_{\bk{\ell^a}}]{ \Theta_X(\delta)\leq \delta-a \cap \inf_{\Theta_X(\delta)+\ivff{0,\tau}} X_u \leq \Lambda(\delta) } },
  \end{align*}
  where under $\Pr_x$, $X$ is stable CSBP starting from $x$ and $\Theta_X(\delta)\eqdef\inf\brcb{u\geq 0 : X_u\geq 2\Lambda(\delta) }$. Applying the strong Markov property, we get for any $x\geq 0$
  \begin{align*}
    \prB[_x]{ \Theta_X(\delta)\leq \delta-a \cap \inf_{\Theta_X(\delta)+\ivff{0,\tau}} X_u \leq \Lambda(\delta) }
    &\leq \prB[_x]{ \Theta_X(\delta)\leq \delta \cap \inf_{\ivff{0,\tau}} X_{\Theta_X(\delta)+u} \leq \Lambda(\delta) } \\
    &= \espB[_x]{ \indi_{\brc{\Theta_X(\delta)\leq \delta}} \prB[_{X_{\Theta_X(\delta)}}]{ \inf_{\ivff{0,\tau}} \widetilde X_u \leq \Lambda(\delta) } }.
  \end{align*}
  Using Lemma~\ref{lemma:csbp_inf_lower_tail} to bound the second term, we get
  \begin{align*}
    \prB[_{X_{\Theta_X(\delta)}}]{ \inf_{\ivff{0,\tau}} \widetilde X_u \leq \Lambda(\delta) } \leq \exp\pthb{ -c_0\,v(\tau) \Lambda(\delta) },
  \end{align*}
  where the constant $c_0$ only depends on $\gamma$. To study $\prb[_x]{ \Theta_X(\delta)\leq \delta }$, we observe that according to Lemma~\ref{lemma:csbp_sup_upper_tail},
  \begin{align*}
    \prb[_x]{ \Theta_X(\delta)\leq \delta } = \prB[_x]{ \sup_{\ivff{0,\delta}} X_u \geq 2\Lambda(\delta) } \leq c_2 \,\prb[_x]{X_\delta \geq c_3\,\Lambda(\delta)} \pthB{1 - \exp\pthb{-c_4\,\Lambda(\delta) v(\delta) } }^{-1}.
  \end{align*}
  Combining the two previous inequalities and the Ray--Knight theorem, we get
  \begin{align*}
    J(a)
    \leq c_5\, \Nb\pthB{ \Theta(\delta)>a \cap \bk{\ell^{\delta+a}} \geq c_3\,\Lambda(\delta) } \pthB{1 - \exp\pthb{-c_4\,\Lambda(\delta) v(\delta) } }^{-1} \exp\pthb{ -c_0\,v(\tau) \Lambda(\delta) }.
  \end{align*}
  Finally, the monotone convergence theorem and the càdlàguity of the local time entail
  \begin{align*}
    \Nb\pthB{ \Theta(\delta) \leq \delta \cap \inf_{\Theta(\delta)+\ivff{0,\tau}} \bk{\ell^u} \leq \Lambda(\delta) }
    &\leq c_6\, v(\delta) \pthB{1 - \exp\pthb{-c_4\,\Lambda(\delta) v(\delta) } }^{-1} \exp\pthb{ -c_0\,v(\tau) \Lambda(\delta) },
  \end{align*}
  where the previous positive constants only depends on $\gamma$.
\end{proof}

Finally, we present as well an estimate on the right tail of the mass measure. Similarly to the local time, the self-similarity ensures that it is sufficient to study $\mb(\rho,1+c)$ under $\Nb_1$, for some fixed constant $c\geq 0$. Recall first that \citet[Lemma 3.1]{Duquesne.Wang-2014} have proved that the left tail has an exponential equivalent. Namely, for any $c\geq 0$ and $\gamma\in\ivof{1,2}$
\begin{align*}
  -\log \Nb_1\pthb{ \mb\pth{B(\rho,1+c)} \leq x } \sim_{x\rightarrow 0+} \pthbb{\frac{\gamma-1}{x}}^{\gamma-1}.
\end{align*}
A less precise bound had also been presented previously in \cite[Lemma 4.1]{Duquesne.LeGall-2006}. In the following lemma, we present the right tail of $m(\rho,1+c)$.
\begin{lemma}  \label{lemma:mmass_tail}
  Suppose $\gamma\in\ivoo{1,2}$ and $c\geq 0$. Then,
  \begin{align*}
    \Nb_1\pthb{ \mb\pth{B(\rho,1+c)} \geq x } \sim_{x\rightarrow\infty} -\frac{(1+c)^{\gamma+1} \, x^{-\gamma}}{(\gamma+1)v(1)\Gamma(1-\gamma)}.
  \end{align*}
\end{lemma}
\begin{proof}
  The Laplace transform of $m(\rho,1+c)$ has been characterised by \citet[Lemma 2.1]{Duquesne.Wang-2014}:
  \begin{align*}
    \Nb_1\pthb{\e^{-\lambda \mb(\rho,1+c)}} = v(1)^{-1}\pthb{ \kappa_1(\lambda,\infty) - \kappa_{1+c}(\lambda,0) }.
  \end{align*}
  Owing to a Tauberian theorem (see \cite[Th. 8.1.6]{Bingham.Goldie.ea-1987}), it is sufficient to get a precise asymptotic expansion as $\lambda$ goes to zero.

  Let us begin with the first term $\kappa_1(\lambda,\infty)$ which solves, according to \eqref{eq:int_kappa}, $\int_{\kappa_1(\lambda,\infty)}^\infty \frac{\dt u}{u^\gamma-\lambda} = 1$.
  For any $\lambda>0$ sufficiently small, $\kappa_1(\lambda,\infty) \geq v(1) > \lambda^{1/\gamma}$. Hence, we may expand the previous expression into:
  \begin{align*}
    1 = \int_{\kappa_1(\lambda,\infty)}^\infty  u^{-\gamma} \bktbb{ \sum_{k=0}^\infty (\lambda u^{-\gamma})^{k} } \dt u.
  \end{align*}
  Consequently, noting that $\int_{\kappa_1(\lambda,\infty)}^\infty  u^{-\gamma} \dt u = (\gamma-1)^{-1} \kappa_1(\lambda,\infty)^{1-\gamma}$ and $v(1)=(\gamma-1)^{-1/(\gamma-1)}$,
  \begin{align*}
    \kappa_1(\lambda,\infty) = v(1)\brcbb{ 1 - \int_{\kappa_1(\lambda,\infty)}^\infty u^{-\gamma} \bktbb{ \sum_{k=1}^\infty (\lambda u^{-\gamma})^{k} } \dt u }^{-\tfrac{1}{\gamma-1}}.
  \end{align*}
  Observing that $ \int_{\kappa_1(\lambda,\infty)}^\infty u^{-\gamma} \bktB{ \sum_{k=1}^\infty (\lambda u^{-\gamma})^{k} } \dt u = \bigo(\lambda)$, the Taylor expansion of previous term yields
  \begin{align*}
    \kappa_1(\lambda,\infty)
    &= v(1)\brcbb{ 1 + \frac{1}{\gamma-1} \int_{\kappa_1(\lambda,\infty)}^\infty u^{-\gamma} \bktbb{ \sum_{k=1}^\infty (\lambda u^{-\gamma})^{k} } \dt u } + \bigo(\lambda^2) \\
    &= v(1)\brcbb{ 1 + \frac{\lambda}{\gamma-1} \int_{\kappa_1(\lambda,\infty)}^\infty u^{-2\gamma} \dt u } + \bigo(\lambda^2).
  \end{align*}
  Re-injecting the estimate $\kappa_1(\lambda,\infty) =  v(1)\pthb{ 1 + \bigo(\lambda) }$ in the previous expression, we finally obtain:
  \begin{align}  \label{eq:asymp_kappa_infty}
    \kappa_1(\lambda,\infty) = v(1) + \frac{\gamma-1}{2\gamma-1}\lambda + \bigo(\lambda^2).
  \end{align}
  Note that $\tfrac{\gamma-1}{2\gamma-1} < 1$.

  Let us now investigate the component $\kappa_{1+c}(\lambda,0)$. Similarly, it solves $\int_0^{\kappa_1(\lambda,0)} \frac{\dt u}{\lambda-u^\gamma} = 1+c$, and observing that $\lambda^{1/\gamma} > \kappa_1(\lambda,0)$, we get
  \begin{align*}
    1+c = \lambda^{-1} \int_0^{\kappa_1(\lambda,0)} \bktbb{ \sum_{k=0}^\infty (\lambda^{-1} u^{\gamma})^{k} } \dt u.
  \end{align*}
  Consequently,
  \begin{align*}
    \kappa_1(\lambda,0) = (1+c)\lambda - \sum_{k=1}^\infty \frac{1}{\gamma k + 1} \lambda^{-k}\kappa_1(\lambda,0)^{\gamma k + 1} = (1+c)\lambda + \littleo(\lambda),
  \end{align*}
  noting that $\lambda^{-k}\kappa_1(\lambda,0)^{\gamma k} \lessapprox \lambda^{(\gamma-1)k}$
  Then, re-injecting the second estimate into the first equality, we get
  \begin{align}  \label{eq:asymp_kappa_0}
    \kappa_1(\lambda,0) = (1+c)\lambda -  \frac{(1+c)^{\gamma+1}}{\gamma+1} \lambda^\gamma + \littleo(\lambda^\gamma).
  \end{align}
  Combining the asymptotic expansions \eqref{eq:asymp_kappa_infty} and \eqref{eq:asymp_kappa_0}, there exists a constant $C(c,\gamma)>0$ such that
  \begin{align*}
    \Nb_1\pthb{\e^{-\lambda \mb(\rho,1+c)}} = 1 - C(c,\gamma) \lambda + \frac{(1+c)^{\gamma+1}}{(\gamma+1)v(1)} \lambda^\gamma + \littleo(\lambda^\gamma).
  \end{align*}
  A Tauberian theorem presented by \citet[Th. 8.1.6]{Bingham.Goldie.ea-1987} then entails the desired result.
\end{proof}

To end this technical section, let us introduce a few notations that will be extensively used in the rest of the article:
\[
  \forall x\in\ivoo{0,1};\quad g(x) \eqdef \pthb{\log x^{-1} }^{-1} \quad\text{and}\quad h(x) \eqdef \pthb{\log\log x^{-1} }^{-1}.
\]
Note that $h(x) = g(g(x))$.
For any $\delta>0$, we will denote by $\Di(\delta)$ the following collection of subintervals: $\Di(\delta) \eqdef \brcb{ \ivff{k\delta,(k+1)\delta} : k\in\N}$. In addition, we will use the notation $\Di_n$ for dyadic intervals, i.e. $\Di_n\eqdef\Di(2^{-n})$. Finally, for any set $F$, we define $\Di_n(F)$ as $\Di_n(F)\eqdef\brc{ I\in\Di_n : I\cap F\neq \vset }$.


\section{Multifractal spectrum of stable trees}  \label{sec:stable_trees_spectrum}

This section is devoted to the proof of Proposition~\ref{prop:scaling_exponents} and Theorems~\ref{th:levy_tree_upper_spectrum1}, \ref{th:levy_tree_upper_spectrum2} and \ref{th:levy_tree_upper_spectrum3}. As typically in the fractal literature, we divide the proofs of fractal dimensions into two parts corresponding to the upper and lower bounds.

We recall a classic Chernoff bound that will be extensively used in the rest of the article: if X is a Poisson distribution parametrised by $\lambda$:
\begin{align}  \label{eq:chernoff_poisson}
  \pr{ X \leq y } \leq \e^{-\lambda} (\e\lambda)^y y^{-y} \quad\text{and}\quad \pr{ X \geq x } \leq \e^{-\lambda} (\e\lambda)^x x^{-x},
\end{align}
where $y\leq \lambda\leq x$.

\subsection{Mass measure and local time scaling exponents}  \label{ssec:stable_trees_exponents}

In this section, we aim to investigate the connections between the different scaling exponents presented in the introduction. For the sake of readability, we divide the proof of Proposition~\ref{prop:scaling_exponents} into two parts. In order to the relation~\eqref{eq:ltime_mass_exponents}, we study in the following lemma a tail behaviour on the joint law of the supremum of the local time and the mass measure.
\begin{lemma}  \label{lemma:scaling_exponents0}
  Suppose $\gamma\in\ivoo{0,2}$ and $\eps>0$. Then, there exists $\delta_0>0$ such for any $\delta\in\ivoo{0,\delta_0}$ and $h\in\ivffb{0,\tfrac{1}{\gamma-1}}$,
  \begin{align}
    \Nb\pthB{ \sup_{\ivff{\delta,2\delta}} \bk{\ell^u} \geq \delta^{h+\eps} \text{ and } \mb(B(\rho,3\delta)) < \delta^{h+1+3\eps} } \leq c_0\, \delta^{-\tfrac{1}{\gamma-1}-\eps} \exp\pthb{ -c_1\,\delta^{-\eps} },
  \end{align}
  where the constants $c_0$ and $c_1$ are independent of $\delta$ and $h$.
\end{lemma}
\begin{proof}
  The proof mainly relies on the estimate presented in Lemma~\ref{lemma:local_time_sup_inf_tail}. Following the notations introduced in the latter, we set $\Lambda(2\delta) = \delta^{h+\eps} / 2$ and $\tau = \delta^{(h+2\eps)(\gamma-1)} \geq \delta^{1+2\epsilon(\gamma-1)}$. To begin with, we observe that
  \begin{align*}
    &\Nb\pthB{ \sup_{\ivff{\delta,2\delta}} \bk{\ell^u} \geq \delta^{h+\eps} \cap \mb(B(\rho,3\delta)) < \delta^{h+1+3\eps} } \\
    &\leq \Nb\pthB{ {\Theta(2\delta) \leq 2\delta} \cap \inf_{\Theta(2\delta)+\ivff{0,\tau}} \bk{\ell^u} \leq \Lambda(2\delta) } \\
    &+\Nb\pthB{ {\Theta(2\delta) \leq 2\delta} \cap \inf_{\Theta(2\delta)+\ivff{0,\tau}} \bk{\ell^u} > \Lambda(2\delta) \cap \mb(B(\rho,3\delta)) < \delta^{h+1+3\eps} }.
  \end{align*}
  One can easily note that the second term is equal to zero. Indeed, assuming that $\Theta(2\delta) \leq 2\delta$ and $\inf_{\Theta(2\delta)+\ivff{0,\tau}} \bk{\ell^u} > \Lambda(2\delta)$, we get for any $\delta$ sufficiently small
  \begin{align*}
    \mb(B(\rho,3\delta)) = \int_{\ivfo{0,3\delta}} \dt u \bk{\ell^u} \geq \delta^{1+2\epsilon(\gamma-1)} \cdot \delta^{h+\eps} / 2 \geq \delta^{1+h+3\eps}.
  \end{align*}
  On the other hand, we deduce from Lemma~\ref{lemma:local_time_sup_inf_tail} a bound on the first term:
  \begin{align*}
    \Nb\pthB{ {\Theta(2\delta) \leq 2\delta} \cap \inf_{\Theta(2\delta)+\ivff{0,\tau}} \bk{\ell^u} \leq \Lambda(2\delta) }
    &\leq c_0\, v(\delta) \pthB{1 - \exp\pthB{-c_1\,\delta^{h+\eps-\tfrac{1}{\gamma-1}} } }^{-1} \exp\pthb{ -c_2\,\delta^{-\eps} } \\
    &\leq c_3\, \delta^{-\tfrac{1}{\gamma-1}-\eps} \exp\pthb{ -c_2\,\delta^{-\eps} },
  \end{align*}
  where the previous constants are independent of $\delta$ and $h$.
\end{proof}

We may now prove the first part of Proposition~\ref{prop:scaling_exponents}.
\begin{proof}[Proof of Proposition~\ref{prop:scaling_exponents} - Part one]
  Set $h\leq\tfrac{1}{\gamma-1}$, $\eps>0$ and define for any $k,n\in\N$ the r.v.
  \[
    Z(k,n,h)= \#\brcB{ \Ti_\sigma\in\Tbb(k\delta_n,\delta_n) : \sup_{\ivff{\delta_n,2\delta_n}} \bk{\ell^u}(\Ti_\sigma) \geq \delta_n^{h+\eps} \text{ and } \mb(\Ti_\sigma\cap B(\sigma,3\delta_n)) < \delta_n^{h+1+3\eps} }.
  \]
  The branching property of stable trees entails that under $\Nb_{k\delta_n}$ and given $\Gi_{k\delta_n}$, $Z(k,n,h)$ is a Poisson random variable. Moreover, according to Lemma~\ref{lemma:scaling_exponents0}, its parameter $\lambda_{k,n}$ is bounded by:
  \begin{align*}
    \lambda_{k,n} \leq c_0\, \bk{\ell^{k\delta_n}} \delta_n^{-\tfrac{1}{\gamma-1}-\eps} \exp\pthb{ -c_1\,\delta_n^{-\eps} }.
  \end{align*}
  As a consequence,
  \begin{align*}
    \Nb_{k\delta_n}\pthcb{Z(k,n,h)\geq 1}{\Gi_{k\delta_n}} \leq c_0\, \bk{\ell^{k\delta_n}} \delta_n^{-\tfrac{1}{\gamma-1}-\eps} \exp\pthb{ -c_1\,\delta_n^{-\eps} }.
  \end{align*}
  Recalling that $\Nb_{k\delta_n}\pthb{\bk{\ell^{k\delta_n}}} = v(k\delta_n)^{-1}$, we get $\Nb\pthb{Z(k,n,h)\geq 1} \leq c_0\, \delta_n^{-\tfrac{1}{\gamma-1}-\eps} \exp\pthb{ -c_1\,\delta_n^{-\eps} }$.
  Therefore, for any fixed level $b>0$
  \begin{align*}
    \sum_{n\in\N} \sum_{k\delta_n\in\ivoo{0,b}} \Nb_{k\delta_n}\pthb{Z(k,n,h)\geq 1} \leq  c_0\,b\, \sum_{n\in\N} \delta_n^{-\tfrac{1}{\gamma-1}-1-\eps} \exp\pthb{ -c_1\,\delta_n^{-\eps} } < \infty,
  \end{align*}
  and Borel--Cantelli lemma entails: $\Nb$-a.e. for any $n$ sufficiently large,
  \begin{align*}
    \forall k\delta_n\in\ivoo{0,b},\ \forall \Ti_\sigma\in\Tbb(k\delta_n,\delta_n);\quad \sup_{\ivff{\delta_n,2\delta_n}} \bk{\ell^u}(\Ti_\sigma) \geq \delta_n^{h+\eps} \Longrightarrow \mb(\Ti_\sigma\cap B(\sigma,3\delta_n)) \geq \delta_n^{h+1+3\eps}.
  \end{align*}
  We may now prove the first statement in Proposition~\ref{prop:scaling_exponents}. Suppose $\sigma_0\in\Ti$ such that $\alpha_\ell(\sigma_0,\Ti) \leq h$, with $h\leq\tfrac{1}{\gamma-1}$. In addition, we set $a=d(\rho,\sigma_0)$ and suppose that $a\in\ivoo{0,b}$. Then, the definition of the local time scaling exponent induce that for an infinite number of scales $\delta_n$, there exist $k\geq 1$ and $\Ti_\sigma\in\Tbb(k\delta_n,\delta_n)$ such that $a\in\ivff{\delta_n,2\delta_n}$ and
  \begin{align*}
    \sup_{\ivff{\delta_n,2\delta_n}} \bk{\ell^u}(\Ti_\sigma) \geq \ell^a\pthb{B(\sigma_0,\delta_n)} \geq \delta_n^{h+\eps}.
  \end{align*}
  Owing to the previous lemma, we get
  \begin{align*}
    \mb\pthb{B(\sigma_0,6\delta_n)} \geq \mb\pthb{\Ti_\sigma\cap B(\sigma,3\delta_n)} \geq \delta_n^{h+1+3\eps},
  \end{align*}
  proving that $\alpha_\mb(\sigma_0,\Ti) \leq 1+h+3\eps$. Taking the limit $\eps\rightarrow 0$ and a countable and dense collection of $h\in\ivffb{0,\tfrac{1}{\gamma-1}}$, we obtain the desired property.
\end{proof}

The proof of the second part in Proposition~\ref{prop:scaling_exponents} is quite similar, even though slightly more technical. As previously, we start by presenting a lemma investigating the tail behaviour on the join law of the mass measure and the local branching index. For the purpose of the proof, we use a definition of the branching index which generalises the form \eqref{eq:branching_exponent0} presented in the introduction: for any $\sigma\in\Ti$ and $\delta>0$,
\begin{align}  \label{eq:branching_exponent_ext}
  n_{b,\kappa}(\sigma,\delta) \eqdef \#\brcb{ \text{connected components diameter $>\kappa\delta$ in } \Ti\setminus \overline{B}(\sigma,\delta) }.
\end{align}
where $\kappa\geq 1$ is a fixed parameter.
\begin{lemma}  \label{lemma:scaling_exponents1}
  Suppose $\gamma\in\ivoo{0,2}$, $\kappa\geq 1$ and $\eps>0$. Then, for every $h\in \ivoob{-\infty,\tfrac{\gamma}{\gamma-1}-2\eps}$ and all $\delta$ sufficiently small
  \begin{align}  \label{eq:mmass_branching_tail0}
    \Nb\pthB{ \mb(B(\rho,\delta)) \geq \delta^{h+\eps} \text{ and } n_{b,\kappa}(\rho,\delta) \leq \delta^{h+2\eps-\tfrac{\gamma}{\gamma-1}} } \leq c_0\,v(\delta) \exp\pthb{-c_1\,\delta^{-\eps}}
  \end{align}
  and
  \begin{align}  \label{eq:mmass_branching_tail1}
    \Nb\pthB{ n_{b,\kappa}(\rho,\delta) \geq \delta^{h+\eps-\tfrac{\gamma}{\gamma-1}} \text{ and } \mb(B(\rho,2\delta)) \leq \delta^{h+2\eps} } \leq c_0\,v(\delta) \exp\pthb{-c_1\,\delta^{-\eps}}
  \end{align}
  where the constants $c_0$ and $c_1$ are independent of $h$ and $\delta$.
\end{lemma}
\begin{proof}
  Let us set $\eps>0$ and $h\in\ivoob{-\infty,\tfrac{\gamma}{\gamma-1}-2\eps}$, and investigate the first inequality.  As we aim to apply as well Lemma~\ref{lemma:local_time_sup_inf_tail}, we also set, following the notations in the latter, $\Lambda(\delta) \eqdef \delta^{h-1+\eps}/2$ and $\tau \eqdef \delta$. Then, we may simply observe that $\mb(B(\rho,\delta)) \geq \delta^{h+\eps}$ implies that $\Theta(\delta)\eqdef\inf\brc{u:\bk{\ell^u}\geq 2\Lambda(\delta)} \leq \delta$. Therefore, the left hand term in \eqref{eq:mmass_branching_tail0} is upper bounded by
  \begin{align*}
    \Nb\pthB{ \Theta(\delta)\leq\delta \cap n_{b,\kappa}(\rho,\delta) \leq \delta^{h+2\eps-\tfrac{\gamma}{\gamma-1}} }
    &\leq \Nb\pthB{ \Theta(\delta)\leq\delta \cap \inf_{\Theta(\delta)+\ivff{0,\delta}} \bk{\ell^u} \leq \Lambda(\delta) } \\
    &+ \Nb\pthB{ \Theta(\delta)\leq\delta \cap \inf_{\Theta(\delta)+\ivff{0,\delta}} \bk{\ell^u} > \Lambda(\delta) \cap n_{b,\kappa}(\rho,\delta) \leq \delta^{h+2\eps-\tfrac{\gamma}{\gamma-1}} }.
  \end{align*}
  Noting that $\Lambda(\delta)v(\delta) \geq c\delta^{-\eps}$, Lemma~\ref{lemma:local_time_sup_inf_tail} then provides a bound on the first term:
  \begin{align*}
     \Nb\pthB{ \Theta(\delta)\leq\delta \cap \inf_{\Theta(\delta)+\ivff{0,\delta}} \bk{\ell^u} \leq \Lambda(\delta) } \leq c_0\,v(\delta)\exp\pthb{-c_1\,\delta^{-\eps}}.
  \end{align*}
  On the other hand, the second term is itself bounded by $\Nb\pthb{ \bk{\ell^{\delta}} > \Lambda(\delta) \cap Z(\delta,\kappa\delta) \leq \delta^{h+2\eps-\tfrac{\gamma}{\gamma-1}} }$ as we observe that $Z(\delta,\kappa\delta)$ and $n_{b,\kappa}(\rho,\delta)$ coincide. Recall that that given $\Gi_\delta$, $Z(\delta,\kappa\delta)$ is a Poisson random variable parametrised by $\bk{\ell^\delta}v(\kappa\delta)$. Furthermore, conditionally on the event $\bk{\ell^{\delta}} > \Lambda(\delta)$, $\bk{\ell^\delta}v(\kappa\delta) \geq c\,\delta^{-\eps}\cdot\delta^{h+2\eps-\tfrac{\gamma}{\gamma-1}}$.
  As a consequence, Chernoff bound~\eqref{eq:chernoff_poisson} entails
  \begin{align*}
    \Nb\pthB{ \bk{\ell^{\delta}} > \tfrac{1}{2}\delta^{h-1+\eps} \cap Z(\delta,\kappa\delta) \leq \delta^{h+2\eps-\tfrac{\gamma}{\gamma-1}} }
    &= v(\delta) \Nb_\delta\pthB{ \bk{\ell^{\delta}} > \tfrac{1}{2}\delta^{h-1+\eps} \Nb\pthcB{Z(\delta,\kappa\delta) \leq \delta^{h+2\eps-\tfrac{\gamma}{\gamma-1}}}{\Gi_\delta} } \\
    &\leq c_0\,v(\delta) \exp\pthb{-c_1\,\delta^{-\eps}},
  \end{align*}
  therefore concluding the first part of the proof.\vsp

  The proof of the second inequality is analogue to the first one. Begin by observing:
  \begin{align*}
    \Nb\pthB{ n_{b,\kappa}(\rho,\delta) \geq \delta^{h+\eps-\tfrac{\gamma}{\gamma-1}} \cap \mb(B(\rho,2\delta)) \leq \delta^{h+2\eps} }
    &\leq \Nb\pthB{ Z(\delta,\kappa\delta) \geq \delta^{h+\eps-\tfrac{\gamma}{\gamma-1}} \cap \bk{\ell^\delta} \leq \delta^{h+2\eps-1} } \\
    &+ \Nb\pthB{ \bk{\ell^\delta} > \delta^{h+2\eps-1} \cap \mb(B(\rho,2\delta)) \leq \delta^{h+2\eps} }.
  \end{align*}
  Chernoff bound provides similarly a bound on the first term:
  \begin{align*}
    \Nb\pthB{ Z(\delta,\kappa\delta) \geq \delta^{h+\eps-\tfrac{\gamma}{\gamma-1}} \cap \bk{\ell^\delta} \leq \delta^{h+2\eps-1} } \leq v(\delta)\exp\pthb{-c_1\,\delta^{-\eps}}.
  \end{align*}
  Still setting $\Lambda(\delta) \eqdef \delta^{h-1+\eps}/2$ and $\tau \eqdef \delta$, and using the notation of Lemma~\ref{lemma:local_time_sup_inf_tail}, the second part is bounded by
  \begin{align*}
    \Nb\pthb{ \Theta(\delta) \leq \delta \cap \mb(B(\rho,2\delta)) \leq \delta^{h+2\eps} }
    &\leq \Nb\pthB{ \Theta(\delta) \leq \delta \cap \inf_{\Theta(\delta)+\ivff{0,\delta}} \bk{\ell^u} \leq \Lambda(\delta) } \\
    &+ \Nb\pthB{ \Theta(\delta) \leq \delta  \cap \inf_{\Theta(\delta)+\ivff{0,\delta}} \bk{\ell^u} > \Lambda(\delta) \cap \mb(B(\rho,2\delta)) \leq \delta^{h+2\eps} }.
  \end{align*}
  Conditionally on the event $\brc{\Theta(\delta) \leq \delta}$,
  \begin{align*}
    \mb(B(\rho,2\delta)) \geq \int_{\Theta(\delta)+\ivff{0,\delta}} \bk{\ell^u}\dt u \geq \delta \cdot \inf_{\Theta(\delta)+\ivff{0,\delta}} \bk{\ell^u},
  \end{align*}
  therefore proving the second term is null. Finally, the first one is bounded by $c_0\,v(\delta)\exp\pthb{-c_1\,\delta^{-\eps}}$ using Lemma~\ref{lemma:local_time_sup_inf_tail}.
\end{proof}

In order to prove the equivalence between the mass measure and branching exponents, we aim to use of the previous estimates on a finite collection of vertices in the tree that are analogues of dyadic numbers. A construction of such a collection has already been presented by \citet{Duquesne.LeGall-2005}. Nevertheless, in order to make use of the re-rooting invariance principle, one needs to pick uniformly vertices in the tree, making the former construction inappropriate for our purpose. As a consequence, we study in the next lemma a few properties of a collection of vertices uniformly picked on the tree.
\begin{lemma}  \label{lemma:unif_dyadic_vertices}
  Suppose $\gamma\in\ivoo{0,2}$ and $\eps>0$. $\Nb$-a.e., for every $n\in\N$, let $\Sigma_n \eqdef \brcb{\sigma_{1,n}, \sigma_{2,n},\dotsc}\subset\Ti$ be a finite collection of independent random vertices such that for every $n\geq 1$
  \begin{enumerate}[(i)]
    \item $\#{\Sigma_n} = \ceilB{2^{\tfrac{\gamma}{\gamma-1}n+\eps n}}$ and $\Sigma_n\subset\Sigma_{n+1}$;
    \item for every $\sigma_{i,n}\in\Sigma_n$, $\sigma_{i,n}$ follows the uniform distribution $\tfrac{\mb(\dt\sigma)}{\mb(\Ti)}$ on the tree.
  \end{enumerate}
  Then, $\Nb(\dt\Ti)$-a.e., for any $n$ sufficiently large,
  \begin{align*}
    \forall \sigma\in\Ti, \ \exists \sigma_{i,n}\in\Sigma_n;\quad d\pth{\sigma,\sigma_{i,n}} \leq 2^{-n+1}.
  \end{align*}
\end{lemma}
\begin{proof}
  If $\zeta$ denotes the lifetime of the excursion $H$ encoding $\Ti$, we know that the mass measure $\mb$ is induced by the Lebesgue measure $\Li$ on $\ivff{0,\zeta}$ (see \cite[Eq. 32]{Duquesne.LeGall-2002}). Namely, for any Borel set of $(\Ti,d)$, $\mb(A) = \Li\pthb{p_H^{-1}(A)}$.

  As a consequence, for every $n\geq 1$, one can construct a proper collection $\Sigma_n$ as follows. Let $\Ei_n=\brcb{X_{1,n},X_{2,n},\dotsc}$ be a finite collection of i.i.d. uniform random variables on $\ivff{0,\zeta}$ such that $\#\Ei_n=\ceilb{2^{\tfrac{\gamma}{\gamma-1}n+\eps n}}$ and $\Ei_n\subset\Ei_{n+1}$. If we define $\Sigma_n = p_H(\Ei_n)$, one clearly observes that $\Sigma_n$ is a collection r.v. on the tree satisfying the previous assumptions.

  Let us set $\theta_n=\delta_n^{\tfrac{\gamma}{\gamma-1}+\eps/2}$ and $J\subset\ivff{0,\zeta}$ be a close interval of size $\theta_{n-1}$. Then,
  \begin{align*}
    \prb{ J\cap\Ei_n =\vset} = \prod_{X_{i,n}\in\Ei_n} \pr{X_{i,n}\notin J} = \pthbb{ 1 - \frac{\theta_{n-1}}{\zeta} }^{\#\Ei_n}.
  \end{align*}
  As a consequence,
  \begin{align*}
    \prb{\exists I \text{ interval of size } \theta_n : I\cap \Ei_n=\vset} \leq \ceilb{ \zeta\theta_n^{-1} }\pthbb{ 1 - \frac{\theta_{n-1}}{\zeta} }^{\#\Ei_n}\leq \exp\pthb{-c_0 \delta_n^{-\eps}},
  \end{align*}
  by standard estimates. Hence, Borel--Cantelli lemma entails that for any $x\in\ivff{0,\zeta}$, for any $n$ sufficiently large, there exists $X_{i,n}\in\Ei_n$ such that $\abs{X_{i,n}-x} \leq \theta_n$. In addition, as proved by \citet[Th. 1.4.4]{Duquesne.LeGall-2002}, the height process is locally Hölder continuous for any exponent $\alpha\in\ivoo{0,\tfrac{\gamma-1}{\gamma}}$. In particular, setting $\alpha=\pthb{\tfrac{\gamma}{\gamma-1}+\tfrac{\eps}{2}}^{-1}$, $x\in\ivff{0,\zeta}$ and $X_{i,n}\in\Ei_n$ such that $\abs{X_{i,n}-x} \leq \theta_n$, we get:
  \begin{align*}
    d_\Ti\pthb{p_H(x),p_H(X_{i,n})} = H(x)+H(X_{i,n}) - 2\inf_{\ivff{x\wedge X_{i,n},x\vee X_{i,n}}} H_u \leq 2\theta_n^\alpha = 2\delta_n,
  \end{align*}
  therefore proving the desired approximating property on the collection $\Sigma_n$.
\end{proof}

We may now present the proof of the second part of Proposition~\ref{prop:scaling_exponents}.
\begin{proof}[Proof of Proposition~\ref{prop:scaling_exponents} - Part two]
  Let us set $\eps>0$, $h\in\ivoob{-\infty,\tfrac{\gamma}{\gamma-1}-2\eps}$ and, for every $n\in\N$, $\Sigma_n \eqdef \brcb{\sigma_{1,n}, \sigma_{2,n},\dotsc}\subset\Ti$ be the random collection of vertices constructed in Lemma~\ref{lemma:unif_dyadic_vertices}. By convention, we also set $\sigma_{0,n}\eqdef\rho$. For any $\sigma_{i,n}\in\Sigma_n$, let $A(\sigma_{i,n})$ be the event:
  \begin{align*}
    A\pthb{\sigma_{i,n}} &= \brcB{ \mb(B(\sigma_{i,n},\delta_n)) \geq \delta_n^{h+\eps} \text{ and } n_{b,\kappa}(\sigma_{i,n},\delta_n) \leq \delta_n^{h+2\eps-\tfrac{\gamma}{\gamma-1}} }.
  \end{align*}
  Using the re-rooting invariance principle and Lemma~\ref{lemma:scaling_exponents1}, we get
  \begin{align*}
    \Nb\pthbb{\bigcup_{\sigma_{i,n}\in\Sigma_n} A\pthb{\sigma_{i,n}} } \leq \sum_{i=1}^{\#\Sigma_n} \Nb\pthb{A\pthb{\sigma_{i,n}} } = \#\Sigma_n \Nb\pthb{A\pthb{\rho} } \leq c_0\,\#\Sigma_n\,v(\delta_n) \exp\pthb{-c_1\delta_n^{-\eps}}.
  \end{align*}
  Borel--Cantelli lemma therefore entails that $\Nb$-a.e. for any $n$ sufficiently large,
  \begin{align*}
    \forall \sigma_{i,n}\in\Sigma_n;\quad  \mb(B(\sigma_{i,n},\delta_n)) \geq \delta_n^{h+\eps} \Longrightarrow n_{b,\kappa}(\sigma_{i,n},\delta_n) > \delta_n^{h+2\eps-\tfrac{\gamma}{\gamma-1}}.
  \end{align*}

  Let us now set $\sigma\in\Ti$ such that $\alpha_\mb(\sigma,\Ti) \leq h$. There exists an infinite subsequence of scales $\delta_n$ such that $\mb(B(\sigma_{i,n},\delta_n)) \geq (3\delta_n)^{h+\eps}$. Then according to Lemma~\ref{lemma:unif_dyadic_vertices}, there exists $\sigma_{i,n}\in\Sigma_n$ such that $d(\sigma,\sigma_{i,n})\leq 2\delta_n$. As a consequence,
  \begin{align*}
    \mb(B(\sigma_{i,n},3\delta_n)) \geq (3\delta_n)^{h+\eps} \text{ and } n_{b,\kappa}(\sigma_{i,n},3\delta_n) > (3\delta_n)^{h+2\eps-\tfrac{\gamma}{\gamma-1}}.
  \end{align*}
  By choosing the parameter $\kappa\geq 2$, the geometry of trees yields $n_b(\sigma,5\delta_n) > (3\delta_n)^{h+2\eps-\tfrac{\gamma}{\gamma-1}}$, proving that $\alpha_b(\sigma,\Ti) \geq \tfrac{\gamma}{\gamma-1}-h-2\eps$.

  Using completely analogue arguments and the second bound presented in Lemma~\ref{lemma:scaling_exponents1}, we eventually get for any $h\in\ivoob{-\infty,\tfrac{\gamma}{\gamma-1}}$
  \begin{align*}
    \forall \sigma\in\Ti;\quad \alpha_\mb(\sigma,\Ti) \leq h \Longleftrightarrow \alpha_b(\sigma,\Ti) \geq \tfrac{\gamma}{\gamma-1}-h,
  \end{align*}
  In order to obtain the complete equality between the two scaling exponents, one needs to treat the specific case $\alpha_\mb(\sigma,\Ti)=\tfrac{\gamma}{\gamma-1}$. For that purpose, let us recall the result obtained by \citet{Duquesne.Wang-2014}:
  \begin{align*}
    k_\gamma \leq \liminf_{\delta\rightarrow 0} \frac{1}{f_\gamma(\delta)} \inf_{\sigma\in\Ti} \mb\pthb{B(\sigma,\delta)}\quad\text{where } f_\gamma(\delta) \eqdef \frac{ \delta^{\tfrac{\gamma}{\gamma-1}} }{ \pth{\log1/\delta}^{\tfrac{\gamma}{\gamma-1}} }.
  \end{align*}
  In particular, one gets from the previous bound that $\Nb$-a.e. for every $\sigma\in\Ti$, $\alpha_\mb(\sigma,\Ti)\leq\tfrac{\gamma}{\gamma-1}$. Furthermore, it is a direct consequence of the definition of the branching exponent that $\alpha_b(\sigma,\Ti)\geq 0$ for every $\sigma\in\Ti$. Therefore, since $h=\tfrac{\gamma}{\gamma-1}$ is the only possible value left and the two exponents coincide on the interval $\ivoob{-\infty,\tfrac{\gamma}{\gamma-1}}$, one must have $\Nb$-a.e.
  \begin{align*}
    \forall \sigma\in\Ti;\quad \alpha_\mb(\sigma,\Ti) = \frac{\gamma}{\gamma-1} - \alpha_b(\sigma,\Ti),
  \end{align*}
  which concludes the second part of the proof of Proposition~\ref{prop:scaling_exponents}.
\end{proof}




\subsection{Upper-bound estimates}  \label{ssec:stable_trees_ub}


This section is devoted to the proof of upper bounds in Theorems~\ref{th:levy_tree_upper_spectrum1}, \ref{th:levy_tree_upper_spectrum2} and \ref{th:levy_tree_upper_spectrum3}. To begin with, we obtain a uniform upper bound on the Hausdorff dimension of the level sets of stable trees.
\begin{lemma}  \label{lemma:usp_global_ub}
  $\Nb(\dt\Ti)$-a.e. for every nonempty Borel set $F$ of $\ivoo{0,h(\Ti)}$,
  \[
    \dimH \,\Ti(F) \leq \frac{1}{\gamma-1} + \dimH F.
  \]
  In addition, $\Nb(\dt\Ti)$-a.e. for every $a>0$, $\dimP \,\Ti(a) \leq \frac{1}{\gamma-1}$.
\end{lemma}
\begin{proof}
  For every $k,n\in\N$, let $Z(k,n)\eqdef Z(k\delta_n,\delta_n)$, i.e. the number of subtrees $\Ti_\sigma$ rooted at level $k\delta_n$ and higher than $\delta_n$, where $\delta_n\eqdef 2^{-n}$. Due to the branching property, under $\Nb_{k\delta_n}$ and given $\Gi_{k\delta_n}$, $Z(k,n)$ is a Poisson random variable with parameter $\lambda_{k,n} = \bk{\ell^{k\delta_n}}\Nb\pthb{\bk{\ell^{\delta_n}}>0} = \bk{\ell^{k\delta_n}}v(\delta_n)$. Hence, setting $x_{k,n} = \pthb{8\bk{ \ell^{k\delta_n} } + 1}v(\delta_n)$, Chernoff bound \eqref{eq:chernoff_poisson} entails
  \begin{align*}
    \log\pthb{ \Nb_{k\delta_n}\pthcb{Z(k,n) \geq x_{k,n}}{\Gi_{k\delta_n}} }
    \leq -\lambda_{k,n} - x_{k,n}\pthb{\log\pthb{x_{k,n}\lambda_{k,n}^{-1}}-1}
    \leq -v(\delta_n).
  \end{align*}
  The latter inequality implies
  \begin{align*}
    \Nb\pthb{Z(k,n) \geq x_{k,n} }
    = v\pth{k\delta_n} \,\Nb_{k\delta_n}\pthb{Z(k,n) \geq x_{k,n} }
    \leq v\pth{k\delta_n} \exp\pthb{-v(\delta_n)}.
  \end{align*}
  Then, since $\tfrac{1}{\gamma-1} > 1$,
  \begin{align*}
    \sum_{n\in\N} \ \Nb\pthbb{ \bigcup_{k=1}^{\infty} Z(k,n) \geq x_{k,n} }
    \leq \sum_{n\in\N} \sum_{k=1}^{\infty}  \Nb\pthb{Z(k,n) \geq x_{k,n}  }
    &\leq \sum_{n\in\N} \sum_{k=1}^{\infty} k^{-\tfrac{1}{\gamma-1}}  v\pth{\delta_n} \exp\pthb{-v(\delta_n)} \\
    &\leq c_0 \sum_{n\in\N} v\pth{\delta_n} \exp\pthb{-v(\delta_n)} < \infty.
  \end{align*}
  Hence, owing to Borel--Cantelli lemma, there exists $\Nb(\dt\Ti)$-a.e. $n_0(\Ti)\in\N$ such that for all $n\geq n_0(\Ti)$
  \begin{align}  \label{eq:bound_cover}
    \forall k\in\N\setminus\brc{0};\quad Z(k,n) \leq \pthb{8\bk{ \ell^{k\delta_n} } + 1}v(\delta_n).
  \end{align}

  Relying on the last bound, we may now prove the main statement of the lemma. Let $F$ be a nonempty Borel set of $\ivoo{0,h(\Ti)}$ and $s>\dimH F$. There exists a constant $c(s)>0$ such that for any $\delta>0$, there is a $\delta$-cover $\Oi=(O_i)_{i\in\N}$ of $F$ satisfying $\sum_{i\in\N} \abs{O_i}^s < c(s)$. Note that without any loss of generality, we may restrict ourselves to covers where $O_i = \ivff{k\delta_n,(k+1)\delta_n}\eqdef I_{k,n}$, for some $k,n\in\N$ depending on the index $i$. Then, for every $n\in\N$, if we denote by $\Ji(\Oi,n)$ the collection $\Ji(\Oi,n) = \brc{O\in\Oi : \abs{O}=\delta_n}$, we clearly get $\sum_{n\in\N} \#\Ji(\Oi,n) \,\delta_n^s < c(s)$.

  Let us set $n,k\in\N$ and $a\in\ivfo{(k+1)\delta_n,(k+2)\delta_n}$. A simple geometric argument on tree entails that the number of balls $N(a,n)$ of radius $2\delta_n$ necessary to cover $\Ti(a)$ is smaller than $Z(k,n)$. Hence, using this property, we may deduce from $\Oi$ the construction of a $\delta$-cover $\Vi = (V_i)_{i\in\N}$ of $\Ti(F)$. The latter satisfies, for any $\delta$ sufficiently small and any $\eta>0$
  \begin{align*}
    \sum_{i\in\N} \abs{V_i}^\eta
    &\leq \sum_{n\in\N}  \sum_{I_{k,n}\in\Ji(\Oi,n)} \pthb{8\bk{ \ell^{k\delta_n} } + 1}v(\delta_n) \delta_n^{\eta} \\
    &\leq c_0 \sum_{n\in\N} \#\Ji(\Oi,n) \delta_n^{\eta-1/(\gamma-1)},
  \end{align*}
  owing to Equation~\eqref{eq:bound_cover} and the boundedness of the local time $u\mapsto\bk{\ell^{u}}$. The latter sum is upper bounded by $c(s)$ for any $\eta \geq s + 1/(\gamma-1)$, therefore proving the first inequality.

  The bound on the packing dimension of level sets is a consequence of Equation~\eqref{eq:bound_cover} which entails: $\dimBu \,\Ti(a) \leq \frac{1}{\gamma-1}$.
\end{proof}
Note that Lemma~\ref{lemma:usp_spectrum_ub1} provides the upper bound to Corollary~\ref{cor:hdim_level_sets} and Equation~\eqref{eq:cor_dimP_images} on the uniform Hausdorff and packing dimensions of level sets.

Extending the previous estimates, we may now obtain the upper bound on the multifractal spectrum of the local time. For that purpose, we will study a slightly different class of fractal sets defined by
\begin{align*}
  F_\ell(h,\Ti) = \brcb{\sigma\in\Ti : \alpha_\ell(\sigma,\Ti) \leq h} \quad\text{and}\quad F_\mb(h,\Ti) = \brcb{\sigma\in\Ti : \alpha_\mb(\sigma,\Ti) \leq h}.
\end{align*}
Since $E_\star(h,\Ti)\subset F_\star(h,\Ti)$, any upper bound proved on the latter collection also holds on the former.
\begin{lemma}  \label{lemma:usp_spectrum_ub1}
  $\Nb(\dt\Ti)$-a.e. for every nonempty Borel set $F$ of $\ivoo{0,h(\Ti)}$,
  \begin{align}  \label{eq:usp_spectrum_ub1}
    \forall h\in\ivffb{\tfrac{1}{\gamma},\tfrac{1}{\gamma-1}};\quad \dimH \pthb{ F_\ell(h,\Ti)\cap \Ti(F) } \leq \gamma h - 1 + \dimP F.
  \end{align}
\end{lemma}
\begin{proof}
  Let us first briefly recall a regularisation argument implying that it is sufficient to prove the bound \eqref{eq:usp_spectrum_ub1} with the upper box dimension of $F$. It is known (see e.g. \citet[Th. 5.1.1]{Mattila-1995}) that the packing dimension of $F$ can be characterised as following:
  \begin{align*}
    \dimP F = \inf\brcbb{ \sup_i \dimBu F_i : F=\bigcup_{i\in\N} F_i, F_i\text{ is bounded} }.
  \end{align*}
  As a consequence, there exists an increasing sequence of sets $(F_i)_{i\in\N}$ such that $F=\cup_{i\in\N} F_i$ and $\dimBu F_i\rightarrow_\infty \dimP F$. To obtain Inequality~\eqref{eq:usp_spectrum_ub1}, it is therefore sufficient to verify the latter on every $F_i$ for the upper box dimension, i.e. proving $\dimH \pth{ F_\ell(h,\Ti)\cap \Ti(F) } \leq \gamma h - 1 + \dimBu F$ for any set $F$.

  In the particular case $h=\tfrac{1}{\gamma-1}$, the upper bound is a consequence of Lemma~\ref{lemma:usp_global_ub}. Hence, from now on, let us set $h\in\ivfob{\tfrac{1}{\gamma},\tfrac{1}{\gamma-1}}$ and $\eps>0$.
  For every $k,n\in\N$, we define the random variable
  \[
    Z(k,n,h) = \#\brcB{ \Ti_\sigma\in\Tbb(k\delta_n,\delta_n) : \sup_{\ivff{\delta_n,2\delta_n}} \bk{\ell^u}(\Ti_\sigma) \geq \delta_n^{h+\eps} },
  \]
  where $\delta_n\eqdef 2^{-n}$. Under $\Nb_{k\delta_n}$ and given $\Gi_{k\delta_n}$, $Z(k,n,h)$ is a Poisson random variable parametrised by $p_n\bk{ \ell^{k\delta_n} }$, where
  \begin{align*}
    p_n \eqdef v(\delta_n)\Nb_{\delta_n}\pthB{ \sup_{\ivff{\delta_n,2\delta_n}} \bk{\ell^u} \geq \delta_n^{h+\eps} } \asymp \delta_n^{1-\gamma (h+\eps)}.
  \end{align*}
  according to Lemma~\ref{lemma:local_time_sup_tails}.
  Setting $x_{k,n}\eqdef (\bk{ \ell^{k\delta_n} }+1) \delta_n^{1-\gamma h-2\eps}$ and noting that $1-\gamma h\leq 0$, Chernoff bound \eqref{eq:chernoff_poisson} entails $\Nb_{k\delta_n}\pthc{Z(k,n,h) \geq x_{k,n}}{\Gi_{k\delta_n}}\leq \exp\pth{-\delta_n^{-\eps}}$. The latter inequality implies
  \begin{align*}
    \Nb\pthb{Z(k,n,h) \geq x_{k,n} }
    = v\pthb{k\delta_n} \Nb_{k\delta_n}\pthb{Z(k,n) \geq x_{k,n} }
    \leq v(\delta_n)k^{-1/(\gamma-1)} \exp\pth{-\delta_n^{-\eps}}.
  \end{align*}
  Hence, recalling that $\tfrac{1}{\gamma-1} > 1$,
  \begin{align*}
    \sum_{n\in\N} \ \Nb\pthbb{ \bigcup_{k=1}^{\infty} Z(k,n,h) \geq x_{k,n} }
    \leq \sum_{n\in\N} \sum_{k=1}^{\infty} \Nb\pthb{Z(k,n,h) \geq x_{k,n} }
    \leq c_0 \sum_{n\in\N} v(\delta_n) \exp\pth{-\delta_n^{-\eps}} < \infty.
  \end{align*}
  Owing to Borel--Cantelli lemma, $\Nb(\dt\Ti)$-a.e. there exists $n_0(\Ti)\in\N$ such that for all $n\geq n_0(\Ti)$
  \begin{align}  \label{eq:bound_cover2}
    \forall k\in\N\setminus\brc{0};\quad Z(k,n,h) \leq (\bk{ \ell^{k\delta_n} }+1) \delta_n^{1-\gamma h-2\eps}.
  \end{align}

  We may now prove the main statement of the lemma. Similarly to the proof of Lemma~\ref{lemma:usp_global_ub}, F is a nonempty Borel set of $\ivoo{0,h(\Ti)}$ and for every $n\in\N$, we denote by $\Ji(n)$ the collection of intervals of type $I_{k,n}\eqdef\ivfo{k\delta_n,(k+1)\delta_n}$ necessary to cover the former.
  For any $a\in\ivoob{0,h(\Ti)}$, observe that
  \begin{align*}
    F_\ell(h,\Ti)\cap\Ti(F)
    &\subset \brcB{ \sigma\in\Ti(F) : \limsup_{r\rightarrow 0} r^{-(h+\eps/2)}\,\ell^{a}(B(\sigma,2r)) = \infty}.
  \end{align*}
  Suppose $\sigma\in\Ti(F)$ is such that $\ell^{a}(B(\sigma,2r)) \geq r^{h+\eps/2}$. There exist $k,n\in\N$ such that $2r\in\ivfo{\delta_{n+1},\delta_n}$ and $a\in\ivfo{(k+1)\delta_n,(k+2)\delta_n}$. In addition, there is a unique subtree $\Ti_{\sigma'}\in\Tbb(k\delta_n,\delta_n)$ embedded in $\Ti$ which satisfies $B(\sigma,2r)\subset \Ti_{\sigma'}$ and
  \begin{align*}
    \sup_{\ivff{\delta_n,2\delta_n}} \bk{\ell^u}(\Ti_{\sigma'}) \geq \ell^a(B(\sigma,2r)) \geq r^{h+\eps/2} \geq \delta_n^{h+\eps}.
  \end{align*}
  The last property proves that for any $N\in\N$, we may cover the set $F_\ell(h,\Ti)\cap\Ti(F)$ with balls of radius $4\delta_n$, $n\geq N$, centered at levels $k\delta_n$. Therefore, if we denote by $\Vi=(V_i)_{i\in\N}$ the resulting $\delta$-cover of $F_\ell(h,\Ti)\cap\Ti(F)$, Equation~\eqref{eq:bound_cover2} entails for any $\eta>0$
  \begin{align*}
    \sum_{i\in\N} \abs{V_i}^\eta
    \leq \sum_{n\geq N} \sum_{I(k,n)\in\Ji(n)} Z(k-1,n,h) \,\delta_n^\eta
    \leq c_0\sum_{n\geq N} \#\Ji(n)\, \delta_n^{\eta + 1-\gamma h-2\eps}.
  \end{align*}
  The latter sums converges for any $\eta > \gamma h-1+2\eps+\dimBu F$, therefore proving that $\Nb$-a.e. $\dimH \pthb{ F_\ell(h,\Ti)\cap \Ti(F) } \leq \gamma h - 1 + \dimBu F$. The regularisation argument previously described leads to the desired result with the packing dimension $\dimP F$ and, finally, we note that the uniformity of the result in the variable $h$ is a straightforward consequence of the monotonicity of the collection $h\mapsto F_\ell(h,\Ti)\cap\Ti(F)$ for the inclusion.
\end{proof}

To conclude the first part of this section on the local time, we present an upper bound to Theorem~\ref{th:levy_tree_upper_spectrum2} (and as a corollary, the lower bound in Theorem~\ref{th:levy_tree_upper_spectrum3}).
\begin{lemma}  \label{lemma:usp_spectrum_ub2}
  Suppose $F$ is a nonempty Borel set. Then, $\Nb\pth{ \dt \Ti }$-a.e.,
  \begin{align}  \label{eq:usp_spectrum_ub2}
    \forall h\in\ivfob{0, \tfrac{1}{\gamma}};\quad \dimH \,\pthb{ F_\ell(h,\Ti)\cap \Ti(F) } \leq \gamma h - 1 + \dimP F_{|\Ti}.
  \end{align}
  where $F_{|\Ti}\eqdef F\cap\ivoo{0,h(\Ti)}$. In particular, $F_\ell(h,\Ti)\cap \Ti(F)$ is empty when $h < \tfrac{1-\dimH F_{|\Ti}}{\gamma}$. Moreover, $\Nb\pth{ \dt \Ti }$-a.e. for every level $a>0$, the set $F_\ell(h,\Ti)\cap \Ti(a)$ is either empty or has zero Hausdorff dimension.
\end{lemma}
\begin{proof}
  Let us begin by observing that it is sufficient to prove that $\dimH \,\pthb{ F_\ell(h,\Ti)\cap \Ti(F) } \leq \gamma h - 1 + \dimP F$ to obtain the optimal bound \eqref{eq:usp_spectrum_ub2}. Indeed, by applying the former inequality to any set of the form $F\cap\ivoo{0,a}$, $a\in\Q_+$, common properties on the Hausdorff and packing dimensions induce the precise estimate \eqref{eq:usp_spectrum_ub2}.  Moreover, similarly to Lemma~\ref{lemma:usp_spectrum_ub1}, with the help of a regularisation argument it is sufficient to prove the previous bound with the upper box dimension $\dimBu F$.

  Then, let us set $h\in\ivfo{0,\tfrac{1}{\gamma}}$ and $\eps>0$. For any $n,k\in\N$, we still designate by $Z(k,n,h)$ the r.v.
  \[
    Z(k,n,h)= \#\brcB{ \Ti_\sigma\in\Tbb(k\delta_n,\delta_n) :  \sup_{\ivff{\delta_n,2\delta_n}} \bk{\ell^u}(\Ti_\sigma) \geq \delta_n^{h+\eps} }.
  \]
  Under $\Nb_{k\delta_n}$ and given $\Gi_{k\delta_n}$, $Z(k,n,h)$ is a Poisson random variable of parametrised by $\lambda_{k,n} = \bk{ \ell^{k\delta_n} } v(\delta_n)\Nb_\delta\pthb{ \sup_{\ivff{\delta_n,2\delta_n}} \bk{\ell^u} \geq \delta_n^h } \asymp \bk{ \ell^{k\delta_n} }\delta_n^{1-\gamma (h+\eps)}$.

  We may first prove $Z(k,n,h)$ is finite for any $n$ sufficiently large. For every $p\in\N$,
  \begin{align*}
    \Nb_{k\delta_n}\pthcb{ Z(k,n,h)\geq p }{ \Gi_{k\delta_n} }
    \leq c(p)\bk{ \ell^{k\delta_n} }^p\delta_n^{p-p\gamma (h+\eps)}.
  \end{align*}
  $\eps$ can be supposed small enough to satisfy $1-\gamma (h+3\eps) > 0$ and there exists $p\in\N$ sufficiently large such that $p(1-\gamma(h+3\eps)) > \tfrac{1}{\gamma-1}>$. As a consequence,
  \begin{align*}
    \sum_{n\in\N} \sum_{k\geq 1} \Nb\pthB{ Z(k,n,h)\geq p \cap \sup_{u>0} \bk{\ell^u} \leq \delta_n^{-\gamma\eps} } \leq c(p) \sum_{n\in\N} \sum_{k\geq 1} v(k\delta_n) \delta_n^{p-p\gamma (h+3\eps)} < \infty.
  \end{align*}
  Borel--Cantelli lemma then entails that $\Nb$-a.e. for any $n$ sufficiently large,
  \begin{align*}
    \forall k\in\N\setminus\brc{0};\quad Z(k,n,h)< p \quad\text{or}\quad \sup_{u>0} \bk{\ell^u} > \delta_n^{-\gamma\eps}.
  \end{align*}
  Since the process $a\mapsto\bk{\ell^a}$ is almost surely càdlàg with bounded support, $\sup_{u>0} \bk{\ell^u} \leq \delta_n^{-\gamma\eps}$ for any $n$ sufficiently large, therefore proving that $Z(k,n,h) < p$ for all $k\geq 1$.

  We may now construct an optimal cover of $F_\ell(h,\Ti)\cap \Ti(F)$. We still denote by $\Ji(n)$ the collection of intervals of type $I_{k,n}\eqdef\ivfo{k\delta_n,(k+1)\delta_n}$ necessary to cover $F$. Then, let
  \begin{align*}
    N(n,h) \eqdef \#\brcb{k\in\N : I_{k,n}\in\Ji(n) \text{ and } Z(k-1,n,h)\geq 1}
  \end{align*}
  and observe that $\Nb\pth{ N(n,h) } = \sum_{I_{k,n}\in\Ji(n)} \Nb\pthb{ Z(k-1,n,h)\geq 1 }$. The branching property entails $\Nb_{k\delta_n}\pthcb{ Z(k,n,h)\geq 1 }{ \Gi_{k\delta_n} } \leq c_0\bk{ \ell^{k\delta_n} }\delta_n^{1-\gamma (h+\eps)}$, and thus, $\Nb\pthb{ Z(k,n,h)\geq 1 } \leq c_0\,\delta_n^{1-\gamma (h+\eps)}$ where the constant $c_0$ is independent of $k,n\in\N$ and $h$. Setting $s>\dimBu F$, we get
  \begin{align*}
    \Nb\pthb{ N(n,h) \geq \delta_n^{1-\gamma h-s-3\eps} }
    \leq \delta_n^{\gamma h-1+s+3\eps} \Nb\pthb{ N(n,h) }
    \leq c_0\,\#\Ji(n) \,\delta_n^{\eps+s}
    \leq c_1\,\delta_n^{\eps}.
  \end{align*}
  Borel--Cantelli lemma then yields: $\Nb$-a.e. for every $n$ sufficiently large, $N(n,h) \leq \delta_n^{1-\gamma h-s-3\eps}$.\vsp

  For any $a\in\ivoob{0,h(\Ti)}$ and $r>0$, set $n,k\in\N$ such that $2r\in\ivfo{\delta_{n+1},\delta_n}$ and $a\in\ivfo{(j+1)\delta_n,(j+2)\delta_n}$. Then, if $\ell^a\pthb{ B(\sigma,2r) } \geq r^{h+\eps/2}$, there exists a subtree $\Ti_{\sigma'}$ rooted at level $k\delta_n$ such that $\sup_{\ivff{\delta_n,2\delta_n}} \bk{\ell^u}(\Ti_{\sigma'}) \geq \delta_n^{h+\eps}$.

  Hence, for any $N\in\N$, we may cover the set $F_\ell(h,\Ti)\cap \Ti(F)$ with balls of radius $4\delta_n$, $n\geq N$, centered at levels $k\delta_n$ where $k$ is such that $I_{k,n}\in \Ji(n)$ and $Z(k-1,n,h)\geq 1$. In addition, since $Z(k-1,n,h) < p$, only $p$ balls are required for the cover at every level in the tree.

  If $\gamma h - 1 + \dimBu F < 0$, the bound $N(n,h) \leq \delta_n^{1-\gamma h-s-3\eps}$ clearly implies that $N(n,h) = 0$ for any $n\in\N$ sufficiently large (and $\eps$ small enough), hence proving the emptiness of the set $F_\ell(h,\Ti)\cap \Ti(F)$. In the second case, if we denote by $\Vi=(V_j)_{j\in\N}$ the cover obtained using the aforementioned construction, we get
  \begin{align*}
    \sum_{j\in\N} \abs{V_j}^\eta
    \leq \sum_{n\geq n_0} N(n,h) \,\delta_n^\eta
    \leq c_0\sum_{n\geq n_0} \delta_n^{\eta+1-\gamma h-s-3\eps}.
  \end{align*}
  The latter sums converges for any $\eta > \gamma h-1+s+3\eps$, therefore proving that $\Nb$-a.e. $\dimH \,\pthb{ F_\ell(h,\Ti)\cap \Ti(F) } \leq \gamma h - 1 + \dimBu F$. The two arguments previously outlined then ensure the optimal inequality~\eqref{eq:usp_spectrum_ub2}.

  Finally, due to the bound $Z(k-1,n,h) < p$, the set $F_\ell(h,\Ti)\cap \Ti(a)$ has zero Hausdorff dimension if it is non-empty.
\end{proof}

In the second part of this section, we investigate the upper bound on the multifractal spectrum of the mass measure. Since the structure of the proofs is very similar to the case of the local time, we only present the main steps and refer to the previous lemmas for the technical details.
\begin{lemma}  \label{lemma:usp_mspectrum_ub1}
  $\Nb(\dt\Ti)$-a.e. for every nonempty Borel set $F$ of $\ivoo{0,h(\Ti)}$,
  \begin{align}
    \forall h\in\ivffb{\tfrac{1+\gamma}{\gamma},\tfrac{\gamma}{\gamma-1}};\quad \dimH \pthb{ F_\mb(h,\Ti)\cap \Ti(F) } \leq \gamma (h-1) - 1 + \dimP F.
  \end{align}
\end{lemma}
\begin{proof}
  Similarly to Lemma~\ref{lemma:usp_spectrum_ub1}, we only need to prove the inequality for the upper box dimension, and we may assume $h<\tfrac{\gamma}{\gamma-1}$. Set $\eps>0$ and for every $k,n\in\N$, define the random variable
  \[
    Z(k,n,h) = \#\brcb{ \Ti_\sigma\in\Tbb(k\delta_n,\delta_n) : \mb(B(\sigma,\delta_n)) \geq \delta_n^{h+\eps} },
  \]
  Under $\Nb_{k\delta_n}$ and given $\Gi_{k\delta_n}$, $Z(k,n,h)$ is a Poisson random variable parametrised by $p_n\bk{ \ell^{k\delta_n} }$, where
  \begin{align*}
    p_n \eqdef v(\delta_n)\Nb_{\delta_n}\pthb{ \mb(B(\sigma,\delta_n)) \geq \delta_n^{h+\eps} } \sim_{n\rightarrow\infty} \delta_n^{\tfrac{\gamma^2-1}{\gamma-1}-\gamma (h+\eps)} = \delta_n^{\gamma+1-\gamma (h+\eps)}.
  \end{align*}
  according to Lemma~\ref{lemma:mmass_tail} and the self-similarity of stable trees. Chernoff bound \eqref{eq:chernoff_poisson} and Borel--Cantelli lemma then entail that $\Nb(\dt\Ti)$-a.e. for every $n$ sufficiently large
  \begin{align}  \label{eq:bound_cover_m2}
    \forall k\in\N\setminus\brc{0};\quad Z(k,n,h) \leq (\bk{ \ell^{k\delta_n} }+1) \delta_n^{\gamma+1-\gamma h-3\eps}.
  \end{align}

  Similarly to Lemma~\ref{lemma:usp_spectrum_ub1}, we observe that if $\sigma\in\Ti$ is such that $\mb(B(\sigma,2r)) \geq r^{h+\eps}$, for some $r>0$, then there exist $k,n\in\N$ and $\Ti_{\sigma'}\in\Tbb(k\delta_n,\delta_n)$ such that $2r\in\ivfo{\delta_{n+1},\delta_n}$ and $\mb(\sigma',\delta_n) \geq c\,\delta_n^{h+\eps}$, where the constant $c$ is independent of $k$ and $n$. Consequently, we may deduce from this property a cover of the set $F_\mb(h,\Ti)\cap\Ti(F)$ with balls of radius $\delta_n$, $n\geq N$, centered at levels $k\delta_n$. The bound \eqref{eq:bound_cover_m2} then provides an upper bound on the number of balls required for such a cover, entailing that $\Nb$-a.e. for any set $F\subset\ivoo{0,h(\Ti)}$,
  \begin{align*}
    \dimH \pthb{ F_\mb(h,\Ti)\cap \Ti(F) } \leq \gamma (h-1) - 1 + \dimBu F.
  \end{align*}
\end{proof}

\begin{lemma}  \label{lemma:usp_mspectrum_ub2}
  Suppose $F$ is a nonempty Borel set. Then, $\Nb\pth{ \dt \Ti }$-a.e.,
  \begin{align}  \label{eq:usp_mspectrum_ub2}
    \forall h\in\ivfob{0, \tfrac{1+\gamma}{\gamma}};\quad \dimH \,\pthb{ F_\mb(h,\Ti)\cap \Ti(F) } \leq \gamma (h-1) - 1 + \dimP F_{|\Ti}.
  \end{align}
  where $F_{|\Ti}\eqdef F\cap\ivoo{0,h(\Ti)}$. In particular, $F_\mb(h,\Ti)\cap \Ti(F)$ is empty when $h < \tfrac{\gamma+1-\dimH F_{|\Ti}}{\gamma}$. Moreover, $\Nb\pth{ \dt \Ti }$-a.e. for every level $a>0$, the set $F_\mb(h,\Ti)\cap \Ti(a)$ is either empty or has zero Hausdorff dimension.
\end{lemma}
\begin{proof}
  As previously, it sufficient to prove \eqref{eq:usp_mspectrum_ub2} with upper box dimension of $F$. We first bound uniformly the r.v. $Z(k,n,h) \eqdef \#\brcb{ \Ti_\sigma\in\Tbb(k\delta_n,\delta_n) : \mb(B(\sigma,\delta_n)) \geq \delta_n^{h+\eps} }$. Since the former is a Poisson random variable, for any $p\geq 1$
  \begin{align*}
    \Nb_{k\delta_n}\pthcb{ Z(k,n,h)\geq p }{ \Gi_{k\delta_n} }
    \leq c(p)\bk{ \ell^{k\delta_n} }^p\delta_n^{p(\gamma+1)-p\gamma (h+\eps)}.
  \end{align*}
  Assuming $\eps$ is sufficiently small and $p$ large enough to have $p(\gamma+1-\gamma(h+3\eps)) > \tfrac{1}{\gamma-1}>$, Borel--Cantelli lemma entails that $\Nb$-a.e. for any $n$ sufficiently large and every $k\in\N$, $Z(k,n,h) < p$.

  To construct the optimal cover of $F_\mb(h,\Ti)\cap \Ti(F)$, we introduce as well $N(n,h) \eqdef \#\brcb{k\in\N : I_{k,n}\in\Ji(n) \text{ and } Z(k-1,n,h)\geq 1}$, where $\Ji(n)$ denotes the collection of intervals of type $I_{k,n}\eqdef\ivfo{k\delta_n,(k+1)\delta_n}$ necessary to cover $F$. Still according to the law of $Z(k,n,h)$,  $\Nb\pthb{ Z(k,n,h)\geq 1 } \leq c_0\,\delta_n^{1+\gamma-\gamma (h+\eps)}$. Hence, setting $s>\dimBu F$, we get
  \begin{align*}
    \Nb\pthb{ N(n,h) \geq \delta_n^{1+\gamma-\gamma h-s-3\eps} }
    \leq \delta_n^{\gamma h-\gamma-1+s+3\eps} \Nb\pthb{ N(n,h) }
    \leq c_0\,\#\Ji(n) \,\delta_n^{\eps+s}
    \leq c_1\,\delta_n^{\eps},
  \end{align*}
  and therefore, $\Nb$-a.e. for every $n$ sufficiently large, $N(n,h) \leq \delta_n^{1+\gamma-\gamma h-s-3\eps}$.

  Using the same arguments presented in the proof of Lemma~\ref{lemma:usp_spectrum_ub2}, we deduce that $F_\mb(h,\Ti)\cap \Ti(F)$ is empty when $\gamma h - 1-\gamma + \dimBu F < 0$. In addition, using the previous bound and a similar construction of an optimal cover, we get $\dimH \,\pthb{ F_\mb(h,\Ti)\cap \Ti(F) } \leq \gamma (h-1) - 1 + \dimBu F$.
\end{proof}


\subsection{Lower bound estimates}  \label{ssec:stable_trees_lb}

In this section, we are interested in proving the lower bound of the uniform multifractal spectrum described in Theorem~\ref{th:levy_tree_upper_spectrum1}. We may begin by observing that owing to the connection between the local time and the mass measure exponents presented in Proposition~\ref{prop:scaling_exponents}, the key element is to investigate the behaviour of large masses of the local time.

As we might expect, the lower bound of the Hausdorff dimension is more technical to obtain. In addition to the classic difficulties related to the proof of a dimension's lower bound, we also must take into account the uniformity of the latter, i.e. for all $h\in\ivffb{\tfrac{1}{\gamma},\tfrac{1}{\gamma-1}}$ and any regular set $F$. In the multifractal literature, ubiquity theorems are often used to ensure the uniformity on the parameter $h$ (see for instance the seminal work of \citet{Jaffard-1999} on Lévy processes). In this section, we adopt an alternative approach. Instead of looking for a representation necessary to use these ubiquity theorems, we directly construct a proper family of Hausdorff measures using the distribution properties of stable trees. Note that our construction share common ideas with the work of \citet{Dembo.Peres.ea-2000} on the Hausdorff dimension of limsup random fractals.

\subsubsection{Proof of Theorem~\ref{th:levy_tree_upper_spectrum1} (lower bound)}

To start with, we aim to obtain uniform bounds on the size of the well-behaving parts of stable tree, i.e. the subset of $\Ti(u)$ which does not contain any exceptional small or large mass. For that purpose, we define for any $\delta,\rho\in\ivoo{0,\infty}$
\begin{align}  \label{eq:balance_set_Lambda}
  \Lambda(u,\rho,\delta) = \bigcup_{r\in\ivoo{\delta,\kappa\rho}}  \brcb{\sigma\in \Ti(u) : \ell^{u}(B(\sigma,2r)) \notin \ivffb{\underline{r}(r),\overline{r}(r)} },
\end{align}
where $\underline{r}(r) \eqdef \pthb{r \,g(r)^{1+\epsilon}}^{1/(\gamma-1)}$ and $\overline{r}(r) \eqdef \pthb{r / g(r)^{1+4\epsilon}}^{1/(\gamma-1)}$ for some $\epsilon>0$ fixed. The subset $\Lambda(u,\rho,\delta)$ thus gathers nodes in $\Ti(u)$ where the mass of the local time is locally exceptionally small or large at some scale $r\in\ivoo{\delta,\kappa\rho}$. In order to construct a ``well-behaving'' subtree, we aim to bound the contribution of the set $\Lambda(u,\rho,\delta)$ to the local time at level $u$.

Consequently, we begin by presenting a few technical lemmas on the small and large balls of the local time. In the rest of section, we fix the parameters $\kappa\in\ivoo{0,1}$ and $\epsilon>0$ whose values will be properly set later. In addition, we define the following event for any $v,w\geq 0$ and $\rho>0$
\begin{align*}
  \Ai(v,w,\rho) = \brcB{\Ti : \forall u\in\ivff{v,w};\ \bk{\ell^u}\in\ivffB{\rho^{\tfrac{1}{\gamma-1}}g(\rho)^{-\epsilon} , \rho^{\tfrac{1}{\gamma-1}} g(\rho)^{-2\epsilon} } }.
\end{align*}
For the sake of readability, we will also use simpler notations when possible: $\Ai(v,\rho)\eqdef \Ai(v,v,\rho)$ and $\Ai(\rho)\eqdef\Ai(\kappa\rho,\rho/\kappa,\rho)$.

To begin with, we present a bound on the tail behaviour of large masses of the local time.
\begin{lemma}  \label{lemma:usp_large_balls1}
  Suppose $\rho>0$. Then, for any $v\in\ivff{2\kappa\rho,\rho/\kappa}$ and all $\delta\in\ivoo{0,\kappa\rho}$
  \begin{align*}
    \Nb_{\kappa\rho}\bktbb{ \sup_{u\in\ivff{0,\delta}} \int \ell^{v+u}(\dt \sigma)\indi_{\brc{\ell^{v+u}(B(\sigma,2\delta)) > \overline{r}(\delta)}} \geq \ell(\rho,\delta) \cap \Ai(\rho) }
    \leq c_0\,\delta\rho^{-1} g(\delta)^{-\gamma(2+5\epsilon)},
  \end{align*}
  where $\ell(\rho,\delta) \eqdef g(\delta)^{1+\epsilon} \rho^{1/(\gamma-1)}$, $\overline{r}(\delta) \eqdef \pthb{\delta / g(\delta)^{1+4\epsilon}}^{1/(\gamma-1)}$ and the constant $c_0$ is independent of $\rho$, $\delta$ and $v$.
\end{lemma}
\begin{proof}
  Let us start with a simple observation. Since $v-\delta\geq \kappa\rho$, $\Ai(\rho) \subset \Ai(v-\delta,\rho)$ and it is thus sufficient to obtain an upper bound of the following expression
  \[
    \Nb_{v-\delta}\bktbb{ \sup_{u\in\ivff{0,\delta}} \int \ell^{v+u}(\dt \sigma)\indi_{\brc{\ell^{v+u}(B(\sigma,2\delta)) > \overline{r}(\delta)}} \geq \ell(\rho,\delta) \cap \Ai(v-\delta,\rho) }.
  \]
  We recall that $\Tbb(v-\delta,\delta)$ denotes the collection of subtrees rooted at level $v-\delta$ and higher than $\delta$. Then, for any $u\in\ivff{0,\delta}$ and $\sigma\in\Ti(v+u)$, there exists a unique $\Ti_{\sigma'}\in \Tbb(v-\delta,\delta)$ such that $B(\sigma,2\delta)\subset \Ti_{\sigma'}$. Hence, $\ell^{v+u}(B(\sigma,2\delta)) \leq \bk{\ell^{\delta+u}}(\Ti_{\sigma'})$ and
  \begin{align*}
    \sup_{u\in\ivff{0,\delta}} \int \ell^{v+u}(\dt \sigma)\indi_{\brc{\ell^{v+u}(B(\sigma,2\delta)) > \overline{r}(\delta)}}
    &\leq \sup_{u\in\ivff{0,\delta}} \sum_{\Ti_\sigma\in \Tbb(v-\delta,\delta)} \bk{\ell^{\delta+u}}(\Ti_\sigma) \indi_{\brc{\bk{\ell^{\delta+u}}(\Ti_\sigma) > \overline{r}(\delta)}} \\
    &\leq \sum_{\Ti_\sigma\in \Tbb(v-\delta,\delta)}  \sup_{u\in\ivff{0,\delta}} \bk{\ell^{\delta+u}}(\Ti_\sigma) \indi_{\brc{ \sup_{u\in\ivff{0,\delta}}\bk{\ell^{\delta+u}}(\Ti_\sigma) > \overline{r}(\delta)}}.
  \end{align*}
  Set $R(\delta) > \overline{r}(\delta)$, where the precise value will be set latter. For any $k\in\N$, define $r_k\eqdef 2^k \overline{r}(\delta)$. There exists $K\in\N$ such that $r_{K}\leq R(\delta) < r_{K+1}$. Then, let $N(k,\delta)$ denotes the number of subtrees in $\Tbb(v-\delta,\delta)$ such that $\sup_{u\in\ivff{0,\delta}}\bk{\ell^{\delta+u}}(\Ti_\sigma)\in\ivfo{r_k,r_{k+1}}$. Owing to the branching property of Lévy trees, given $\Gi_{v-\delta}$, $N(k,\delta)$ are independent Poisson random variables respectively parametrised by
  \begin{align*}
    \lambda_k
    = \bk{\ell^{v-\delta}} \,v(\delta) \Nb_{\delta}\pthB{ \sup_{u\in\ivff{0,\delta}}\bk{\ell^{\delta+u}} \in \ivfo{r_k,r_{k+1}} }
    \asymp \bk{\ell^{v-\delta}} \,\delta r_k^{-\gamma},
  \end{align*}
  according to Lemma~\ref{lemma:local_time_sup_tails}. Then, observing that
  \[
    \sum_{\Ti_\sigma\in \Tbb(v-\delta,\delta)}  \sup_{u\in\ivff{0,\delta}}\bk{\ell^{\delta+u}}(\Ti_\sigma) \indi_{\brc{ \sup_{u\in\ivff{0,\delta}}\bk{\ell^{\delta+u}} \in \ivfo{\overline{r}(\delta),R(\delta)} }}
    \leq 2 \sum_{k\leq K} r_k N(k,\delta),
  \]
  the classic Markov inequality on exponential moments entails
  \begin{align*}
    &\Nb_{v-\delta}\pthcbb{ \sum_{\Ti_\sigma\in \Tbb(v-\delta,\delta)}  \sup_{u\in\ivff{0,\delta}}\bk{\ell^{\delta+u}}(\Ti_\sigma) \indi_{\brc{ \sup_{u\in\ivff{0,\delta}}\bk{\ell^{\delta+u}} \in \ivfo{\overline{r}(\delta),R(\delta)} }} \geq \ell(\rho,\delta) }{ \Gi_{v-\delta}} \\
    &\leq \Nb_{v-\delta}\pthcbb{ 2 \sum_{k\leq K} r_k N(k,\delta) \geq \ell(\rho,\delta) }{ \Gi_{v-\delta}} \\
    &\leq \exp\pthb{-\mu \ell(\rho,\delta) } \Nb_{v-\delta}\pthcbb{ \exp\pthbb{ 2\mu\sum_{k\leq K} r_k N(k,\delta) }}{ \Gi_{v-\delta} },
  \end{align*}
  for some $\mu>0$.
  Using the independence on r.v. $N(k,\delta)$, the latter term is equal to
  \begin{align*}
    \Nb_{v-\delta}\pthcbb{ \exp\pthbb{ 2\mu\sum_{k\leq K} r_k N(k,\delta) }}{ \Gi_{v-\delta} }
    &= \Nb_{v-\delta}\pthcbb{ \exp\pthbb{ \sum_{k\leq K} \lambda_k\pthb{ \e^{2\mu r_k}-1 } } }{ \Gi_{v-\delta} }.
  \end{align*}
  Since $\lambda_k\asymp \bk{\ell^{v-\delta}}(\Ti) \,\delta r_k^{-\gamma}$ and $\gamma>0$, by choosing $\mu=R(\delta)^{-1}$, we obtain
  \begin{align*}
    \Nb_{v-\delta}\pthcbb{ \exp\pthbb{ \sum_{k\leq K} \lambda_k\pthb{ \e^{2\mu r_k}-1 } } }{ \Gi_{v-\delta} }
    &\leq \Nb_{v-\delta}\pthcbb{ \exp\pthbb{ \sum_{k\leq K} 8\mu \lambda_k r_k } }{ \Gi_{v-\delta} } \\
    &\leq \exp\pthB{ c_0 \mu \delta \bk{\ell^{v-\delta}} \sum_{k\leq K} r_k^{1-\gamma} } \\
    &\leq \exp\pthb{ c_1 \mu \delta \bk{\ell^{v-\delta}} \overline{r}(\delta)^{1-\gamma} } .
  \end{align*}
  Recalling $\overline{r}(\delta) = \pthb{\delta / g(\delta)^{1+4\epsilon}}^{1/(\gamma-1)}$,
  \begin{align*}
    \Nb_{v-\delta}\pthcbb{  2 \sum_{k\leq K} r_k N(k,\delta) \geq \ell(\rho,\delta) }{ \Gi_{v-\delta} }
    \leq  \exp\pthB{ -\mu\brcb{ \ell(\rho,\delta) - c_1 g(\delta)^{1+4\epsilon} \bk{\ell^{v-\delta}} } }.
  \end{align*}
  Finally, since $\ell(\rho,\delta) = g(\delta)^{1+\epsilon}\rho^{1/(\gamma-1)}$ and $\bk{\ell^{v-\delta}}\leq\rho^{1/(\gamma-1)}g(\rho)^{-2\epsilon}$ on the event $\Ai(v-\delta,\rho)$, we obtain by setting $\mu^{-1} = R(\delta) \eqdef g(\delta)^{2+2\epsilon} \rho^{1/(\gamma-1)}$
  \begin{align*}
    \Nb_{v-\delta}\pthbb{ 2 \sum_{k\leq K} r_k N(k,\delta) \geq \ell(\rho,\delta) \cap  \Ai(v-\delta,\rho) }
    &\leq \exp\pthb{ -c_0 \mu \rho^{1/(\gamma-1)} g(\delta)^{1+\epsilon} } \\
    &= \exp\pthb{ -c_0 g(\delta)^{-1-\eps} },
  \end{align*}
  where $c_0 > 1/2$ for any $\rho>0$ sufficiently small.

  Let us now consider the subtrees such that $\sup_{u\in\ivff{0,\delta}}\bk{\ell^{\delta+u}}(\Ti_\sigma)\geq R(\delta)$. Similarly, the number of such subtrees is Poisson random variable $N(R,\delta)$ parametrised by $\lambda_R \asymp \bk{\ell^{v-\delta}}\,\delta R(\delta)^{-\gamma}$.
  The probability of obtaining one such subtree is then,
  \begin{align*}
    \Nb_{v-\delta}\pthcb{ N(R,\delta) \geq 1 \cap \Ai(v-\delta,\rho) }{ \Gi_{v-\delta} }
    &\leq 1 - \exp\pthb{-c_1\, g(\rho)^{-2\epsilon} \rho^{1/(\gamma-1)} \,\delta R(\delta)^{-\gamma}} \\
    &\leq c_1\, \delta\rho^{-1} g(\rho)^{-2\epsilon} g(\delta)^{-\gamma(2+3\epsilon)} \\
    &\leq c_1\, \delta\rho^{-1} g(\delta)^{-\gamma(2+5\epsilon)}.
  \end{align*}
  To conclude the proof, we simply observe that $\delta\rho^{-1} g(\delta)^{-\gamma(2+5\epsilon)} \geq \exp\pthb{ -c_0 g(\delta)^{-1-\epsilon} },$ for any $\rho$ sufficiently small.
\end{proof}

We slightly extend the previous lemma to a more general form.
\begin{lemma}  \label{lemma:usp_large_balls2}
  Suppose $\rho>0$. Then, for any $v\in\ivff{2\kappa\rho,\rho/\kappa}$, all $\delta\in\ivoo{0,\kappa\rho}$ and any $\tau\in\ivff{\delta,\rho}$,
  \begin{align*}
    \Nb_{\kappa\rho}\bktbb{ \sup_{u\in\ivff{0,\tau}} \int \ell^{v+u}(\dt \sigma)\indi_{\brc{\ell^{v+u}(B(\sigma,2\delta)) > \overline{r}(\delta)}} \geq \ell(\rho,\delta) \cap \Ai(\rho) }
    \leq c_0\,\tau\rho^{-1} g(\delta)^{-\gamma(2+5\epsilon)},
  \end{align*}
  where the constant $c_0$ is independent of $\rho$, $\delta$, $\tau$ and $v$.
\end{lemma}
\begin{proof}
  Let us divide the interval $\ivff{0,\tau}$ into successive and disjoint subintervals of size $\delta$. Owing the previous Lemma~\ref{lemma:usp_large_balls1}, for every $k\in\brc{0,\dotsc,\ceil{\tau/\delta}}$,
  \begin{align*}
    \Nb_{\kappa\rho}\bktbb{ \sup_{u\in\ivff{0,\delta}} \int \ell^{v_k+u}(\dt \sigma)\indi_{\brc{\ell^{v_k+u}(B(\sigma,2\delta)) > r(\delta)}} \geq \ell(\rho,\delta) \cap \Ai(\rho) }
    \leq c_0\,\delta\rho^{-1} g(\delta)^{-\gamma(2+5\epsilon)},
  \end{align*}
  where $v_k\eqdef v+k\delta$. Therefore, since
  \begin{align*}
    &\brcbb{ \sup_{u\in\ivff{0,\tau}} \int \ell^{v+u}(\dt \sigma)\indi_{\brc{\ell^{v+u}(B(\sigma,2\delta)) > \overline{r}(\delta)}} \geq \ell(\rho,\delta) } \\
    &\subset \bigcup_{k\leq \ceil{\tau/\delta}} \brcbb{ \sup_{u\in\ivff{0,\delta}} \int \ell^{v_k+u}(\dt \sigma)\indi_{\brc{\ell^{v_k+u}(B(\sigma,2\delta)) > \overline{r}(\delta)}} \geq \ell(\rho,\delta) },
  \end{align*}
  the sum over $k\in\brc{0,\dotsc,\ceil{\tau/\delta}}$ entails the result.
\end{proof}

Let us now present similar estimates on the small masses of the local time.
\begin{lemma}  \label{lemma:usp_small_balls1}
  Suppose $\rho>0$. Then, for any $v\in\ivff{2\kappa\rho,\rho/\kappa}$, all $\delta\in\ivoo{0,2^{-1/\rho}}$ and $\tau\in\ivoo{0,\delta\, g(\delta)^{1+2\epsilon}}$
  \begin{align*}
    \Nb_{\kappa\rho}\bktbb{ \sup_{u\in\ivff{0,\tau}} \int \ell^{v+u}(\dt \sigma)\indi_{\brc{\ell^{v+u}(B(\sigma,2\delta)) < \underline{r}(\delta)}} \geq \ell(\rho,\delta) \cap \Ai(\rho) }
    \leq c_0\exp\pthB{-(\rho\delta^{-1})^{\tfrac{1}{\gamma-1}} g(\delta)^{-\epsilon} },
  \end{align*}
  where $\ell(\rho,\delta) \eqdef g(\delta)^{1+\epsilon} \rho^{1/(\gamma-1)}$, $\underline{r}(\delta) \eqdef \pthb{\delta \, g(\delta)^{1+\epsilon}}^{1/(\gamma-1)}$ and the constant $c_0$ is independent of $\rho$, $\delta$ and $v$.
\end{lemma}
\begin{proof}
  As previously, it is enough to bound the following quantity
  \[
    \Nb_{v-\tau}\bktbb{ \sup_{u\in\ivff{0,\tau}} \int \ell^{v+u}(\dt \sigma)\indi_{\brc{\ell^{v+u}(B(\sigma,2\delta)) < \underline{r}(\delta)}} \geq \ell(\rho,\delta) \cap \Ai(v-\vartheta,\rho) }.
  \]
  where $\vartheta=\delta-\tau$.
  Then, we similarly observe that
  \begin{align*}
    \sup_{u\in\ivff{0,\tau}} \int \ell^{v+u}(\dt \sigma)\indi_{\brc{\ell^{v+u}(B(\sigma,2\delta)) < \underline{r}(\delta)}}
    &\leq \sup_{u\in\ivff{0,\tau}} \sum_{\Ti_\sigma\in \Tbb(v-\vartheta,\vartheta)} \bk{\ell^{\vartheta+u}}(\Ti_\sigma) \indi_{\brc{\bk{\ell^{\vartheta+u}}(\Ti_\sigma) < \underline{r}(\delta)}} \\
    &\leq \underline{r}(\delta) \sum_{\Ti_\sigma\in \Tbb(v-\vartheta,\vartheta)} \indi_{\brc{ \inf_{u\in\ivff{0,\tau}}\bk{\ell^{\vartheta+u}}(\Ti_\sigma) < \underline{r}(\delta)}}.
  \end{align*}
  Let us denote by $N(\delta)$ the latest sum. Owing to the branching property of Lévy trees, given $\Gi_{v-\delta}$, $N(\delta)$ is a Poisson random variable parametrised by
  \begin{align*}
    \lambda(\delta)
    = \bk{\ell^{v-\vartheta}} \,v(\vartheta) \Nb_{\vartheta}\pthB{ \inf_{u\in\ivff{\vartheta,\delta}}\bk{\ell^{u}}(\Ti) <  \underline{r}(\delta) }.
  \end{align*}
  Lemma~\ref{lemma:local_time_inf_tail1} then entails
  \begin{align*}
    c_0 \, g(\delta)^{1+\epsilon} \leq \Nb_{\vartheta}\pthB{ \inf_{u\in\ivff{\vartheta,\delta}}\bk{\ell^{u}}(\Ti) <  \underline{r}(\delta) } \leq c_1\,g(\delta)^{1+\epsilon}  + \exp\pthb{-v(\tau) \, \underline{r}(\delta) } \leq c_2\,g(\delta)^{1+\epsilon},
  \end{align*}
  since $v(\tau) \, \underline{r}(\delta)\geq g(\delta)^{-\epsilon/(\gamma-1)}$. We also may note that
  \begin{align*}
    \brcbb{ \sup_{u\in\ivff{0,\tau}} \int \ell^{v+u}(\dt \sigma)\indi_{\brc{\ell^{v+u}(B(\sigma,2\delta)) < \underline{r}(\delta)}} \geq \ell(\rho,\delta) }
    &\subseteq \brcb{ N(\delta) \geq \ell(\rho,\delta) \,\underline{r}(\delta)^{-1} } \\
    &\subseteq \brcB{ N(\delta) \geq (\rho\delta^{-1})^{\tfrac{1}{\gamma-1}} g(\delta)^{-2\epsilon} },
  \end{align*}
  Chernoff bound then yields
  \begin{align*}
    \indi_{\Ai(v-\vartheta,\rho)} \Nb_{v-\vartheta}\pthcB{ N(\delta) \geq (\rho\delta^{-1})^{\tfrac{1}{\gamma-1}} g(\delta)^{-2\epsilon} }{ \Gi_{v-\vartheta} }
    &\leq\exp\pthB{-c_3\, (\rho\delta^{-1})^{\tfrac{1}{\gamma-1}} g(\delta)^{-2\epsilon} } \\
    &\leq c_4\exp\pthB{-(\rho\delta^{-1})^{\tfrac{1}{\gamma-1}} g(\delta)^{-\epsilon} }.
  \end{align*}
  This last equality concludes the proof of the lemma.
\end{proof}

We also extend the previous bound to a slightly more general form.
\begin{lemma}  \label{lemma:usp_small_balls2}
  Suppose $\rho>0$. Then, for any $v\in\ivff{2\kappa\rho,\rho/\kappa}$, all $\delta\in\ivoo{0,\kappa\rho}$ and any $\tau\in\ivff{\delta,\rho}$,
  \begin{align*}
    \Nb_{\kappa\rho}\bktbb{ \sup_{u\in\ivff{0,\tau}} \int \ell^{v+u}(\dt \sigma)\indi_{\brc{\ell^{v+u}(B(\sigma,2\delta)) < \underline{r}(\delta)}} \geq \ell(\rho,\delta) \cap \Ai(\rho) }
    \leq c_0\,\tau \rho^{-1} g(\rho)^{-1-2\epsilon},
  \end{align*}
  where the constant $c_0$ is independent of $\rho$, $\delta$, $\tau$ and $v$.
\end{lemma}
\begin{proof}
  We divide the interval $\ivff{0,\tau}$ into successive and disjoint subintervals of size $\delta g(\delta)^{1+2\epsilon}$. Owing the previous Lemma~\ref{lemma:usp_small_balls1}, for every $k\in\brc{0,\dotsc,\ceil{\tau/\delta}}$,
  \begin{align*}
    \Nb_{\kappa\rho}\bktbb{ \sup_{u\in\ivff{0,\delta g(\delta)^{1+2\epsilon}}} \int \ell^{v_k+u}(\dt \sigma)\indi_{\brc{\ell^{v_k+u}(B(\sigma,2\delta)) < \underline{r}(\delta)}} \geq \ell(\rho,\delta) \cap \Ai(\rho) }
    \leq c_0\exp\pthB{-(\rho\delta^{-1})^{\tfrac{1}{\gamma-1}} g(\delta)^{-\epsilon} },
  \end{align*}
  where $v_k\eqdef v+k\delta$. Therefore, using the same argument as in proof of Lemma~\ref{lemma:usp_large_balls2},
  \begin{align*}
    &\Nb_{\kappa\rho}\bktbb{ \sup_{u\in\ivff{0,\tau}} \int \ell^{v+u}(\dt \sigma)\indi_{\brc{\ell^{v+u}(B(\sigma,2\delta)) < \underline{r}(\delta)}} \geq \ell(\rho,\delta) \cap \Ai(\rho) } \\
    &\leq c_1 \tau \delta^{-1}g(\delta)^{-1-2\epsilon}\exp\pthB{-(\rho\delta^{-1})^{\tfrac{1}{\gamma-1}} g(\delta)^{-\epsilon} }
    \leq c_2 \tau \rho^{-1} g(\rho)^{-1-2\epsilon},
  \end{align*}
  which entails the result.
\end{proof}

The combination of Lemmas~\ref{lemma:usp_large_balls2} and \ref{lemma:usp_small_balls2} provides a uniform estimate on the size of exceptional behaviours of the local time at given scale, i.e. for any $v\in\ivff{2\kappa\rho,\rho/\kappa}$, all $\delta\in\ivoo{0,\kappa\rho}$ and any $\tau\in\ivff{\delta,\rho}$,
\begin{align*}
  \Nb_{\kappa\rho}\bktbb{ \sup_{u\in\ivff{0,\tau}} \int \ell^{v+u}(\dt \sigma)\indi_{\brc{\ell^{v+u}(B(\sigma,2\delta)) \notin \ivff{\underline{r}(\delta),\overline{r}(\delta)}}} \geq 2\ell(\rho,\delta) \cap \Ai(\rho) }
  \leq c_0\,\tau\rho^{-1} g(\delta)^{-\gamma(2+5\epsilon)}.
\end{align*}
In the following lemma, we may now present a uniform bound of these exceptional masses of the local time, i.e. estimate the size of $\Lambda(u,\rho,\delta)\subset\Ti(u)$, recalling that
\begin{align*}
  \Lambda(u,\rho,\delta) = \bigcup_{r\in\ivoo{\delta,\kappa\rho}}  \brcb{\sigma\in \Ti(u) : \ell^{u}(B(\sigma,2r)) \notin \ivffb{\underline{r}(r),\overline{r}(r)} },
\end{align*}
where $\underline{r}(r) \eqdef \pthb{r \, g(r)^{1+\epsilon}}^{1/(\gamma-1)}$ and $\overline{r}(r) \eqdef \pthb{r / g(r)^{1+4\epsilon}}^{1/(\gamma-1)}$.
\begin{lemma}  \label{lemma:usp_large_small_balls1}
  Suppose $\rho>0$ and $\delta\in\ivoo{0,\kappa\rho}$.
  Then, for all $v\in\ivff{2\kappa\rho,\rho/\kappa}$,
  \begin{align*}
    \Nb_{\kappa\rho}\bktB{ \sup_{u\in\ivff{0,\tau}} \ell^{v+u}(\Lambda(u+v,\rho,\delta)) \geq \ell(\rho) \cap \Ai(\rho) }
    \leq c_0\,\tau\rho^{-1} g(\delta)^{-1-\gamma(2+5\epsilon)},
  \end{align*}
  where $\ell(\rho) = g(\rho)^{\epsilon} \rho^{1/(\gamma-1)}$.
\end{lemma}
\begin{proof}
  Let us first observe that
  \begin{align*}
    \Lambda(u,\rho,\delta) \subset \bigcup_{2^{-k}\in\ivoo{\delta,\kappa\rho}}  \brcb{\sigma\in \Ti(u) : c_0 \,\ell^{u}(B(\sigma,2^{-k+1})) \notin \ivffb{\underline{r}(2^{-k}),\overline{r}(2^{-k})} },
  \end{align*}
  for some constant $c_0>0$. Then,
  \begin{align*}
    \sup_{u\in\ivff{0,\tau}} \ell^{v+u}(\Lambda(u+v,\rho,\delta))
    &\leq \sup_{u\in\ivff{0,\tau}} \sum_{2^{-k}\in\ivoo{\delta,\kappa\rho}}  \int \ell^{v+u}(\dt \sigma) \indi_{\brc{c_0\ell^{v+u}(B(\sigma,2^{-k+1})) \notin \ivff{\underline{r}(2^{-k}),\overline{r}(2^{-k})} }} \\
    &\leq \sum_{2^{-k}\in\ivoo{\delta,\kappa\rho}} \sup_{u\in\ivff{0,\tau}} \int \ell^{v+u}(\dt \sigma) \indi_{\brc{c_0\ell^{v+u}(B(\sigma,2^{-k+1})) \notin \ivff{\underline{r}(2^{-k}),\overline{r}(2^{-k})} }}.
  \end{align*}
  In addition,
  \begin{align*}
    \sum_{2^{-k}\in\ivoo{\delta,\kappa\rho}} \ell(\rho,2^{-k})
    = \sum_{2^{-k}\in\ivoo{\delta,\kappa\rho}} g(2^{-k})^{1+\epsilon} \rho^{1/(\gamma-1)}
    &\leq c_{1} \, g(\kappa\rho)^{\epsilon} \rho^{1/(\gamma-1)} \\
    &\leq c_{1} \, g(\rho)^{\epsilon} \rho^{1/(\gamma-1)},
  \end{align*}
  observing that $g(2^{-k}) = k^{-1} (\log 2)^{-1}$.
  Therefore, based Lemmas~\ref{lemma:usp_large_balls2} and \ref{lemma:usp_small_balls2}
  \begin{align*}
    &\Nb_{\kappa\rho}\bktB{ \sup_{u\in\ivff{0,\tau}} \ell^{v+u}(\Lambda(u+v,\rho,\delta)) \geq \ell(\rho) \cap \Ai(\rho) } \\
    &\leq c_{2} \sum_{2^{-k}\in\ivoo{\delta,\kappa\rho}} \Nb_{\kappa\rho}\bktbb{ \sup_{u\in\ivff{0,\tau}} \int \ell^{v+u}(\dt \sigma)\indi_{\brc{c_0\ell^{v+u}(B(\sigma,2^{-k+1})) \notin \ivff{\underline{r}(2^{-k}),\overline{r}(2^{-k})} }} \geq \ell(\rho,2^{-k}) \cap \Ai(\rho) } \\
    &\leq c_3\, \sum_{2^{-k}\in\ivoo{\delta,\kappa\rho}} \tau\rho^{-1}  g(2^{-k})^{-\gamma(2+5\epsilon)}
    \leq c_3\, \tau\rho^{-1} g(\delta)^{-1-\gamma(2+5\epsilon)}.
  \end{align*}
\end{proof}

Lemma~\ref{lemma:usp_large_small_balls1} provides a tight estimate of the total mass of exceptional small and large balls of the local time, ensuring that the set $\Ti(u)\setminus\Lambda(u,\rho,\delta)$ is sufficiently large with high probability.
Hence, we may now start presenting the construction of the collection of Hausdorff measures needed to prove the lower bound of the multifractal spectrum. To begin with, let us evaluate the probability of appearance of well-behaving collection of subtrees.

\begin{lemma}  \label{lemma:usp_small_balls3}
  Suppose $\rho>0$, $\tau\in\ivoo{0,\kappa\rho}$ and $\delta\in\ivoo{0,2^{-1/\rho}}$. For any $u\geq 0$, let us define the following collection of subtrees
  \begin{align*}
    \Ubb(u,\rho,\delta) = \brcB{\Ti_\sigma\in\Tbb(u-\delta,\delta) :\,
    & \bk{\ell^\delta}(\Ti_\sigma)\geq \delta^{\tfrac{1}{\gamma-1}} g(\rho)^{\kappa\epsilon} }.
  \end{align*}
  Then, for all $v\in\ivff{2\kappa\rho,\rho/\kappa}$,
  \begin{align*}
    \Nb_{\kappa\rho}\bktB{ \inf_{u\in\ivff{0,\tau}} \#\Ubb(u+v,\rho,\delta) \leq (\rho\delta^{-1})^{\tfrac{1}{\gamma-1}} g(\rho)^{-\kappa\epsilon} \cap \Ai(\rho) }
    \leq c_0\,\tau \exp\pthb{ -\delta^{-\eta} },
  \end{align*}
  where $\eta\in\ivoo{0,1/(\gamma-1)}$ and $c_0$ is independent of $\rho$, $\tau$ and $\delta$.
\end{lemma}
\begin{proof}
  Let us set $v\in\ivff{2\kappa\rho,\rho/\kappa}$, $\nu\in\ivoo{0,\tau}$ and denote by $X(m)$, $m\in\N$ the following random variable
  \[
    X(m) = \#\brcB{\Ti_\sigma\in\Tbb(v_m-\delta,\delta)  :\,\inf_{w\in\ivff{\delta/2,\delta/2+\nu}}\bk{\ell^w}(\Ti_\sigma)\geq \delta^{\tfrac{1}{\gamma-1}} g(\rho)^{\kappa\epsilon} }.
  \]
  where $v_m \eqdef v + m\nu$. For any $u\in\ivff{0,\tau}$, there exists $m\in\N$ such that $u-m\nu\in\ivff{\delta/2,\delta/2+\nu}$. We then observe that $\bk{\ell^{u-v_m}}(\Ti_\sigma)\geq \delta^{1/(\gamma-1)} g(\rho)^{\kappa\epsilon}$ induces that the subtree $\Ti_u$ rooted at level $u-\delta$ such that $\Ti_\sigma\subset \Ti_u$ satisfies $\bk{\ell^{\delta}}(\Ti_u)\geq\bk{\ell^{u-v_m}}(\Ti_\sigma)\geq \delta^{1/(\gamma-1)} g(\rho)^{\kappa\epsilon}$. Therefore, $\#\Ubb(u+v,\rho,\delta) \geq X(m)$ and
  \[
    \brcB{ \inf_{u\in\ivff{0,\tau}} \#\Ubb(u+v,\rho,\delta) \leq (\rho\delta^{-1})^{\tfrac{1}{\gamma-1}} g(\rho)^{-\kappa\epsilon} } \subset \bigcup_{m\leq2\ceil{\tau/\nu}} \brcb{ X(m) \leq (\rho\delta^{-1})^{\tfrac{1}{\gamma-1}} g(\rho)^{-\kappa\epsilon} }.
  \]
  Hence, it remains to obtain an upper bound of the measure of the latter events. For a given $m\leq2\ceil{\tau/\nu}$, we know that
  \begin{align*}
    \Nb_{\kappa\rho}\bktB{ X(m) \leq (\rho\delta^{-1})^{\tfrac{1}{\gamma-1}} g(\rho)^{-\kappa\epsilon} \cap \Ai(\rho) }
    \leq c_0\,\Nb_{v_m-\delta}\bktB{ X(m) \leq (\rho\delta^{-1})^{\tfrac{1}{\gamma-1}} g(\rho)^{-\kappa\epsilon} \cap \Ai(\rho) },
  \end{align*}
  where the constant $c_0$ is independent of $m$.Given $\Gi_{v_m-\delta}$, $X(m)$ is Poisson random variable parametrized by
  \[
    \lambda_m = \bk{\ell^{v_m-\delta}} v(\delta/2) \Nb_{\delta/2}\pthB{ \inf_{\ivff{\delta/2,\delta/2+\tau}} \bk{\ell^w} \geq \delta^{\tfrac{1}{\gamma-1}} g(\rho)^{\kappa\epsilon} }.
  \]
  Lemma~\ref{lemma:local_time_inf_tail1} entails
  \begin{align*}
    \Nb_{\delta/2}\pthB{ \inf_{\ivff{\delta/2,\delta/2+\tau}} \bk{\ell^w} \leq \delta^{\tfrac{1}{\gamma-1}} g(\rho)^{\kappa\epsilon} } \leq c_0 \,g(\rho)^{\kappa\epsilon(\gamma-1)} + c_0\, \exp\pthb{ -c_1 \,v(\nu) g(\rho)^{\kappa\epsilon} \delta^{1/(\gamma-1)} }.
  \end{align*}
  Choosing $\nu=\delta\,g(\rho)^{\kappa\epsilon(\gamma-1)}$, we therefore know there exists $c_2>0$ independent of $m$, $\rho$ and $\delta$ such that $\lambda_m \geq c_2\, \bk{\ell^{v_m-\delta}} v(\delta)$. Classic Chernoff bound \eqref{eq:chernoff_poisson} then entails
  \begin{align*}
    \Nb_{v_m-\delta}\pthcb{ X(m) \leq c_2 \bk{\ell^{v_m-\delta}} v(\delta) / 2 }{\Gi_{v_m-\delta}} \leq \exp\pthb{-c_2\bk{\ell^{v_m-\delta}} v(\delta) / 8}.
  \end{align*}
  On the event $\Ai(\rho)$, $\bk{\ell^{v_m-\delta}} \geq \rho^{1/(\gamma-1)}g(\rho)^{-\epsilon}$. Therefore,
  \begin{align*}
    \Nb_{v_m-\delta}\pthB{ X(m) \leq (\rho\delta^{-1})^{\tfrac{1}{\gamma-1}} g(\rho)^{-\kappa\epsilon} \cap \Ai(\rho) } \leq \exp\pthB{- (\rho\delta^{-1})^{\tfrac{1}{\gamma-1}} g(\rho)^{-\kappa\epsilon} }.
  \end{align*}
  Finally,
  \begin{align*}
    \Nb_{\kappa\rho}\pthbb{ \bigcup_{m\leq2\ceil{\tau/\nu}} \brcB{ X(m) \leq (\rho\delta^{-1})^{\tfrac{1}{\gamma-1}} g(\rho)^{-\kappa\epsilon} } }
    &\leq c_3\,\tau \delta^{-1}\,g(\rho)^{-\kappa\epsilon(\gamma-1)} \exp\pthB{-(\rho\delta^{-1})^{\tfrac{1}{\gamma-1}} g(\rho)^{-\kappa\epsilon} } \\
    &\leq c_3\,\tau \exp\pthb{ -\delta^{-\eta} },
  \end{align*}
  for some $\eta < 1/(\gamma-1)$, as we recall that $\delta\leq 2^{-1/\rho}$.
\end{proof}

Combining Lemmas~\ref{lemma:usp_large_small_balls1} and \ref{lemma:usp_small_balls3}, we may define a collection of subtrees $\Vbb(u,\rho,\delta)$ designed to be properly ``balanced'':
\begin{align*}
  \Vbb(u,\rho,\delta) = \brcb{\Ti_\sigma\in\Ubb(u,\rho,\delta) : \Ti_\sigma(\delta) \cap \Lambda(u,\rho,\delta) = \vset }.
\end{align*}
Note that in the previous expression, owing to the definition of the set $\Lambda(u,\rho,\delta)$, either $\Ti_\sigma(\delta) \cap \Lambda(u,\rho,\delta) = \vset$ or $\Ti_\sigma(\delta) \subset \Lambda(u,\rho,\delta)$. Then, let us  introduce the event
\begin{align*}
  \Li(\rho,\delta) = \brcB{\Ti : \inf_{u\in\ivff{3\kappa\rho,\rho/\kappa}} \Vbb(u,\rho,\delta) \geq (\rho\delta^{-1})^{\tfrac{1}{\gamma-1}} }.
\end{align*}
The estimates presented in Lemmas~\ref{lemma:usp_large_small_balls1} and \ref{lemma:usp_small_balls3} allow us in the following lemma to obtain a tight bound on the measure of the latter.
\begin{lemma}  \label{lemma:usp_balanced_balls}
  Suppose $\rho>0$ and $\delta\in\ivoo{0,2^{-1/\rho}}$.
  Then,
  \begin{align*}
    \Nb_{\kappa\rho}\pthcb{ \Li(\rho,\delta)^c }{ \Gi_{\kappa\rho} }
    \leq c_1\exp\pthB{ -c_0\,\bk{\ell^{\kappa\rho}} \rho^{-\tfrac{1}{\gamma-1}} g(\rho)^{\epsilon\gamma} + 8 h(\delta)^{-1} },
  \end{align*}
  where the constants $c_0$ and $c_1$ are independent of $\rho$ and $\delta$.
\end{lemma}
\begin{proof}
  We aim to combine the results of Lemmas~\ref{lemma:usp_large_small_balls1} and \ref{lemma:usp_small_balls3} to obtain our result. Nevertheless, one may note that we can directly not apply the former with $\tau\asymp\rho$, as the bound previously presented do not converge in this case. Therefore, we need to apply a slightly more complex strategy to find the proper estimate.

  Let us divide the interval $\ivff{3\kappa\rho,\rho/\kappa}$ into successive and disjoint subintervals of size $\tau>0$, where the value of the latter will be specified at the end of the proof. For every $m\in\N$, let us define $v_m = 2\kappa\rho + m\tau$ and the random variable $N(m)$:
  \begin{align*}
    N(m) = \#\brcb{\Ti_\rho\in\Tbb(\kappa\rho,0) \text{ s.t. } &\sup_{u\in\ivff{0,\tau}} \ell^{v_m+u}\pthb{ \Lambda(u+v_m,\rho,\delta) }(\Ti_\rho) \leq \ell(\rho) \\
    \text{ and } &\inf_{u\in\ivff{0,\tau}} \#\Ubb(u+v_m,\rho,\delta,\Ti_\rho) \geq (\rho\delta^{-1})^{\tfrac{1}{\gamma-1}} g(\rho)^{-\kappa\epsilon} \\
    \text{ and } &\Ti_\rho\in \Ai(\rho) },
  \end{align*}
  using notations respectively introduced in Lemmas~\ref{lemma:usp_large_small_balls1} and \ref{lemma:usp_small_balls3}. If $N(m)\geq 1$, owing to the definition of the latter, the contribution to the local time at level $u+v$ of subtrees in $\Ubb(u+v,\rho,\delta)$ is large than $\rho^{1/(\gamma-1)}$. Furthermore and as previously outlined, the tree geometry induces that for any $\Ti_\sigma\in\Ubb(u+v,\rho,\delta)$, either $\Ti_\sigma(\delta) \subset \Lambda(u+v,\rho,\delta)$ or $\Ti_\sigma(\delta) \cap \Lambda(u+v,\rho,\delta)=\vset$. Therefore, when $N(m)\geq 1$, we must have $\Vbb(u+v,\rho,\delta) \geq (\rho\delta^{-1})^{\tfrac{1}{\gamma-1} }$, otherwise the contribution of $\Ubb(u+v,\rho,\delta)\setminus\Vbb(u+v,\rho,\delta)$ to the local time contradicts the assumption $\ell^{v_m+u}\pthb{ \Lambda(u+v_m,\rho,\delta) }(\Ti_\rho) \leq g(\rho)^\epsilon \rho^{1/(\gamma-1)}$. Hence,
  \begin{align*}
    \Li(\rho,\delta)^c \subset \bigcup_{m\leq\ceil{\rho/\kappa\tau}} \brcb{ N(m)  = 0 }.
  \end{align*}
  We need to obtain an upper bound of $\Nb_{\kappa\rho}\pthb{N(m) = 0 \cap \bk{\ell^{\kappa\rho}} \geq \rho^{1/(\gamma-1)} g(\rho)^{-\alpha_0} }$. As previously, we may observe that given $\Gi_{\kappa\rho}$, $N(m)$ is a Poisson random variable parametrised by
  \begin{align*}
    \lambda_m = \bk{\ell^{\kappa\rho}}(\Ti) \,v(\kappa\rho) \,\Nb_{\kappa\rho}\bktB{ &\sup_{u\in\ivff{0,\tau}} \ell^{v_m+u} \pthb{ \Lambda(u+v_m,\rho,\delta) } \leq \ell(\rho) \cap \\
    &\inf_{u\in\ivff{0,\tau}}  \#\Ubb(u+v,\rho,\delta) \geq (\rho\delta^{-1})^{\tfrac{1}{\gamma-1}} g(\rho)^{-\kappa\epsilon} \cap \Ai(\rho) }.
  \end{align*}
  Then, Lemmas~\ref{lemma:usp_large_small_balls1} and \ref{lemma:usp_small_balls3} entail
  \begin{align*}
    \lambda_m
    &\geq \bk{\ell^{\kappa\rho}}(\Ti) \,v(\kappa\rho) \pthb{ \Nb_{\kappa\rho}\pthb{ \Ai(\rho) }  - c_1\,\tau\rho^{-1} g(\delta)^{-1-\gamma(2+5\epsilon)}  -  c_1\,\tau \exp\pth{ -\delta^{-\eta} } }.
  \end{align*}
  In addition, owing to Lemma~\ref{lemma:local_time_sup_tails}, $\Nb_{\kappa\rho}\pthb{ \Ai(\rho) } \asymp g(\rho)^{\epsilon\gamma}$.
  Hence, by choosing $\tau = \rho g(\delta)^{1+\gamma(2+7\epsilon)}$, the last two terms are negligible in front of $\Nb_{\kappa\rho}\pthb{ \Ai(\rho) }$, as we recall that $\delta\leq 2^{-1/\rho}$, and there exists a positive constant $c_2$ independent of $m$, $\rho$ and $\delta$ such that
  \begin{align*}
    \Nb_{\kappa\rho}\pthcb{N(m) = 0 }{\Gi_{\kappa\rho}} \leq \exp\pthb{ -c_2 \bk{\ell^{\kappa\rho}} \rho^{-1/(\gamma-1)} g(\rho)^{\epsilon\gamma} },
  \end{align*}
  Finally, this last inequality entails
  \begin{align*}
    \Nb_{\kappa\rho}\pthcb{ \Li(\rho,\delta)^c }{\Gi_{\kappa\rho}}
    &\leq c_3\exp\pthb{ -c_2 \bk{\ell^{\kappa\rho}} \rho^{-1/(\gamma-1)} g(\rho)^{\epsilon\gamma} - \log(\tau) } \\
    &\leq c_3\exp\pthb{ -c_2 \bk{\ell^{\kappa\rho}} \rho^{-1/(\gamma-1)} g(\rho)^{\epsilon\gamma} + 8 h(\delta)^{-1} },
  \end{align*}
  recalling that $\tau = \rho g(\delta)^{1+\gamma(2+7\epsilon)}$ and $\delta\in\ivoo{0,2^{-1/\rho}}$.
\end{proof}

Lemma~\ref{lemma:usp_balanced_balls} presents how a properly balanced collection of subtrees $\Vbb(u,\rho,\delta)$ can uniformly be constructed with high probability. The previous bound would be enough to construct by induction a proper collection of Hausdorff measure on every level $u$ to prove the Hausdorff~\eqref{eq:cor_dimH_images} and packing~\eqref{eq:cor_dimP_images} dimensions of level sets $\Ti(u)$.

Nevertheless, as we aim to obtain in addition the full multifractal spectrum of the local time in Theorem~\ref{th:levy_tree_upper_spectrum1}, we rely on Lemma~\ref{lemma:usp_balanced_balls} and the collections $\Vbb(u,\rho,\delta)$ to construct ``well-behaving'' configurations with large balls of the local time of order $\delta^h$.

For that purpose, we set in the rest of this section a real number $\alpha\in\R$, the latter being more specifically defined later.

\begin{lemma}  \label{lemma:usp_large_balls_spectrum1}
  Suppose $\rho>0$, $\delta\in\ivoo{0,2^{-1/\rho}}$ and $h\in\ivffb{\tfrac{1}{\gamma},\tfrac{1}{\gamma-1}}$. For any $v\in\ivff{\kappa\rho,\rho/\kappa}$, let us denote by $\Vbb(v,\rho,\delta,h)\subset\Tbb(v-\delta,\delta)$ the collection of embedded subtrees satisfying the following properties: for every $\Ti_\sigma\in\Vbb(v,\rho,\delta,h)$,
  \begin{enumerate}[(i)]
    \item $\Ti_\sigma\in\Tbb(v-\delta,\delta)$ and
    \[
      \inf_{u\in\ivff{\kappa\delta,\delta/\kappa}}\bk{\ell^u}(\Ti_\sigma) \geq \delta^h g(\delta)^{-\alpha};
    \]
    \item for every $\sigma'\in \Ti_\sigma(\delta)$ and any $r\in\ivff{\delta,\kappa\rho}$
    \[
      \#\brcb{ \Ti'\in\Vbb(v,\rho,\delta,h) : \Ti' \cap B(\sigma',2r)\neq\vset } \leq 1 + \pthb{r\vartheta^{-1}}^{\tfrac{1}{\gamma-1}} g(r)^{-\beta},
    \]
    where $\beta = (1+5\epsilon)/(\gamma-1)$ and using the notation $\vartheta \eqdef \delta^{(\gamma-1)(\gamma h-1)}g(\delta)^{-\alpha\gamma(\gamma-1)}$.
  \end{enumerate}
  Then, for any $v\in\ivff{\rho,2\rho}$
  \begin{align*}
    \Nb_{\kappa\rho}\bktB{ \#\Vbb(v,\rho,\delta,h) \leq \pthb{\rho\vartheta^{-1}}^{\tfrac{1}{\gamma-1}} g(\rho)^{2\epsilon} \cap \Li(\rho,\vartheta) }
    \leq c_0\exp\brcB{- \pth{\rho\vartheta^{-1}}^{1/(\gamma-1)} g(\rho)^{2\epsilon} },
  \end{align*}
  where the constant $c_0>0$ is independent of $v$, $\rho$, $\delta$ and $h\\in\Hi$.
\end{lemma}
\begin{proof}
  Let us begin by describing more precisely our construction of the collection $\Vbb(v,\rho,\delta,h)$. Note that $\vartheta = \delta^{(\gamma-1)(\gamma h-1)}g(\delta)^{-\alpha\gamma(\gamma-1)} > \delta$ and consider the level $w=v-\vartheta$. To define $\Vbb(v,\rho,\delta,h)$, we are interested in a specific configuration of subtrees rooted at level $w$ and belonging to:
  \[
    \Bi(\delta,h) = \brcB{\Ti : \exists\Ti_\sigma\in\Tbb(\vartheta-\delta,\delta,\Ti) : \,\inf_{u\in\ivff{\kappa\delta,\delta/\kappa}}\bk{\ell^u}(\Ti_\sigma) \geq \delta^h g(\delta)^{-\alpha} }.
  \]
  Using the branching property of Lévy trees, we know that given $\Gi_{\vartheta-\delta}$, the number of subtrees $\Ti_\sigma$ with the previous properties is a Poisson random variable parametrized by $\bk{\ell^{\vartheta-\delta}}\Nb\pthb{\inf_{u\in\ivff{\kappa\delta,\delta/\kappa}}\bk{\ell^u} \geq \delta^h g(\delta)^{-\alpha} }$. Hence,
  \begin{align*}
    \Nb_{\vartheta-\delta}\pthb{ \Bi(\delta,h) } = \Nb_{\vartheta-\delta}\pthB{ 1-\exp\pthB{-\bk{\ell^{\vartheta-\delta}}\Nb\bktB{ \inf_{u\in\ivff{\kappa\delta,\delta/\kappa}}\bk{\ell^u}\geq \delta^h g(\delta)^{-\alpha} }} }.
  \end{align*}
  Recall that owing to Lemma~\ref{lemma:local_time_inf_tail0}, $\Nb\pthb{ \inf_{u\in\ivff{\kappa\delta,\delta/\kappa}}\bk{\ell^u} \geq \delta^h g(\delta)^{-\alpha} } \asymp_{\delta\rightarrow 0} \delta^{1-\gamma h} g(\delta)^{\alpha\gamma}$ and observe in addition that $\vartheta-\delta\asymp\vartheta$.
  The Laplace transform of the local time then entails
  \begin{align*}
    \Nb_{\vartheta-\delta}\pthb{ \Bi(\delta,h) }
    &\geq \Nb_{\vartheta-\delta}\pthb{ 1-\exp\pthb{-c_0 \,\delta^{1-\gamma h}g(\delta)^{\alpha\gamma} \bk{\ell^{\vartheta-\delta}}} } \\
    &\geq \pthb{ 1 + c_1\vartheta^{-1}\delta^{(\gamma-1)(\gamma h -1)}g(\delta)^{-\alpha\gamma(\gamma-1)} }^{-1/(\gamma-1)} \\
    &\geq c_2,
  \end{align*}
  recalling that $\vartheta = \delta^{(\gamma-1)(\gamma h-1)}g(\delta)^{-\alpha\gamma(\gamma-1)}$. Therefore, there exists $c_3>0$ such that $\Nb\pthb{ \Bi(\delta,h) } \geq c_3\, \vartheta^{-1/(\gamma-1)}$, where the constant $c_3>0$ (as well as $c_2$) is independent of $\delta$ and $h$.

  We now aim to count the number of these configurations rooted at level $w=v-\vartheta$. More precisely, we are interested in the elements in the set $\Vbb(w,\rho,\vartheta)$ who give birth to such a configuration. We denote by $N(v,\rho,\delta)$ the number of such elements. Given $\Gi_w$, $N(v,\rho,\delta)$ is the sum of $\#\Vbb(w,\rho,\vartheta)$ independent Bernoulli random variable whose parameter depends on local time $\bk{\ell^\vartheta}(\Ti_\vartheta)$, $\Ti_\vartheta\in\Vbb(w,\rho,\vartheta)$. Nevertheless, owing to definition of $\Vbb(w,\rho,\vartheta)$, it is clearly lower bounded by a Binomial distribution parametrised by $\#\Vbb(w,\rho,\vartheta)$ and $\lambda(\rho,\delta) \geq \vartheta^{\tfrac{1}{\gamma-1}} g(\rho)^\epsilon \,\Nb\pthb{ \Bi(\delta,h) } \geq c_3\, g(\rho)^\epsilon$.
  Therefore, Chernoff bound entails
  \begin{align*}
    \Nb_{w}\pthcb{ N(v) \leq \lambda(\rho,\delta)\# \Vbb(v,\rho,\vartheta) / 2 }{\Gi_w} \leq \exp\pthb{-\lambda(\rho,\delta)\# \Vbb(v,\rho,\vartheta) / 8}.
  \end{align*}
  On the event $\Li(\rho,\vartheta)$, $\#\Vbb(w,\rho,\vartheta) \geq \pth{\rho\vartheta^{-1}}^{1/(\gamma-1)}$. Hence,
  \begin{align*}
    \Nb_{w}\bktB{ N(v) \leq \pthb{\rho\vartheta^{-1}}^{\tfrac{1}{\gamma-1}} g(\rho)^{2\epsilon} \cap \Li(\rho,\vartheta) }
    &\leq \exp\pthb{-\pth{\rho\vartheta^{-1}}^{1/(\gamma-1)} g(\rho)^{2\epsilon} },
  \end{align*}
  for any $\rho$ sufficiently small.\vsp

  As we have obtained the expected bound on the size of $\Vbb(v,\rho,\delta,h)$, it remains to prove that the collection of subtrees $\Vbb(v,\rho,\delta,h)$ constructed in this way satisfies the conditions of the lemma. The first one is easily verified. Concerning the second one, let us set $r\in\ivff{\delta,\kappa\rho}$ and $\sigma'\in\Ti_\sigma(\delta)$ for a fixed $\Ti_\sigma\in\Vbb(v,\rho,\delta,h)$.
  Then, let us distinguish two cases. If $r\in\ivfo{\delta,\vartheta}$, then our construction of $\Vbb(v,\rho,\delta,h)$ ensures that $\#\brcb{ \Ti'\in\Vbb(v,\rho,\delta,h) : \Ti' \cap B(\sigma',2r)\neq\vset } = 1$.

  Hence, we may suppose that $r\in\ivfo{\vartheta,\kappa\rho}$. Let $\sigma_0$ designates the ancestor of $\Ti_\sigma$ at level $w$. Clearly, $\sigma_0\in\Ti(w)$ and owing to the definition of $\Vbb(w,\rho,\vartheta)$, we know that $\ell^{w}(B(\sigma_0,2r)) \leq \overline{r}(r)$, where $\overline{r}(r) = \pthb{r / g(r)^{1+4\epsilon}}^{1/(\gamma-1)}$. Then, the construction of $\Vbb(v,\rho,\delta,h)$ and the tree geometry induce
  \[
    \#\brcb{ \Ti'\in\Vbb(v,\rho,\delta,h) : \Ti'\subset B(\sigma',2r) } \leq \#\brcb{ \widehat\Ti \in\Vbb(w,\rho,\vartheta) : \widehat\Ti \subset B(\sigma_0,2r) }.
  \]
  Moreover, since every subtree $\widehat\Ti\in\Vbb(w,\rho,\vartheta)$ satisfies $\bk{\ell^\vartheta}(\widehat\Ti) \geq \vartheta^{1/(\gamma-1)} g(\rho)^\epsilon$, it follows that
  \begin{align*}
    \#\brcb{ \widehat\Ti \in\Vbb(w,\rho,\vartheta) : \widehat\Ti \subset B(\sigma_0,2r) }\, \vartheta^{1/(\gamma-1)} g(\rho)^\epsilon
    &\leq \sum_{\widehat\Ti \subset B(\sigma_0,2r)} \bk{\ell^\vartheta}(\widehat\Ti) \\
    &\leq \ell^{w}(B(\sigma_0,2r))
    \leq \pthb{r / g(r)^{1+4\epsilon}}^{1/(\gamma-1)}.
  \end{align*}
  The last inequality clearly entails
  \begin{align*}
    \#\brcb{ \Ti'\in\Vbb(v,\rho,\delta,h) : \Ti'\subset B(\sigma',2r) } \leq
    \pthb{r\vartheta^{-1}}^{\tfrac{1}{\gamma-1}} g(r)^{-\beta},
  \end{align*}
  where $\beta \eqdef (1+5\epsilon)/(\gamma-1)$.
\end{proof}

Let us now slightly extend the previous lemma to a collection of levels inside the interval $\ivfo{\rho,2\rho}$. For that purpose, we define the event
\begin{align}  \label{eq:event_ppties_collections}
  \Bi( \rho,\delta,h ) =  \Li(\rho,\vartheta) \cap \bigcap_{ k\in\N:k\delta\in\ivfo{\rho,2\rho} } \brcB{ \Ti : \#\Vbb(k\delta,\rho,\delta,h) \geq \pthb{\rho\vartheta^{-1}}^{\tfrac{1}{\gamma-1}} g(\rho)^{2\epsilon} }.
\end{align}
where we remind that $\vartheta \eqdef \delta^{(\gamma-1)(\gamma h-1)}g(\delta)^{-\alpha\gamma(\gamma-1)}$ (for the sake of readability, we omit to recall the dependency in $h$). Note that we will always assume that $\alpha\in\R$ is such that $\delta < \vartheta < \rho$, which in particular implies $\alpha>0$ if $h=\tfrac{1}{\gamma-1}$ and  $\alpha<0$ if $h=\tfrac{1}{\gamma}$.

We present in the following lemma a bound on the measure of the event $\Bi( \rho,\delta,h )$.
\begin{lemma}  \label{lemma:usp_large_balls_spectrum2}
  Suppose $\rho>0$, $\delta\in\ivoo{0,2^{-1/\rho}}$ and $h\in\ivffb{\tfrac{1}{\gamma},\tfrac{1}{\gamma-1}}$. Then,
  \begin{align*}
    \Nb_{\kappa\rho}\pthb{ \Bi( \rho,\delta,h )^c \cap \bk{\ell^{\kappa\rho}} \geq \rho^{1/(\gamma-1)} g(\rho)^{-\alpha} }
    &\leq c_1\exp\pthb{ -c_0 g(\rho)^{-\alpha+\gamma\epsilon} + 8 h(\delta)^{-1} } \\
    &+ c_1\exp\pthb{- \delta^{1-\gamma h}g(\delta)^{\alpha\gamma+1+\epsilon} + g(\delta)^{-1} }.
  \end{align*}
  where $c_0$ and $c_1$ are independent of $\rho$ and $\delta$.
\end{lemma}
\begin{proof}
  Defining
  \begin{align*}
    \Bi_\star( \rho,\delta,h ) =  \bigcap_{ k\in\N:k\delta\in\ivff{\rho,2\rho} } \brcB{ \Ti : \#\Vbb(k\delta,\rho,\delta,h) \geq \pthb{\rho\vartheta^{-1}}^{\tfrac{1}{\gamma-1}} g(\rho)^{2\epsilon} },
  \end{align*}
  we then simply first observe that
  \begin{align*}
    \Bi( \rho,\delta,h )^c \cap \brcb{ \bk{\ell^{\kappa\rho}} \geq \rho^{1/(\gamma-1)} g(\rho)^{-\alpha} }
    \subset &\,\Bi_\star( \rho,\delta,h )^c \cap \Li(\rho,\vartheta) \\
    \cup & \,\Li(\rho,\vartheta)^c \cap \brcb{ \bk{\ell^{\kappa\rho}} \geq \rho^{1/(\gamma-1)} g(\rho)^{-\alpha} }.
  \end{align*}
  Lemma~\ref{lemma:usp_balanced_balls} provides a bound on the measure of the second term:
  \begin{align*}
    \Nb_{\kappa\rho}\pthb{ \Li(\rho,\vartheta)^c \cap  \bk{\ell^{\kappa\rho}} \geq \rho^{1/(\gamma-1)} g(\rho)^{-\alpha} }
    \leq c_1\exp\pthb{ -c_0 g(\rho)^{-\alpha+\gamma\epsilon} + 8 h(\delta)^{-1} },
  \end{align*}
  as $\vartheta \geq \delta$.
  Using Lemma~\ref{lemma:usp_large_balls_spectrum1}, the first one is upper bounded by
  \begin{align*}
    \Nb_{\kappa\rho}\pthb{ \Bi_\star( \rho,\delta,h )^c \cap \Li(\rho,\vartheta) }
    &\leq \sum_{ k\in\N:k\delta\in\ivff{\rho,2\rho} } \Nb_{\kappa\rho}\bktB{ \#\Vbb(k\delta,\rho,\delta,h) \leq  \pthb{\rho\vartheta^{-1}}^{\tfrac{1}{\gamma-1}} g(\rho)^{2\epsilon} \cap \Li(\rho,\vartheta)} \\
    &\leq c_2\,\rho\delta^{-1} \exp\pthb{- \delta^{1-\gamma h}g(\delta)^{\alpha\gamma}\rho^{1/(\gamma-1)} g(\rho)^{2\epsilon} } \\
    &\leq \exp\pthb{- \delta^{1-\gamma h}g(\delta)^{\alpha\gamma+1+\epsilon} + g(\delta)^{-1} }.
  \end{align*}
  The two previous bounds conclude the proof of the lemma.
\end{proof}

Based on the estimate obtained in the Lemma~\ref{lemma:usp_large_balls_spectrum2}, we are now able to describe and study more precisely the construction of Hausdorff measures on the sets $E_\ell(h,\Ti)\cap \Ti(a)$. We first focus on the case $h\in\ivofb{\tfrac{1}{\gamma}, \tfrac{1}{\gamma-1}}$ and set $\alpha>1+\gamma\epsilon$.
From now on, we will consider a sequence $(\rho_n)_{n\in\N}$ such that
\[
  \rho_n = 2^{-\rho_{n-1}^{-1}}\quad\text{and}\quad \rho_0=1.
\]
The latter clearly converges exponentially fast to zero and is such that $\rho_{n-1} = \pthb{\log_2 1/\rho_n}^{-1}$.

As we wish to obtain a uniform result on the lower bound of the multifractal spectrum, we need to construct simultaneously a collection of Hausdorff measures indexed by $h\in\ivofb{\tfrac{1}{\gamma}, \tfrac{1}{\gamma-1}}$. For that purpose, we will consider a sub-interval $\Hi\eqdef\ivff{h_1,h_2} \subset \ivofb{\tfrac{1}{\gamma}, \tfrac{1}{\gamma-1}}$. Then, similarly to the dyadic decomposition of real numbers, the interval $\Hi$ can be represented by a binary tree $T_\Hi$ whose rays (i.e. infinite branches) correspond to real numbers $h\in\Hi$. As a consequence, for any $h\in \Hi$, we may write
\[
  h = \eps_0\eps_1\cdots \eps_n\cdots\quad \text{where }\eps_k\in\brc{0,1},
\]
which is equivalent to the representation of nodes in a tree using the classic lexicographical order. For any $n\in\N$, $\Hi_n$ will denote the set of dyadics of order $n$, i.e.
\[
  \Hi_n= \brcb{ \eps_0\eps_1\cdots \eps_n : \eps_k\in\brc{0,1}}.
\]
Furthermore, $p_n:\Hi\mapsto \Hi_n$ designates the classic projection defined by $p_n(h) = \eps_0\eps_1\cdots \eps_n$.

Let us begin with the main lemma necessary to our construction by induction.
\begin{lemma}  \label{lemma:usp_construction_spectrum1}
  Suppose $b>0$, $h\in\Hi$ and $n\in\N$. We denote by $\Ni(n,h)$ the following random variable
  \begin{align*}
    \Ni(n,h) = \#\brcB{\Ti_\sigma\in\Tbb(j\rho_n,0) \text{ where } j\rho_{n}\in\ivoo{0,b} \text{ and } &\Ti_\sigma\in \Bi( \rho_n,\rho_{n+1},h )^c \\
    \text{and } &\bk{\ell^{\kappa\rho_n}}(\Ti_\sigma) \geq \rho_n^{1/(\gamma-1)} g(\rho_n)^{-\alpha} }.
  \end{align*}
  Then,
  \begin{align*}
    \Nb\pthb{ \Ni(n,h) \geq 1 } \leq b\,c_0 \exp\pthb{ -c_1 g(\rho_n)^{-\alpha+\epsilon} },
  \end{align*}
  where the constants $c_0$ and $c_1$ are independent of $n$, $h$ and $b$.
\end{lemma}
\begin{proof}
  The random variable $\Ni(n,h)$ can be rewritten as a sum
  $\Ni(n,h) = \sum_{j\in\N:0<j\rho_{n}<b}\Ni(j,n,h)$, where $\Ni(j,n,h)$ designates the number of such configurations rooted at level $j\rho_{n}$. Then, we easily observe that
  \[
    \brcb{ \Ni(n,h) \geq 1 } \subset \bigcup_{j\in\N:0<j\rho_{n}<b} \brcb{ \Ni(j,n,h) \geq 1 }.
  \]
  Let us fixed $j\in\N$ such that $0<j\rho_{n}<b$. We know that $\Nb\pthb{ \Ni(j,n,h) \geq 1 } = v(j\rho_{n}) \Nb_{j\rho_{n}}\pthb{ \Ni(j,n,h) \geq 1 }$. Furthermore, given the $\sigma$-field $\Gi_{j\rho_{n}}$, $\Ni(j,n,h)$ is Poisson random variable parametrised by $\lambda_{j,n} = \bk{\ell^{j\rho_{n}}} \Nb\pthb{ \Bi( \rho_n,\rho_{n+1},h )^c \cap \bk{\ell^{\kappa\rho_n}} \geq \rho_n^{1/(\gamma-1)} g(\rho_n)^{-\alpha} }$. Hence, Lemma~\ref{lemma:usp_large_balls_spectrum2} entails
  \begin{align*}
    \Nb_{j\rho_{n}}\pthcb{ \Ni(j,n,h) \geq 1 }{\Gi_{j\rho_{n}}}
    &\leq \bk{\ell^{j\rho_{n}}} \Nb\pthb{ \Bi( \rho_n,\rho_{n+1},h )^c \cap \bk{\ell^{\kappa\rho_n}} \geq \rho_n^{1/(\gamma-1)} g(\rho_n)^{-\alpha} } \\
    &\leq c_1\, \bk{\ell^{j\rho_{n}}} v\pthb{\rho_n }\exp\pthb{ -c_0 g(\rho_n)^{-\alpha+\epsilon} + 8 h(\rho_{n+1})^{-1} }.
  \end{align*}
  Note that $h(\rho_{n+1})^{-1} \leq c\,g(\rho_n)^{-1}$, $v(\rho_n)=v(1)\exp\pthb{+g(\rho_n)^{-1}/(\gamma-1)}$ and $\Nb_{j\rho_{n}}\pthb{ \bk{\ell^{j\rho_{n}}} } = v(j\rho_{n})^{-1}$. Therefore, as $\alpha > 1+\epsilon$,
  \begin{align*}
    \Nb\pthb{ \Ni(j,n,h) \geq 1 } \leq c_2\,\exp\pthb{ -c_3 g(\rho_n)^{-\alpha+\epsilon} },
  \end{align*}
  where the constants $c_2$ and $c_3$ are independent of $j$, $n$ and $h$. Summing over $j\in\N$, we obtain
  \begin{align*}
    \Nb\pthb{ \Ni(n,h) \geq 1 } \leq 2c_2b \,\rho_{n}^{-1} \exp\pthb{ -c_3 g(\rho_n)^{-\alpha+\epsilon} },
  \end{align*}
  which concludes the proof of the lemma.
\end{proof}

We may now present our key lemma for the existence of the Hausdorff measures.
\begin{lemma}  \label{lemma:usp_construction_spectrum2}
  Suppose $b>0$. $\Nb(\dt \Ti)$-a.e. there exists $n_0(\Ti)$ such that for all $n\geq n_0(\Ti)$ and for any subtree $\Ti_\sigma$ rooted at level $j\rho_{n}\in\ivoo{0,b}$ satisfying $\bk{\ell^{\kappa\rho_n}}(\Ti_\sigma) \geq \rho_n^{1/(\gamma-1)} g(\rho_n)^{-\alpha}$, we have
  \begin{align*}
    \forall h_{n+1}\in \Hi_{n+1};\quad \Ti_\sigma\in \Bi( \rho_n,\rho_{n+1},h_{n+1} ).
  \end{align*}
\end{lemma}
\begin{proof}
  Based on the estimate obtained in Lemma~\ref{lemma:usp_construction_spectrum1},
  \begin{align*}
    \Nb\pthbb{ \bigcup_{h_{n+1}\in \Hi_{n+1}} \brcb{ \Ni(n,h_{n+1}) \geq 1 } } \leq b\,c_0 2^{n+1}\exp\pthb{ -c_1 g(\rho_n)^{-\alpha+\epsilon} }.
  \end{align*}
  Hence,
  \begin{align*}
    \sum_{n\in\N} \,\Nb\pthbb{ \bigcup_{h_{n+1}\in \Hi_{n+1}} \brcb{ \Ni(n,h_{n+1}) \geq 1 } } < \infty,
  \end{align*}
  and Borel--Cantelli lemma entails the result.
\end{proof}
Finally, let us also prove that we are able to initialise properly our construction by induction.
\begin{lemma}  \label{lemma:usp_construction_spectrum3}
  Suppose $b>0$. Then, $\Nb(\dt \Ti)$-a.e. there exists $n_0(\Ti)$ such that for all $n\geq n_0(\Ti)$ and for every $j\in\N$ such that $j\rho_{n}\in\ivoo{0,b}$, we have
  \begin{align*}
    \bk{\ell^{j\rho_{n}}} \leq \rho_n^\epsilon\quad\text{or}\quad \exists \Ti_\sigma\in\Tbb(j\rho_n,0) \text{ s.t. }\bk{\ell^{\kappa\rho_n}}(\Ti_\sigma) \geq \rho_n^{1/(\gamma-1)} g(\rho_n)^{-\alpha},
  \end{align*}
  under the assumption $\tfrac{1}{\gamma-1} > 1+\epsilon$ on $\epsilon$.
\end{lemma}
\begin{proof}
  Let us denote by $N(j,n)$
  \[
    N(j,n) = \#\brcb{ \Ti_\sigma\in\Tbb(j\rho_n,0)  : \bk{\ell^{\kappa\rho_n}}(\Ti_\sigma) \geq \rho_n^{1/(\gamma-1)} g(\rho_n)^{-\alpha} }.
  \]
  Given $\Gi_{j\rho_{n}}$, the former is a Poisson random variable parametrised by $\lambda_{j,n} = \bk{\ell^{j\rho_{n}}} \Nb\pthb{ \bk{\ell^{\kappa\rho_n}}(\Ti_\sigma) \geq \rho_n^{1/(\gamma-1)} g(\rho_n)^{-\alpha} }$. Therefore, Lemma~\ref{lemma:local_time_tail} entails
  \begin{align*}
    \Nb\pthb{ N(j,n) = 0 \cap \bk{\ell^{j\rho_{n}}} \geq \rho_n^\epsilon }
    &\leq v(j\rho_n) \exp\pthb{-c_0\rho_n^{-1/(\gamma-1)+\epsilon} g(\rho_n)^{\alpha\gamma} },
  \end{align*}
  inducing that
  \begin{align*}
    \Nb\pthbb{ \bigcup_{j:0<j\rho_n<b} \brcb{ N(j,n) = 0 } \cap \brcb{ \bk{\ell^{j\delta_{n}}} \geq \rho_n^\epsilon } }
    &\leq \sum_{0<j\rho_n<b} v(j\rho_n) \exp\pthb{-c_0\rho_n^{-1/(\gamma-1)+\epsilon} g(\rho_n)^{\alpha\gamma} } \\
    &\leq c_1\,\rho_n^{-1/(\gamma-1)}\exp\pthb{-c_0\rho_n^{-1/(\gamma-1)+\epsilon} g(\rho_n)^{\alpha\gamma} } \\
    &<\infty,
  \end{align*}
  Borel--Cantelli lemma then concludes the proof.
\end{proof}

Provided the previous two lemmas, we may now define by induction a family of collections of nested subtrees $(\Vbb_n(h))_{n\geq n_0, h\in\Hi_n}$. Indeed, up to a modification of $n_0(\Ti)$, the combination of Lemmas~\ref{lemma:usp_construction_spectrum2} and \ref{lemma:usp_construction_spectrum3} ensure the existence of nonempty collections $\Vbb_{n_0}(j,h)\subset\Tbb((j-1)\rho_{n_0},\rho_{n_0})$ where $j\rho_{n_0}\in\ivoo{\epsilon,h(\Ti)-\epsilon}$, $h\in\Hi_{n_0}$ and
\begin{align*}
  \forall \Ti_\sigma\in\Vbb_{n_0}(j,h);\quad \inf_{u\in\ivff{\kappa\rho_{n_0},\rho_{n_0}/\kappa}}\bk{\ell^u}(\Ti_\sigma) \geq \rho_{n_0}^h \,g(\rho_{n_0})^{-\alpha} \geq \rho_{n_0}^{1/(\gamma-1)} g(\rho_{n_0})^{-\alpha}.
\end{align*}
$\Vbb_{n_0}(h)$ is then simply defined as $\Vbb_{n_0}(h) = \cup_{j\rho_{n_0}\in\ivoo{\epsilon,h(\Ti)-\epsilon}} \Vbb_{n_0}(j,h)$.

Lemma~\ref{lemma:usp_construction_spectrum2} then allows to proceed by induction for $n>n_0$. Suppose $\Vbb_{n-1}(j,h')$, $h'\in\Hi_{n-1}$ has been properly constructed. For every $\Ti_\sigma\in\Vbb_{n-1}(j,h')$, based on the notation introduced in Lemma~\ref{lemma:usp_construction_spectrum1}, let us define $\Vbb_{n}(k,h,\Ti_\sigma)$ as
\begin{align}
  \Vbb_{n}(k,h,\Ti_\sigma) \eqdef \Vbb(k\rho_{n},\rho_{n-1},\rho_{n},h,\Ti_\sigma) \subset\Tbb((k-1)\rho_{n},\rho_{n})
\end{align}
where $h=h'\eps$, $\eps\in\brc{0,1}$ and $k\rho_n\in\ivfo{j\rho_{n-1},(j+1)\rho_{n-1}}$. The previous definition is entirely licit as, according to Lemma, $\Ti_\sigma\in \Bi( \rho_{n-1},\rho_{n},h )$. Then, we may define in addition
\begin{align}
  \Vbb_{n}(k,h) = \bigcup_{\Ti_\sigma\in\Vbb_{n-1}(j,h')}  \Vbb_{n}(k,h,\Ti_\sigma) \quad\text{and}\quad\Vbb_{n}(h)=\bigcup_{k\rho_n\in\ivoo{\epsilon,h(\Ti)-\epsilon}} \Vbb_{n}(k,h).
\end{align}
Note that the previous lemmas ensure us that
\begin{align}
  \forall \Ti_\sigma\in\Vbb_{n}(h);\quad \inf_{u\in\ivff{\kappa\rho_{n},\rho_{n}/\kappa}}\bk{\ell^u}(\Ti_\sigma) \geq \rho_{n}^h \,g(\rho_{n})^{-\alpha} \geq \rho_{n}^{1/(\gamma-1)} g(\rho_{n})^{-\alpha},
\end{align}
hence proving the correctness of our construction by induction. The collections $(\Vbb_n(h))_{n\geq n_0, h\in\Hi_n}$ are the cornerstone to prove the lower bound of the multifractal spectrum.\vsp

Using the collections $(\Vbb_n(h))_{n\geq n_0, h\in\Hi_n}$, we may for now construct a collection of Hausdorff measures sufficient to complete the proof of Theorem~\ref{th:levy_tree_upper_spectrum1}.
\begin{lemma}  \label{lemma:usp_construction_spectrum4}
  Suppose $\Hi\subset\ivofb{\tfrac{1}{\gamma},\tfrac{1}{\gamma-1}}$. Then, $\Nb(\dt \Ti)$-a.e. for every $a\in\ivoo{\epsilon,h(\Ti)-\epsilon}$ and for any $h\in\Hi$, there exists $G(a,h)\subset \Ti(a)$ such that
  \[
    \forall \sigma\in G(a,h);\quad \liminf_{r\rightarrow 0} \frac{ \log \ell^a\pthb{ B(\sigma,r) } }{ \log r } \leq h.
  \]
  Furthermore, there exist a probability measure $\mu_{a,h}$ supported by $G(a,h)$ such that
  \begin{align}  \label{eq:mass_pple_spectrum1}
    \forall \sigma\in G(a,h),\ \forall r>0;\quad \mu_{a,h}\pthb{ B(\sigma,r) } \leq r^{\gamma h-1-\eps(r)} g(r)^{-\beta}.
  \end{align}
  where $\beta>0$ is independent of $a$ and $h$, and $\eps(\cdot)$ is a positive non-decreasing function satisfying $\lim_{\eps\rightarrow 0} \eps(r) = 0$. Finally, there also exists a decreasing sequence $r_n\rightarrow 0$ such that
  \begin{align}  \label{eq:mass_pple_spectrum2}
    \forall \sigma\in G(a,h),\ \forall n\in\N;\quad \mu_{a,h}\pthb{ B(\sigma,r_n) } \leq r_n^{1/(\gamma-1)} g(r_n)^{-\beta}.
  \end{align}
\end{lemma}
\begin{proof}
  Set $h\in\Hi\subset\ivofb{\tfrac{1}{\gamma},\tfrac{1}{\gamma-1}}$, $a\in\ivoo{\epsilon,h(\Ti)-\epsilon}$ and let us begin with the construction of the set $G(a,h)$. For every $n\in\N$, let $h_n\rightarrow h$ be the dyadic approximation of $h$. Then, based on the collections $(\Vbb_n(h_n))_{n\geq n_0}$, define
  \[
    \forall n\geq n_0;\quad G(a,h_n,n) = \bigcup_{\Ti_\sigma\in \Vbb_n(k_n,h_n)} \Ti_\sigma\cap \Ti(a).
  \]
  where for every $n$, $k_n\in\N$ is such that $a\in\ivfo{k_n\rho_n,(k_n+1)\rho_n}$.
  Since the family $(\Vbb_n(h_n))_{n\geq n_0}$ is composed of nested subtrees, $(G(a,h_n,n))_{n\geq n_0}$ is a decreasing sequence for the inclusion, and we may define the limit $G(a,h)\eqdef\cap_{n\geq n_0} G(a,h_n,n)$.\vsp

  Let us now prove that for any $\sigma\in G(a,h)$, the local time has the expected behaviour. Set $\sigma\in G(a,h)$, $r\in\ivoo{0,\rho_{n_0}}$ and $n\in\N$ such that $r\in\ivfo{\rho_{n}, \rho_{n-1}}$. Due to the construction of $G(a,h)$, there exists $\Ti_{\sigma'}\in\Vbb_n(h_n)$ such that $\sigma\in \Ti_{\sigma'}\cap\Ti(a)$. In addition, the tree structure induce that
  \[
    \ell^a\pthb{ B(\sigma,2r) } \geq \ell^{a}\pthb{\Ti_{\sigma'}\cap\Ti(a)} \geq \rho_n^{h_n} g(\rho_n)^{-\alpha},
  \]
  where the last inequality is a consequence of the definition of $\Vbb_n(h_n)$. Therefore, we obtain
  \begin{align*}
    \frac{ \log \ell^a\pthb{ B(\sigma,2r) } }{ \log 2r } \leq h_n \frac{ \log \rho_n }{ \log 2r } -\alpha \frac{ \log g(\rho_n) }{ \log 2r } \leq h_n \frac{ \log \rho_n }{ \log 2\rho_n }.
  \end{align*}
  The last term clearly converges to $h$ as $r\rightarrow 0$, therefore proving the expected property on the set $G(a,h)$.\vsp

  In the last part of the proof, we describe the construction of an Hausdorff measure $\mu_{a,h}$ supported by the set $G(a,h)$, which will then provide the lower bound on the Hausdorff dimension. For that purpose, recall that $G(a,h)=\cap_{n \geq n_0} G(a,h_n,n)$, where the latter sequence is decreasing for the inclusion. Quite naturally, the simplest way to build a measure on such a set is to mimic the classic construction of the mass measure on the Cantor set.

  Let us begin by defining the probability measure $\mu_{n_0}$ on $\Ti(a)$:
  \begin{align}  \label{eq:constr_mass_measure0}
    \forall V\in\Bi(\Ti(a));\quad \mu_{n_0}(V) \eqdef c_0 \sum_{\Ti_{\sigma}\in \Vbb(k_{n_0},h_{n_0})}\, \frac{\ell^a(V\cap\Ti_{\sigma})}{\ell^a(\Ti(a)\cap\Ti_{\sigma})},
  \end{align}
  where the normalising constant $c_0$ is such that $\mu_{n_0}(\Ti(a)) = 1$ and $\Bi(\Ti(a))$ denotes the Borel sets of $\Ti(a)$.

  The sequence of probability measures $(\mu_n)_{n>n_0}$ is then easily defined by induction. Suppose $\mu_n$ has been properly constructed and is supported by $G(a,h,n)$, $\mu_{n+1}$ is then defined as following: for every $\Ti_\sigma\in\Vbb_n(k_n,h_n)$, the mass $\mu_n(\Ti_\sigma\cap\Ti(a))$ is equally distributed among the subtrees $\Ti_{\sigma'}\in\Vbb_{n+1}(k_{n+1},h_{n+1},\Ti_\sigma)$, using the principle described in the construction \eqref{eq:constr_mass_measure0} of $\mu_{n_0}$.

  The sequence of probability measures $(\mu_n)_{n\geq n_0}$ is clearly tight, as $\Ti(a)$ is a compact set, inducing the existence of a limit for extracted sequences. Then, the typical Cantor structure and Portemanteau lemma ensure the uniqueness of the limit $\mu_{a,h}$. In addition, $\mu_{a,h}$ is supported by $G(a,h)$.\vsp

  It remains to prove that this Hausdorff measure $\mu_{a,h}$ satisfies Equations~\eqref{eq:mass_pple_spectrum1} and \eqref{eq:mass_pple_spectrum2}. Let us set $\sigma\in G(a,h)$, $r\in\ivoo{0,\rho_{n_0}}$ and $n\in\N$ such that $r\in\ivof{\rho_{n}, \rho_{n-1}}$. There exists a unique $\Ti_{n-1}\in\Vbb_{n-1}(h_{n-1})$ such that $B(\sigma,2r)\cap\Ti(a)\subset \Ti_{n-1}\cap\Ti(a)$. To obtain a precise upper bound of $\ell^a(B(\sigma,2r))$, we need to estimate the number $N(\sigma,r)$ of elements $\Ti_n\in\Vbb_n(k_n,h_n, \Ti_{n-1})$ such that $B(\sigma,2r)\cap \Ti_n \neq \vset$. More precisely, since $r\in\ivof{\rho_{n}, \rho_{n-1}}$, we either have $\Ti_n \subset B(\sigma,2r)$ or $B(\sigma,2r)\cap \Ti_n = \vset$. Then, Lemma~\ref{lemma:usp_large_balls_spectrum1} entails that
  \begin{align}  \label{eq:bound_number_subtrees}
    N(\sigma,r) \leq 1 + \pthb{r\vartheta_n^{-1}}^{\tfrac{1}{\gamma-1}} g(r)^{-\beta},
  \end{align}
  for a constant $\beta>0$ independent of $n$ and $\sigma$, and recalling the notation introduced in Lemma~\ref{lemma:usp_large_balls_spectrum1}: $\vartheta_n = \rho_n^{(\gamma-1)(\gamma h_n-1)}g(\rho_n)^{-\alpha\gamma(\gamma-1)}$.
  Furthermore, the construction of the measure $\mu_{a,h}$ and the property~\eqref{eq:event_ppties_collections} of the collections $(\Vbb_n(h_n))_{n\geq n_0}$ ensure that
  \begin{align}  \label{eq:bound_mass_subtrees}
    \mu_{a,h}\pthb{ \Ti_n }
    \leq c_0 \prod_{k=n_0}^n \pthb{ \vartheta_k\rho_{k-1}^{-1} }^{\tfrac{1}{\gamma-1}}  g(\rho_{k-1})^{-2\epsilon}
    \leq c_1 \pthb{ \vartheta_n\vartheta_{n-1}\rho_{n-1}^{-1} }^{\tfrac{1}{\gamma-1}} g(\rho_{n-1})^{-\eta},
  \end{align}
  for some constant $\eta>0$ independent of $n$, $a$ and $h\in\Hi$. Note that the previous bound holds thanks to the exponential decrease of the sequence $(\rho_n)_{n\geq 1}$.
  Then, combining the two estimates \eqref{eq:bound_number_subtrees} and \eqref{eq:bound_mass_subtrees}, we obtain
  \begin{align*}
    \mu_{a,h}\pthb{ B(\sigma,2r) }
    &\leq N(\sigma,r)\cdot \mu_{a,h}\pthb{ \Ti_n } \\
    &\leq c_2 \pthb{ \vartheta_n\vartheta_{n-1}\rho_{n-1}^{-1} }^{\tfrac{1}{\gamma-1}} g(\rho_{n-1})^{-\eta} \cdot\pthB{ 1 + \pthb{r\vartheta_n^{-1}}^{\tfrac{1}{\gamma-1}} g(r)^{-\beta} }.
  \end{align*}
  We may now distinguish two cases: if $r\in\ivof{\rho_{n}, \vartheta_{n}}$, $\pthb{r\vartheta_n^{-1}}^{\tfrac{1}{\gamma-1}} \leq 1$ and thus,
  \begin{align*}
    \mu_{a,h}\pthb{ B(\sigma,2r) }
    \leq c_2 \pthb{ \vartheta_n\vartheta_{n-1}\rho_{n-1}^{-1} }^{\tfrac{1}{\gamma-1}} g(r)^{-\eta-\beta}
    \leq c_3 \,\vartheta_n^{\tfrac{1}{\gamma-1}} g(r)^{-\beta_0}
    \leq c_4 \,r^{\gamma h_n-1} g(r)^{-\beta_1},
  \end{align*}
  since $r\in\ivof{\rho_{n}, \vartheta_{n}}$ and $\vartheta_{n}<<\rho_{n-1}\leq\vartheta_{n-1}$.
  In the other hand, if $r\in\ivof{\vartheta_{n},\rho_{n-1}}$,
  \begin{align*}
    \mu_{a,h}\pthb{ B(\sigma,2r) }
    &\leq c_2 \pthb{ \vartheta_n\vartheta_{n-1}\rho_{n-1}^{-1} }^{\tfrac{1}{\gamma-1}} g(\rho_{n-1})^{-\eta} \cdot \pthb{r\vartheta_n^{-1}}^{\tfrac{1}{\gamma-1}} g(r)^{-\beta} \\
    &\leq c_3 \, r^{\gamma h_{n-1}-1}  g(r)^{-\beta_2} \cdot\pthbb{ \frac{r}{\rho_{n-1}} }^{1/(\gamma-1) - (\gamma h_{n-1}-1)} \\
    &\leq c_3 \,r^{\gamma h_{n-1} -1}  g(r)^{-\beta_2},
  \end{align*}
  since $\gamma h_{n-1} -1 \leq 1/(\gamma-1)$ and $r\leq \rho_{n-1}$.
  These last two inequalities concludes the proof of Equation~\eqref{eq:mass_pple_spectrum1}, as we know that $h_n\rightarrow h$ uniformly for every $h\in\Hi$.

  Let us now prove the second part \eqref{eq:mass_pple_spectrum2}. For that purpose, we simply consider the sequence $r_n\eqdef\vartheta_n=\rho_n^{(\gamma-1)(\gamma h_n -1 )} g(\rho_n)^{-\alpha\gamma(\gamma-1)}$. In this case, the construction of the collection $\Vbb_n(h_n)$ ensures that $N(\sigma,r_n)=1$, and thus
  \begin{align*}
    \mu_{a,h}\pthb{ B(\sigma,2r_n) }
    = \mu_{a,h}\pthb{ \Ti_n }
    \leq c_1 \pthb{ \vartheta_n\vartheta_{n-1}\rho_{n-1}^{-1} }^{\tfrac{1}{\gamma-1}} g(\rho_{n-1})^{-\eta}
    \leq r_n^{1/(\gamma-1)}  g(r_n)^{-\beta_3}.
  \end{align*}
  This last bound concludes the proof of the lemma.
\end{proof}

Based on the result established in Lemma~\ref{lemma:usp_construction_spectrum3}, we may now prove the lower bound of Theorem~\ref{th:levy_tree_upper_spectrum1}.
\begin{lemma}  \label{lemma:stable_trees_spectrum_lb1}
  $\Nb\pth{ \dt \Ti }$-a.e. for every nonempty set $F$ of $\ivoo{0,h(\Ti)}$,
  \begin{align*}
    \forall h\in\ivofb{\tfrac{1}{\gamma}, \tfrac{1}{\gamma-1}};\quad \dimH \pthb{ E_\ell(h,\Ti)\cap \Ti(F) } \geq \gamma h - 1 + \dimH F.
  \end{align*}
  Moreover,
  \begin{align*}
    \forall h\in\ivofb{\tfrac{1+\gamma}{\gamma}, \tfrac{\gamma}{\gamma-1}};\quad \dimH \,\pthb{ E_\mb(h,\Ti) \cap \Ti(F) } \geq \gamma (h-1) - 1 + \dimH F.
  \end{align*}
\end{lemma}
\begin{proof}
  We begin by obtaining the lower bound on the multifractal spectrum of the local time.
  Let us set $\Hi\subset\ivofb{\tfrac{1}{\gamma},\tfrac{1}{\gamma-1}}$, $F$ be a nonempty set of $\ivoo{\epsilon,h(\Ti)-\epsilon}$ and $h\in\Hi$. We may suppose that $\dimH F >0$, otherwise the inequality is direct consequence of the mass distribution principle and Lemma~\ref{lemma:usp_construction_spectrum4}. Then, for any $s < \dimH F$, according to \citet[Cor. 4.12]{Falconer-2003}, there exists a measure $\nu$ supported by $F$ such that
  \[
    \forall a\in F,\ \forall r>0;\quad \nu(B(a,r)) \leq c_0\,r^s.
  \]
  Let us denote by $G(F,h)$ the set $G(F,h) = \cup_{a\in F} G(a,h)$ and define a measure $\mu_{F,h}$ on it:
  \[
    \mu_{F,h}(\dt\sigma) = \int \nu(\dt a) \,\mu_{a,h}(\dt\sigma).
  \]
  Then, for every $\sigma\in G(F,h)$ and any $r>0$, Lemma~\ref{lemma:usp_construction_spectrum4} entails
  \begin{align*}
    \mu_{F,h}\pthb{ B(\sigma,r) }
    \leq \int_{\ivoo{h(\sigma)-r,h(\sigma)+r}} \nu(\dt a) \,\mu_{a,h}( B(\sigma,r) \cap \Ti(a) )
    \leq c_0\,r^{\gamma h-1-\eps(r)+s} g(r)^{-\beta}.
  \end{align*}
  In addition, we note that the upper bound presented in Lemma~\ref{lemma:usp_spectrum_ub1} induces that for any $h'<h$ and every $a>0$, $\mu_{a,h}\pthb{ E_\ell(h',\Ti)\cap \Ti(a) } = 0$, and thus $\mu_{F_,h}\pthb{ E_\ell(h',\Ti)\cap \Ti(F) } = 0$. Hence, defining
  \[
    \widehat G(F,h) = G(F,h) \setminus \bigcup_{h'<h} E_\ell(h',\Ti)\cap\Ti(F),
  \]
  we observe that $\widehat G(F,h) \subset E_\ell(h,\Ti)\cap \Ti(F) \subset E_\ell(h,\Ti)\cap \Ti(F)$. Therefore,
  \[
    \dimH\pthb{ E_\ell(h,\Ti)\cap \Ti(F) } \geq \dimH\pthb{ E_\ell(h,\Ti)\cap \Ti(F) } \geq \dimH \widehat G(F,h) \geq \gamma h-1+s,
  \]
  where the last inequality is a consequence of the classic mass distribution principle (we refer to \cite{Falconer-2003} for a complete reference on the subject). Considering a sequence of sets $F_n\subset F$ such that $s_n\rightarrow \dimH F$, we obtain the lower bound on the dimension of $E_\ell(h,\Ti)\cap \Ti(F)$. Finally, observing that $\lim_{\epsilon\rightarrow 0} \dimH F\cap\ivoo{\epsilon,h(\Ti)-\epsilon} = \dimH F$, we get the desired uniform lower bound.\vsp

  Using the previous construction, and arguments, we may also prove the lower bound on the multifractal spectrum of the mass measure. To begin with, we observe that Proposition~\ref{prop:scaling_exponents} implies that for any $h\leq\tfrac{1}{\gamma-1}$, $G(F,h)\subset F_\ell(h,\Ti)\cap\Ti(F) \subset F_\mb(h+1,\Ti)\cap\Ti(F)$. Then, defining
  \[
    \widetilde G(F,h) = G(F,h) \setminus \bigcup_{h'<h} E_\mb(h'+1,\Ti)\cap\Ti(F),
  \]
  we get similarly $\widetilde G(F,h) \subset E_\mb(h+1,\Ti)\cap \Ti(F)$ and therefore $\dimH\pthb{ E_\mb(h+1,\Ti)\cap \Ti(F) } \geq \dimH \widetilde G(F,h) \geq \gamma h-1+s$.
  As previously, we deduce the desired uniform bound from the former inequality.
\end{proof}

To end the proof of Theorem~\ref{th:levy_tree_upper_spectrum1}, we investigate the limit case $h=\tfrac{1}{\gamma}$ (resp. $h=\tfrac{1+\gamma}{\gamma}$).
\begin{lemma}  \label{lemma:stable_trees_spectrum_lb2}
  $\Nb\pth{ \dt \Ti }$-a.e. for every nonempty set $F$ of $\ivoo{0,h(\Ti)}$,
  \begin{align*}
    \dimH \pthb{ E_\ell(\tfrac{1}{\gamma},\Ti)\cap \Ti(F) } \geq \dimH F \quad\text{and}\quad \dimH \,\pthb{ E_\mb(\tfrac{1+\gamma}{\gamma},\Ti) \cap \Ti(F) } \geq \dimH F.
  \end{align*}
  In particular, $\Nb\pth{ \dt \Ti }$-a.e. for every level $a>0$, $E_\ell\pthb{\tfrac{1}{\gamma},\Ti}\cap \Ti(a)\neq\vset$ and $E_\mb\pthb{\tfrac{1+\gamma}{\gamma},\Ti}\cap \Ti(a)\neq\vset$.
\end{lemma}
\begin{proof}
  As previously, we may first focus on the local time, and then deduce the equivalent property for the mass measure. Since the proof of this statement follows the same structure than $h\in\ivofb{\tfrac{1}{\gamma},\tfrac{1}{\gamma-1}}$, we only sketch the main steps.

  As outlined above, in the limit case $h=\tfrac{1}{\gamma}$, we must have $\alpha<0$. Furthermore, we introduce as well a decreasing sequence such that $(\rho_n)_{n\in\N}$ such that $\vartheta_n = g(\rho_n)^{-\alpha\gamma(\alpha-1)}\asymp\rho_{n-1}$. Then, since the conclusions of Lemmas \ref{lemma:usp_balanced_balls} and \ref{lemma:usp_large_balls_spectrum1} still hold, we can construct similarly  for every level $a\in\ivoo{0,h(\Ti)}$ a set $G(a)\subset F_\ell(\Ti)\cap\Ti(a)$ and a measure $\mu_a(\dt\sigma)$ supported by $G(a)$. Furthermore, the estimates presented in Lemma~\ref{lemma:usp_construction_spectrum4} are still valid, proving that
  \begin{align}  \label{eq:mass_pple_spectrum1_bis}
    \forall \sigma\in G(a),\ \forall r>0;\quad \mu_{a}\pthb{ B(\sigma,r) } \leq g(r)^{-\alpha\gamma-\beta}.
  \end{align}
  where $\beta>0$ is a fixed constant. Furthermore,
  \begin{align}  \label{eq:mass_pple_spectrum2_bis}
    \forall \sigma\in G(a),\ \forall n\in\N;\quad \mu_{a}\pthb{ B(\sigma,\vartheta_n) } \leq \vartheta_n^{\tfrac{1}{\gamma-1}+\tfrac{\beta}{\alpha\gamma(\gamma-1)}}.
  \end{align}
  Note that $\abs{\alpha}$ can be supposed to be sufficiently large such that $-\alpha\gamma-\beta > 1$ and $\tfrac{\beta}{\alpha\gamma(\gamma-1)} < \eps$ for any $\eps>0$.

  Then, we observe according to the cover pf $E_\ell(h,\Ti)\cap\Ti(F)$ constructed in Lemma~\ref{lemma:usp_spectrum_ub2}, for any $h<\tfrac{1}{\gamma}$ and every level $a>0$, $\mu_{a}\pthb{ E_\ell(h,\Ti)\cap\Ti(a) } = 0$. Consequently, we may follow the procedure presented in \ref{lemma:stable_trees_spectrum_lb1} and get $\dimH \pthb{ E_\ell(\tfrac{1}{\gamma},\Ti)\cap \Ti(F) } \geq \dimH F$. A similar extension of the previous proof also provides the desired bound on the mass measure. Finally, the previous construction also proves the uniform non-emptiness.
\end{proof}

To end this first part on the uniform lower bound, let us briefly prove the lower bound in Proposition~\ref{prop:packing_dim_spectrum}.
\begin{lemma}
  $\Nb\pth{ \dt \Ti }$-a.e. for all levels $a\in\ivoo{0,h(\Ti)}$,
  \begin{align*}
    \forall h\in\ivffb{\tfrac{1}{\gamma}, \tfrac{1}{\gamma-1}};\quad \dimP \,E_\ell(h,\Ti)\cap\Ti(a) \geq \frac{1}{\gamma-1}\quad\text{and}\quad \dimP \,E_\mb(h+1,\Ti)\cap\Ti(a) \geq \frac{1}{\gamma-1}.
  \end{align*}
\end{lemma}
\begin{proof}
  The two lower bounds are a direct consequence of the second statement in Lemma~\ref{lemma:usp_construction_spectrum4} (or Equation~\ref{eq:mass_pple_spectrum2_bis} for the limit case) and the mass distribution principle for the packing dimension (see Theorem 6.11 in \citet{Mattila-1995}).
\end{proof}

\subsubsection{Proof of Theorem~\ref{th:levy_tree_upper_spectrum2} (lower bound)}

In the second part of this section, we present the proof of Theorem~\ref{th:levy_tree_upper_spectrum2}'s lower bound. The strategy adopted to tackle this question is similar to the proof of Theorem~\ref{th:levy_tree_upper_spectrum1} as we describe as well a constructive method to get simultaneously a collection of suitable Hausdorff measures related to a fixed Borel set $F\subset\ivoo{0,\infty}$.

Before obtaining the general statement presented in Theorem~\ref{th:levy_tree_upper_spectrum2}, we will start by considering the particular case of a set $F$ satisfying the following assumptions:
\begin{enumerate}[\it (i)]
  \item $F$ is a regular and compact set;
  \item $F$ satisfies the strong Frostman's lemma \eqref{eq:strong_frostman} with a probability measure $\mu_F$ supported by the set $F$;
  \item for any $n$ sufficiently large and every $I\in\Di_n$
  \begin{align}  \label{eq:ass_lower_bound}
    F\cap I^\circ = \vset\quad\text{or}\quad \mu_F\pth{F\cap I} \geq (n+1)^{-4}\delta_n.
  \end{align}
\end{enumerate}
In the rest of this section, we will say that such a set $F$ satisfies \emph{(i)-(iii)}. In addition, $b$ will denote a real number such that $F\subset\ivoo{0,b}$. The next lemma shows such a specific configuration can always be extracted.
\begin{lemma}  \label{lemma:stable_usp3_tech2}
  Suppose $F\subset\R$ is a regular Borel set satisfying the strong Frostman's lemma \eqref{eq:strong_frostman}. Then, there exists $F_\star \subset F$ verifying assumptions \emph{(i)-(iii)}.
\end{lemma}
\begin{proof}
  First note there exists $b>0$ sufficiently large such that $\mu_F(F\cap\ivoo{0,b})>0$. Then, defining $F_0=\supp F \cap \ivoo{0,b}$, the mass distribution principle and common properties of fractal dimension imply that $\dimH F_0=\dimP F_0=\dimH F=\dimP F$.

  To obtain $(iii)$, define for every $k\in\N$: $\Ei_k \eqdef \brc{I\in\Di_k : \mu_F(F_0\cap I) \leq k^{-2}\delta_k \text{ and } F_0\cap I^\circ\neq\vset }$. Then, for any $n\in\N$,
  \begin{align*}
    \mu_F\pthbb{ \bigcup_{k\geq n} \bigcup_{I\in\Ei_k} F_0\cap I }
    \leq \sum_{k\geq n} \sum_{I\in\Ei_k} \mu_F\pth{ F_0\cap I }
    \leq c_0\sum_{k\geq n} k^{-2}
    \longrightarrow_{n\rightarrow\infty} 0
  \end{align*}
  In particular, for $n$ large enough, the sum is smaller than $\mu_F(F_0)>0$. Hence, let us define
  \begin{align*}
    F_\star = F_0 \setminus \bigcup_{k\geq n}\bigcup_{I\in\Ei_k} I^\circ.
  \end{align*}
  and show it satisfies \emph{(i)-(iii)} with the normalised measure $\mu_{F_\star}(\dt x) = \tfrac{\mu_{F}(\dt x\cap F_\star)}{\mu_{F}(F_\star)}$. $F_\star$ is clearly compact, and since it still satisfies the strong Frostman's lemma, it is as well regular. Suppose now $J\in\Di_n$ such that $F_\star\cap J^\circ\neq\vset$. According to the previous construction, $J\notin\Ei_n$, and thus,
  \begin{align*}
    \mu_{F}(J\cap F_\star) = \mu_F(J\cap F_0) - \mu_F\pthbb{ \bigcup_{k\geq n+1} \bigcup_{I\in\Ei_k} F_0\cap I\cap J }.
  \end{align*}
  There exists a partition $\Pi(J)\subset\cup_{k\geq n+1} \Ei_k$ of non-overlapping intervals such that
  \begin{align*}
    \bigcup_{k\geq n+1} \bigcup_{I\in\Ei_k} F_0\cap I\cap J = \bigcup_{I\in\Pi(J)} F_0\cap I.
  \end{align*}
  Consequently,
  \begin{align*}
    \mu_F\pthbb{ \bigcup_{k\geq n+1} \bigcup_{I\in\Ei_k} F_0\cap I\cap J }
    = \sum_{I\in\Pi(J)} \mu_F(F_0\cap I) \leq \sum_{k\geq n+1} k^{-2}\delta_k \,\#\pthb{\Pi(J)\cap\Ei_k} \leq (n+1)^{-2}\delta_n,
  \end{align*}
  since $\sum_{k\geq n+1} \delta_k \,\#\pthb{\Pi(J)\cap\Ei_k} \leq \delta_n$. Therefore,
  \begin{align*}
    \mu_{F_\star}(J\cap F_\star) \geq c_0\pthb{ n^{-2} - (n+1)^{-2} } \delta_n \geq n^{-4} \delta_n
  \end{align*}
  which proves the desired property \emph{(iii)}.
\end{proof}

As previously, we start by proving a few technical lemmas which will be necessary to the construction of a proper collection of measures. For that purpose, let us introduce a few useful notations, starting with the events $\Ai(v,w,\rho,h)$:
\begin{align*}
  \Ai(v,w,\rho,h) = \brcB{\Ti : \forall u\in\ivff{v,w};\  \bk{\ell^u}(\Ti) \in \ivffb{ \rho^{h}, 2\rho^{h} } },
\end{align*}
In addition, we will write $\Ai(v,\rho,h)\eqdef \Ai(v,v,\rho,h)$ and $\Ai(\rho,h)\eqdef \Ai(\kappa\rho,\rho/\kappa,\rho,h)$ to designate the simpler forms. Then, we define the collection of subtrees:
\begin{align*}
  \Tbb(a,\delta,h) = \brcB{ \Ti_\sigma\in\Tbb(a,\delta) : \forall u\in\ivff{\kappa\delta,\delta/\kappa};\  \bk{\ell^u}(\Ti_\sigma)  \in \ivffb{ \rho^{h}, 2\rho^{h} } }
\end{align*}
Recalling that that $\Di(\delta) = \brcb{ \ivff{k\delta,(k+1)\delta} : k\in\N}$, we also introduce the following random collection of intervals:
\begin{align*}
  \Di(\delta,h) = \brcb{I_k\in\Di(\delta) : \Tbb((k-1)\delta,\delta,h) \neq \vset },
\end{align*}
where $I_k$ stands for the interval $\ivff{k\delta,(k+1)\delta}$.

Finally, we set in the rest of this section $\epsilon>0$ and a closed interval $\Hi=\ivff{h_0,h_1}\subset\ivoob{0,\tfrac{1}{\gamma}}$ such that $\gamma h_0 - 1+s>2\epsilon$.\vsp

In the following key lemma, we present how to properly modify the measure $\mu_F(\dt x)$ in order to then construct recursively a collection of desired Hausdorff measures necessary to the proof of Theorem~\ref{th:levy_tree_upper_spectrum3}.
\begin{lemma}  \label{lemma:stable_usp3_tech0}
  Suppose $F$ satisfies \emph{(i)-(iii)}, $\rho,\delta>0$, $h_\star\in\ivof{0,\tfrac{1}{\gamma}}$ and $h\in\Hi$ such that $\delta|\rho$ and $\delta\leq 2^{-1/\rho}$. We introduce the following random measure on $F\cap\ivff{\rho,2\rho}$
  \begin{align*}
    \nu(\dt x) \eqdef \rho^{-h_\star}\delta^{\gamma h-1}\sum_{I\in\Di(\delta,h)} \indi_{\brc{x\in I\cap F\cap \ivff{\rho,2\rho}}} \, \mu_F(\dt x)
  \end{align*}
  We aim to control to the behaviour of the previous measure in terms of mass distribution and mass conservation, defining for that purpose the following event:
  \begin{align*}
    \Bi(\rho,\delta,h,h_\star) = \brcB{\Ti :\
    & \nu(F) \geq c_\nu\,\mu_F\pthb{\ivff{\rho,2\rho}} \text{ and } \\
    &\forall x\in F,\ \forall r\in\ivff{\delta,\rho} : \nu(B(x,r)) \leq g(r)^{-\beta}r^{\gamma h-1+s-\epsilon} \rho^{1-\gamma h} },
  \end{align*}
  where $c_\nu>0$ is chosen sufficiently small and $\beta>3+2\epsilon$.

  Then, the measure of the event $\Bi(\rho,\delta,h,h_\star)^c$ is bounded as following
  \begin{align*}
    \Nb_{\kappa\rho}\pthb{ \Bi(\rho,\delta,h,h_\star)^c \cap \Ai(\rho,h_\star) } \leq c_0\,\brcB{ \exp\pthb{ -\mu_F\pthb{\ivff{\rho,2\rho}}\delta^{-\epsilon} } + \exp\pthb{-g(\rho)^{-1-\epsilon}} },
  \end{align*}
  where the constant $c_0>0$ is independent of $\delta$, $\rho$, $h$ and $h_\star$. In addition, we note that if $F\cap\ivff{\rho,2\rho}=\vset$, $\Nb_{\kappa\rho}\pthb{ \Bi(\rho,\delta,h,h_\star)^c \cap \Ai(\rho,h_\star) } = 0$.
\end{lemma}
\begin{proof}
  Let us first obtain an upper bound on the event $\brc{\Ti : \exists x\in F, \exists r\in\ivff{\delta,\rho} : \nu(B(x,r)) \geq g(r)^{-\beta}\pth{r\rho^{-1}}^{\gamma h-1+s} }$. Observe that the previous event is included in
  \begin{align*}
    \brcb{\Ti : \exists \delta_n\in\ivff{\delta,\rho}, \exists J\in\Di_n : \nu(J) \geq c_\star\,g(\delta_n)^{-\beta}\pth{\delta_n\rho^{-1}}^{\gamma h-1+s} },
  \end{align*}
  for a constant $c_\star$ independent of $\delta$, $\rho$, $h$ and $h_\star$.

  Hence, let us now set $v>0$ and $r\in\ivff{\delta,\rho}$ such that $J\eqdef\ivff{v,v+r}\subset\ivff{\rho,2\rho}$. We then decompose the measure $\nu$ into two separate components: define $\Di_i(\delta,h) = \brc{I_k\in\Di(\delta,h) : k\text{ mod } 2=0}$, where $i\in\brc{0,1}$, and
  \[
    \forall i\in\brc{0,1};\quad\nu_i(\dt x) = \rho^{-h_\star}\delta^{\gamma h-1}\sum_{I\in\Di_i(\delta,h)} \indi_{\brc{x\in I\cap F\cap \ivff{\rho,2\rho}}} \, \mu_F(\dt x).
  \]
  We may easily observe that $\nu=\nu_0+\nu_1$ and thus, $\brc{ \nu(J) \geq 2z } \subset\brc{ \nu_{0}(J) \geq z }\cup \brc{ \nu_{1}(J) \geq z }$ for any $z\geq 0$. As a consequence, we only need to understand the tail behaviour of these two measures.
  Then, for any $z,\lambda>0$
  \begin{align*}
    \Nb_{\kappa\rho}\pthb{ \nu_{0}(J) \geq z\cap \Ai(\rho,h_\star) }
    &= \Nb_{\kappa\rho}\pthb{ \exp\pthb{ \lambda \rho^{h_\star}\delta^{1-\gamma h}\nu_{0}(J) } \geq \exp\pthb{ \rho^{h_\star}\delta^{1-\gamma h}\lambda z } \cap \Ai(\rho,h_\star) } \\
    &\leq \exp\pthb{ -\rho^{h_\star}\delta^{1-\gamma h}\lambda z } \Nb_{\kappa\rho}\pthb{ \exp\pthb{ \lambda \rho^{h_\star}\delta^{1-\gamma h}\nu_{0}(J) } \indi_{\Ai(\rho,h_\star)} }.
  \end{align*}
  Owing to the definition of the measure $\nu_{0}$, the second term is given by
  \begin{align*}
    \Nb_{\kappa\rho}\pthb{ \exp\pthb{ \lambda \rho^{h_\star}\delta^{1-\gamma h}\nu_{0}(J) } \indi_{\Ai(\rho,h_\star)} } =
    \Nb_{\kappa\rho}\pthbb{ \prod_{I_k\in\Di_0(\delta,h)}\exp\pthb{ \lambda \mu_F(I_k\cap J) } \indi_{\Ai(\rho,h_\star)} }.
  \end{align*}
  Let $m\in 2\N$ be the largest even integer such that $I_m\eqdef\ivff{m\delta,(m+1)\delta}\subset J$. Setting $v_m=(m-1)\delta$, the previous expression is then equal to
  \begin{align*}
    &\frac{v(v_m)}{v(\kappa\rho)}\Nb_{v_m}\pthbb{ \prod_{I_k\neq I_m\in\Di_0(\delta,h)}\exp\pthb{ \lambda \mu_F(I_k\cap J) } \Nb_{v_m}\pthcb{ \exp\pthb{ \lambda \mu_F(I_m\cap J) } \indi_{I_m\in\Di_0(\delta,h)} \indi_{\Ai(\rho,h_\star)} }{\Gi_{v_m}} } \\
    &\leq \frac{v(v_m)}{v(\kappa\rho)}\Nb_{v_m}\pthbb{ \prod_{I_k\neq I_m\in\Di_0(\delta,h)}\exp\pthb{ \lambda \mu_F(I_k\cap J) } \indi_{\Ai(\kappa\rho,v_m,\rho,h_\star)} \Nb_{v_m}\pthcb{ \exp\pthb{ \lambda \mu_F(I_m\cap J)  \indi_{I_m\in\Di_0(\delta,h)} } }{\Gi_{v_m}} }.
  \end{align*}
  We recall that $I_m\in\Di_0(\delta,h)$ if and only if $\Tbb(v_m,\delta,h) \neq \vset$. For any $w\in\ivff{\kappa\rho,\rho/\kappa}$, owing to the branching property, under $\Nb_w$ and given $\Gi_w$, $Z(w,\delta,h)=\# \Tbb(w,\delta,h)$ is a Poisson random variable parametrized by $\bk{\ell^w}\Nb\pthb{\Ai(\delta,h)} \asymp \bk{\ell^{w}} \delta^{1-\gamma h}$ using Lemma~\ref{lemma:local_time_inf_tail1}. Hence,
  \begin{align*}
    \Nb_{w}\pthcb{ \Tbb(w,\delta,h)\neq\vset }{\Gi_{w}} \leq 1 - \exp\pthb{-c_1\bk{\ell^{w}} \delta^{1-\gamma h}}
    \leq c_1 \bk{\ell^{w}} \delta^{1-\gamma h}.
  \end{align*}
  As a consequence, $\Nb_{w}\pthcb{ \Tbb(w,\delta,h)\neq\vset }{\Gi_{w}} \indi_{\Ai(\kappa\rho,w,\rho,h_\star)} \leq 2c_1 \, \rho^{h_\star} \delta^{1-\gamma h}$,
  which entails
  \begin{align*}
    \Nb_{v_m}\pthcB{ \exp\pthb{ \lambda \mu_F(I_m\cap J)  \indi_{I_m\in\Di_0(\delta,h)} } }{\Gi_{v_m}}\indi_{\Ai(\kappa\rho,v_m,\rho,h_\star)}
    &\leq 1+ 2c_1 \, \rho^{h_\star} \delta^{1-\gamma h} \pthb{\e^{\lambda\mu_F(I_m\cap J)}-1} \\
    &\leq 1+c_2 \, \rho^{h_\star} \delta^{1-\gamma h} \lambda\mu_F(I_m\cap J),
  \end{align*}
  assuming that $\lambda$ is chosen such that $\lambda\mu_F(I_m\cap J) \leq c$, for some $c>0$ independent of $m$, $\delta$ and $\rho$.
  Hence, by induction on $m\in2\N$, we prove that $\Nb_{\kappa\rho}\pthb{ \exp\pthb{ \lambda \rho^{h_\star}\delta^{1-\gamma h}\nu_{0}(J) } \indi_{\Ai(\rho,h_\star)} }$ is upper bounded by
  \begin{align*}
    \Nb_{\kappa\rho}\pthb{ \Ai(\kappa\rho,\rho,h_\star) }\prod_{I_k\in\Di_0(\delta),k\in2\N}\brcB{ 1+c_2 \,\rho^{h_\star} \delta^{1-\gamma h} \lambda\mu_F(I_k\cap J) }.
  \end{align*}
  The logarithm of the product term is then itself bounded above by
  \begin{align*}
    \sum_{I_k\in\Di_0(\delta),k\in2\N} \log\pthb{ 1+c_2 \, \rho^{h_\star} \delta^{1-\gamma h} \lambda\mu_F(I_k\cap J) } \
    \leq c_2 \, \rho^{h_\star} \delta^{1-\gamma h} \lambda\mu_F(J),
  \end{align*}
  using to the common inequality $\log(1+y)\leq y$. Combining the previous estimates, we get
  \begin{align*}
    \Nb_{\kappa\rho}\pthb{ \nu_{0}(J) \geq z\cap \Ai(\rho,h_\star) }
    \leq &\exp\brcB{ -\rho^{h_\star}\delta^{1-\gamma h}\lambda \pthb{ z - c_2 \, \mu_F(J) } },
  \end{align*}
  noting that $\Nb_{\kappa\rho}\pthb{ \Ai(\kappa\rho,\rho,h_\star) } \leq 1$. Then, according to the strong Frostman's lemma on $F$, if $\rho$ is sufficiently small, for every $I\in\Di(\delta)$, $\mu_F(I) \leq \delta^{s-\epsilon}$ where $s=\dimH F$. Hence, we may set $\lambda = \delta^{-s+\epsilon}$ and $z = g(r)^{-\beta}r^{s-\epsilon} \pthb{r\rho^{-1}}^{\gamma h-1}$. Then, since $\gamma h-1 < 0$ and $r\leq\rho$, $\mu_F(J) \leq r^{s-\epsilon}\pthb{r\rho^{-1}}^{\gamma h-1} \leq g(r)^{\beta} z$. Consequently,
  \begin{align*}
    \Nb_{\kappa\rho}\pthb{ \nu_{0}(J) \geq z\cap \Ai(\rho,h_\star) }
    &\leq \exp\brcB{ -c_3\,\rho^{h_\star}g(r)^{-\beta}\pthb{\delta r^{-1}}^{1-\gamma h-s+\epsilon}\rho^{1-\gamma h} }.
  \end{align*}
  Since $\gamma h - 1+s>2\epsilon$, for any $\rho$ small enough, the function $r\mapsto\rho^{h_\star}g(r)^{-\beta}\pthb{\delta r^{-1}}^{1-\gamma h-s+\epsilon}\rho^{1-\gamma h}$ reaches its minimum on $\ivff{\delta,\rho}$ at $r=\delta$, and is thus bounded by $\rho^{h_\star+1-\gamma h}g(\delta)^{-\beta} \leq g(\delta)^{-\beta+2}$ as $\delta \leq 2^{-1/\rho}$. Therefore, $\Nb_{\kappa\rho}\pthb{ \nu_{0}(J) \geq z\cap \Ai(\rho,h_\star) }
    \leq \exp\pthb{ -g(\delta)^{-\beta+2} }$ and
  \begin{align*}
    \sum_{n\in\N:\delta_n\in\ivff{\delta,\rho}} \sum_{J\in\Di_n : J\subset\ivff{\rho,2\rho}} &\Nb_{\kappa\rho}\pthb{ \nu_{0}(J) \geq g(\delta_n)^{-\beta} \delta_n^s \pthb{\delta_n\rho^{-1}}^{\gamma h-1}\cap \Ai(\rho,h_\star) } \\
    &\leq \sum_{n:\delta_n\in\ivff{\delta,\rho}} 2\rho \delta_n^{-1} \exp\pthb{ -g(\delta)^{-\beta+2} }
    \leq \exp\pthb{ -g(\delta)^{-1-\epsilon} },
  \end{align*}
  for any $\rho$ sufficiently small. An equivalent bound holds as well on the measure $\nu_{1}(\dt x)$, therefore proving the first part of our statement.\vsp

  To obtain the second part of the lemma, we proceed similarly. Let us first note there exists $i\in\brc{0,1}$ such that
  \begin{align*}
    \sum_{I_k\in\Di(\delta) : k\text{ mod } 2=i} \mu_F(I_k\cap\ivff{\rho,2\rho}) \geq \tfrac{1}{2}\mu_F(\ivff{\rho,2\rho}).
  \end{align*}
  Without any loss of generality we may assume that $i=0$ and simply observe that $\brc{ \nu(F) \leq z } \subset \brc{ \nu_{0}(F) \leq z }$. Moreover, for any $z,\lambda>0$
  \begin{align*}
    &\Nb_{\kappa\rho}\pthb{ \nu_{0}(F) \leq z \cap \Ai(\rho,h_\star) } \\
    &= \Nb_{\kappa\rho}\pthb{ \exp\pthb{ -\lambda \rho^{h_\star}\delta^{1-\gamma h}\nu_{0}(F) } \geq \exp\pthb{ -\rho^{h_\star}\delta^{1-\gamma h}\lambda z } \cap \Ai(\rho,h_\star) } \\
    &\leq \exp\pthb{ \rho^{h_\star}\delta^{1-\gamma h}\lambda z } \Nb_{\kappa\rho}\pthb{ \exp\pthb{ -\lambda \rho^{h_\star}\delta^{1-\gamma h}\nu_{0}(F) } \indi_{\Ai(\rho,h_\star)} }.
  \end{align*}
  The last term corresponds to
  \begin{align*}
    \Nb_{\kappa\rho}\pthb{ \exp\pthb{ -\lambda \rho^{h_\star}\delta^{1-\gamma h}\nu_{0}(F) } \indi_{\Ai(\rho,h_\star)} } =
    \Nb_{\kappa\rho}\pthbb{ \prod_{I_k\in\Di_0(\delta,h)}\exp\pthb{ -\lambda \mu_F(I_k\cap\ivff{\rho,2\rho}) } \indi_{\Ai(\rho,h_\star)} }.
  \end{align*}
  Similarly to the first part of the proof, for any $I_m\in\Di_0(\delta,h)$,
  \begin{align*}
    \Nb_{v_m}\pthcb{ \Tbb(v_m,\delta,h)\neq\vset }{\Gi_{v_m}} \geq 1 - \exp\pthb{-c_1\bk{\ell^{w}} \delta^{1-\gamma h}}
  \end{align*}
  and therefore
  \begin{align*}
    &\Nb_{v_m}\pthcb{ \exp\pthb{ -\lambda \mu_F(I_m\cap\ivff{\rho,2\rho}) \indi_{I_m\in\Di_0(\delta,h)} } }{\Gi_{v_m}}\indi_{\Ai(\kappa\rho,v_m,\rho,h_\star)} \\
    &\leq 1+c_2 \,\rho^{h_\star} \delta^{1-\gamma h} \pthb{\e^{-\lambda\mu_F(I_m\cap\ivff{\rho,2\rho})}-1} \\
    &\leq 1 - c_3 \, \rho^{h_\star} \delta^{1-\gamma h} \lambda\mu_F(I_m\cap\ivff{\rho,2\rho}).
  \end{align*}
  Hence, by induction, $\Nb_{\kappa\rho}\pthb{ \exp\pthb{ -\lambda \rho^{h_\star}\delta^{1-\gamma h}\nu_{0}(F) } \indi_{\Ai(\rho,h_\star)} }$ is bounded by
  \begin{align*}
    \prod_{I_k\in\Di_0(\delta),k\in2\N}\brcB{ 1-c_3 \, \rho^{h_\star} \delta^{1-\gamma h} \lambda\mu_F(I_k\cap \ivff{\rho,2\rho}) }.
  \end{align*}
  The logarithm of the previous term then satisfies
  \begin{align*}
    \sum_{I_k\in\Di_0(\delta),k\in2\N} \log\pthb{ 1-c_3 \rho^{h_\star} \delta^{1-\gamma h} \lambda\mu_F(I_k\cap\ivff{\rho,2\rho}) } \
    &\leq -c_3 \rho^{h_\star} \delta^{1-\gamma h} \sum_{I_k\in\Di_0(\delta),k\in2\N} \lambda\mu_F(I_k\cap\ivff{\rho,2\rho}) \\
    &\leq -c_4 \rho^{h_\star} \delta^{1-\gamma h} \lambda\mu_F(\ivff{\rho,2\rho}).
  \end{align*}
  Since we may as well set $\lambda=\delta^{-s+\epsilon}$,
  \begin{align*}
    \Nb_{\kappa\rho}\pthb{ \nu_{0}(F) \leq z\cap \Ai(\rho,h_\star) }
    \leq &\exp\brcB{ \rho^{h_\star}\delta^{1-\gamma h-s+\epsilon} \pthb{ z - c_4 \mu_F(\ivff{\rho,2\rho}) } }.
  \end{align*}
  Hence, setting $z=c_\nu\,\mu_F(\ivff{\rho,2\rho})$, where $c_\nu>0$ is chosen sufficiently small, we obtain
  \begin{align*}
    \Nb_{\kappa\rho}\pthb{ \nu_{0}() \leq c_\nu\,\mu_F(\ivff{\rho,2\rho}) \cap \Ai(\rho,h_\star) }
    \leq \exp\pthb{ -c_\nu\rho^{h_\star}\delta^{1-\gamma h-s+\epsilon} \mu_F(\ivff{\rho,2\rho}) }
    \leq \exp\pthb{ -\mu_F(\ivff{\rho,2\rho}) \delta^{-\epsilon} }
  \end{align*}
  recalling $1-\gamma h-s < -2\epsilon$.
\end{proof}
Note that even though the measure $\nu$ presented in Lemma~\ref{lemma:stable_usp3_tech0} has a random support, its restriction to intervals $I\in\Di(\delta,h)$ corresponds, up to a constant, to the deterministic measure $\mu_F$.

Similarly to the proof of Theorem~\ref{th:levy_tree_upper_spectrum1}, we now set a decreasing sequence $(\rho_n)_{n\in\N}$ such that $\rho_{n} = 2^{-\rho_{n-1}^{-1}}$ and $\rho_0=1$. Furthermore, we recall that $(\Hi_n)_{n\in\N}$ correspond to the approximation collections of elements in $\Hi=\ivff{h_0,h_1} \subset\ivoob{0,\tfrac{1}{\gamma}}$, where the latter is supposed to satisfy $\gamma h_0-1+s > 2\epsilon$.

We prove in the next lemma the main ingredient to our construction by induction of a collections of proper measures.
\begin{lemma}  \label{lemma:stable_usp3_tech3}
  Suppose $F$ satisfies \emph{(i)-(iii)}. $\Nb$-a.e., there exists $n_0(\Ti)$ such that for every $n\geq n_0$, every $h_n\in\Hi_n$ and all $j\rho_n\in\ivoo{0,b}$
  \begin{align*}
    \Ti_\sigma\in\Tbb(j\rho_n,\kappa\rho_n)\cap \Ai(\rho_n,h_n)\Longrightarrow\Ti_\sigma\in\Bi(\rho_n,\rho_{n+1},h_{n+1},h_n),
  \end{align*}
  where $h_{n+1} \eqdef h_n i$ and $i\in\brc{0,1}$.
\end{lemma}
\begin{proof}
  For any $n\in\N$, $h_{n+1}\in\Hi_{n+1}$ and $j\geq 1$, we define the r.v.
  \begin{align*}
    \Ni(j,n,h_{n+1}) = \brcb{ \Ti_\sigma \in \Tbb(j\rho_n\,\kappa\rho_n) : \Ti_\sigma \in \Bi(\rho_n,\rho_{n+1},h_{n+1},p_n(h_{n+1}))^c \cap \Ai(\rho_n,p_n(h_{n+1})) },
  \end{align*}
  where $p_n$ stands for the canonical projection $p_n:\Hi_{n+1}\rightarrow\Hi_n$,
  and
  \begin{align*}
    \Ni(n,h_{n+1}) = \sum_{j\rho_n\in\ivoo{0,b}} \Ni(j,n,h_{n+1}).
  \end{align*}
  Recall that for any $j\geq 1$ such that $\ivoo{j\rho_n,(j+1)\rho_n}\cap F\neq\vset$, assumption (iii) on $F$ entails: $\mu_F\pthb{\ivoo{j\rho_n,(j+1)\rho_n}\cap F\neq\vset} \geq (n+1)^{-4}\delta_n$. Hence, using the branching property and Lemma~\ref{lemma:stable_usp3_tech0}, we get
  \begin{align*}
    \Nb_{j\rho_n}\pthcb{ \Ni(j,n,h_{n+1}) \geq 1 }{\Gi_{j\rho_n}}
    &\leq \bk{\ell^{j\rho_n}} \Nb\pthB{ \Bi(\rho_n,\rho_{n+1},h_{n+1},p_n(h_{n+1}))^c \cap \Ai(\rho_n,p_n(h_{n+1})) } \\
    &\leq \bk{\ell^{j\rho_n}} v(\rho_n) \brcB{ \exp\pthb{ -\rho_{n+1}^{-\epsilon/2} } +\exp\pthb{-g(\rho_{n+1})^{-1-\epsilon}} }.
  \end{align*}
  As a consequence,
  \begin{align*}
    \Nb\pthb{ \Ni(n,h_{n+1}) \geq 1 }
    &\leq \sum_{j\rho_n\in\ivoo{0,b}} \Nb\pthb{  \Ni(j,n,h_{n+1}) \geq 1 } \\
    &\leq \sum_{j\rho_n\in\ivoo{0,b}} c_0\,v(\rho_n) \exp\pthb{-g(\rho_{n+1})^{-1-\epsilon}}
    \leq c_1\rho_n^{-\tfrac{\gamma}{\gamma-1}} \exp\pthb{ -g(\rho_n)^{-1-\epsilon} }.
  \end{align*}
  and
  \begin{align*}
    \sum_{n\in\N}\Nb\pthbb{ \bigcup_{h_{n+1}\in\Hi_{n+1}} \brcb{\Ni(n,h_{n+1}) \geq 1} } \leq c_1\, \sum_{n\in\N} 2^{n+2}\rho_n^{-\tfrac{\gamma}{\gamma-1}} \exp\pthb{ -g(\rho_n)^{-1-\epsilon} } < \infty.
  \end{align*}
  Borel--Cantelli lemma then entails the desired result.
\end{proof}

Finally, we also need a lemma which allows us to initialise the construction by induction.
\begin{lemma}  \label{lemma:stable_usp3_tech4}
  Suppose $F$ satisfies \emph{(i)-(iii)}. $\Nb_b(\dt\Ti)$-a.e., there exists $n_0(\Ti)$ such that for all $n\geq n_0$,
  \begin{align*}
    \forall h_n\in\Hi_n,\ \exists I_k\in\Di_n(F);\quad \Tbb((k-1)\rho_n,\kappa\rho_n)\cap \Ai(\rho_n,h_n) \neq\vset.
  \end{align*}
\end{lemma}
\begin{proof}
  As the proof is a simpler version of the coming Lemma~\ref{lemma:stable_usp2_tech0}, we only focus on the main arguments.
  Without any loss of generality, we assume that $F\subset\ivoo{\eps,b-\eps}$, where $\eps>0$ is sufficiently small.
  Since the local time is càdlàg, $\brc{h(\Ti)>b} \subset \lim_{\ell_\star\rightarrow 0} \brcb{\inf_{a\in\ivff{\eps,b-\eps}} \bk{\ell^a} \geq \ell_\star }$. Hence, let us set $\ell_\star >0$ and define
  \begin{align*}
    \forall u,v\in\ivoo{0,b};\quad \Bi(u,v) = \brcB{\Ti : \inf_{a\in\ivff{u,v}} \bk{\ell^a} \geq \ell_\star }.
  \end{align*}
  In addition, for any $n\in\N$ and $k\geq 1$, let $Y(k,n) = \#\pthb{\Tbb((k-1)\rho_n,\kappa\rho_n)\cap \Ai(\rho_n,h_n)}$. Then, due to the branching property, we observe
  \begin{align*}
    \Nb_{(k-1)\rho_n}\pthcb{ Y(k,n)=0 \cap \Bi(\eps,(k+1)\rho_n) }{\Gi_{(k-1)\rho_n}} \leq \exp\pthb{-c_0\ell_\star \rho_n^{1-\gamma h_n} } \indi_{\Bi(\eps,(k-1)\rho_n)}.
  \end{align*}
  Therefore, for any $i\in\brc{0,1}$, we get by induction
  \begin{align*}
    \Nb_b\pthbb{ \Bi(\eps,b-\eps) \cap \bigcap_{I_k\in\Di_{n,i}(F)} Y(k,n)=0 } \leq c_1\exp\pthb{-c_0\ell_0 \,\#\Di_{n,i}(F)\rho_n^{1-\gamma h_n} },
  \end{align*}
  where $\Di_{n,i}(F) = \brc{I_k\in\Di_n(F) : k\text{ mod }2 = i}$. Note that for some $i\in\brc{0,1}$, $\#\Di_{n,i}(F)\geq \#\Di_{n}(F) / 2  \geq c\,\mu_F(F) \rho_n^{-s+\epsilon}$. Consequently,
  \begin{align*}
    \sum_{h_n\in\Hi_n} \Nb_b\pthbb{ \Bi(\eps,b-\eps) \cap \bigcap_{I_k\in\Di_{n}(F)} Y(k,n)=0 }
    &\leq c_1 2^n\exp\pthb{-c_2 \mu_F(F) \rho_n^{1-\gamma h_n-s+\epsilon} } \\
    &\leq c_12^n\exp\pthb{-c_2 \mu_F(F) \rho_n^{-\epsilon} }.
  \end{align*}
  Hence, due to Borel--Cantelli lemma, on the event $\Bi(\eps,b-\eps)$, there exists $n_0(\Ti)$ such that for all $n\geq n_0$ and all $h_n\in\Hi_n$, there is $I_k\in\Di_n(F)$ such that $\Tbb((k-1)\rho_n,\kappa\rho_n)\cap \Ai(\rho_n,h_n) \neq\vset$. We conclude the proof of the lemma by considering the limits $\ell_\star\rightarrow 0$ and $\eps\rightarrow 0$.
\end{proof}

Using the previous lemmas, we may now present the construction of a proper collection of measures on the set $\brcb{ a\in F : F_\ell(h,\Ti)\cap \Ti(a) \neq \vset }$.
\begin{lemma}  \label{lemma:usp_th3_cstr1}
  Suppose $F$ satisfies \emph{(i)-(iii)}. $\Nb_b$-a.e. and for every $h\in\Hi$, there exists a nonempty compact set $\Ii(h)$ such that
  \begin{align*}
    \Ii(h) \subset \brcb{ a\in F : F_\ell(h,\Ti)\cap \Ti(a) \neq \vset }.
  \end{align*}
  In addition, there is a probability measure $\mu_h$ supported by $\Ii(h)$ such that for all $a\in\Ii(h)$,
  \begin{align*}
    \forall r\in\ivoo{0,r_0};\quad \mu_h(B(a,r)) \leq r^{\gamma h-1+s-\eps(r)} g(r)^{-\eta},
  \end{align*}
  where $\eps(\cdot)$ is a positive non-decreasing function satisfying $\lim_{\eps\rightarrow 0} \eps(r) = 0$, and $r_0>0$ and $\eta>0$ are independent of $h$ and $a$.
\end{lemma}
\begin{proof}
  Recall that due to the previous Lemma~\ref{lemma:stable_usp3_tech3}, $\Nb_b$-a.e. for any $n\geq n_0$, every $h_n\in\Hi_n$ and all $j\rho_n\in\ivoo{0,b}$
  \begin{align*}
    \Ti_\sigma\in\Tbb(j\rho_n,\kappa\rho_n)\cap \Ai(\rho_n,h_n)\Longrightarrow\Ti_\sigma\in\Bi(\rho_n,\rho_{n+1},h_{n+1},h_n),
  \end{align*}
  where $h_{n+1} \eqdef h_n i$ and $i\in\brc{0,1}$.
  In addition, for any $n\geq n_0$, Lemma~\ref{lemma:stable_usp3_tech4} states that  $\Nb_b$-a.e. there exists $I_k\in\Di_n(F)$ such that  $\Tbb((k-1)\rho_n,\kappa\rho_n)\cap \Ai(\rho_n,h_n) \neq\vset$.

  Let us start with the construction of the set $\Ii(h)$, where $h\in\Hi$ is fixed. Similarly to the proof of Lemma~\ref{lemma:usp_construction_spectrum4}, we simultaneously define by induction the collections $(\Ibb(n))_{n\in\N}$ of nested dyadic intervals and the collections $(\Tbb(n))_{n\in\N}$ of nested subtrees.

  Then, let $I_k$ be the interval satisfying Lemma~\ref{lemma:stable_usp3_tech4}, with $n=n_0$, and define $\Ibb(n_0) \eqdef \brcb{ \ivff{k\rho_{n_0},(k+1)\rho_{n_0}} }$ and $\Tbb(n_0) \eqdef \brc{ \Ti_\sigma }$, where $\Ti_\sigma$ is an element of the non-empty collection $\Tbb((k-1)\rho_{n_0},\kappa\rho_{n_0})\cap \Ai(\rho_{n_0},h_{n_0})$.

  Let us now suppose that $\Ibb(n)$ and $\Tbb(n)$ have been properly defined for a given $n\geq n_0$. For any $I_n\in\Ibb(n)$ and the corresponding $\Ti_n\in\Tbb(n)$, we set:
  \[
    \Ibb(n+1,I_n) \eqdef \brcb{I\in\Di(\rho_{n+1},h_{n+1},\Ti_n) : I\subset I_n }
  \]
  using notations introduced in Lemma~\ref{lemma:stable_usp3_tech0}. $\Tbb(n+1,\Ti_n)$ is then defined as the collection of subtrees naturally associated to every $I_{n+1}\in\Ibb(n+1,I_n)$. Using the previous notations, we also set
  \[
    \Ibb(n+1) = \bigcup_{I_n\in\Ibb(n)} \Ibb(n+1,I_n) \quad\text{and}\quad \Tbb(n+1) = \bigcup_{\Ti_n\in\Tbb(n)} \Tbb(n+1,\Ti_n).
  \]
  The previous construction ensures that for every $\Ti_{n+1}\in\Tbb(n+1)$,
  \[
    \forall u\in\ivff{\kappa\rho_{n+1},\rho_{n+1}/\kappa};\quad \bk{\ell^u}(\Ti_{n+1}) \in \ivffb{\delta^{h_{n+1}}_{n+1}, 2\delta^{h_{n+1}}_{n+1} },
  \]
  which therefore proves the consistency of the induction. Finally, we may define the set $\Ii(h)$ as following
  \[
    \Ii(h) = \bigcap_{n\geq n_0} \Ii(h,n) \quad\text{where } \Ii(h,n) \eqdef \bigcup_{I_n\in\Ibb(n)} I_n.
  \]
  Since $(\Ii(h,n))_{n\in\N}$ is a decreasing sequence of compact set, $\Ii(h)$ is readily compact and non-empty.\vsp

  We may know prove that the set $\Ii(h)$ satisfies the expected properties. For any $a\in\Ii(h)$, the definition of $\Ii(h)$ and the compactness of $F$ ensures that $a\in F$. In addition, there exists a sequence $(\Ti_n)_{n\in\N}$ of embedded subtrees such that for every $n\geq n_0$, $\Ti_n\in\Tbb(j\rho_n,\kappa\rho_n)$ for some $j\in\N$, $a\in j\rho_n+\ivff{\kappa\rho_n,\rho_n/\kappa}$ and $\inf_{u\in\ivff{\kappa\rho_{n},\rho_{n}/\kappa}} \bk{\ell^u}(\Ti_{n}) \geq \delta^{h_n}_{n}$. The compactness of the previous subtrees ensures the existence of $\sigma\in\cap_{n\geq n_0} \Ti_{n}\cap\Ti(a)$, and due to the previous properties,
  \[
    \forall n\in\N;\quad \ell^a(B(\sigma,4\rho_n)) \geq \rho_n^{h_n}.
  \]
  The last bound clearly implies that $\sigma\in F_\ell(h,\Ti)\cap \Ti(a)$ and $\Ii(h) \subset \brcb{ a\in F : F_\ell(h,\Ti)\cap \Ti(a) \neq \vset }$.\vsp

  Let us now present the construction of the probability measure $\mu_h$ on the set $\Ii(h)$. Once more, we proceed similarly to the proof of Lemma~\ref{lemma:usp_construction_spectrum4} and define a converging sequence $(\mu_n)_{n\geq n_0}$ by induction, relying on the modification of $\mu_F$ presented in Lemma~\ref{lemma:stable_usp3_tech0}. We begin by setting
  \[
    \mu_{n_0}(\dt a) = c_0\, \mu_F(\dt a\cap I_{n_0}),
  \]
  where $I_{n_0}$ is the only interval in $\Ibb(n_0)$ and $c_0$ is a normalising constant such that $\mu_{n_0}(F)=1$. Then, given $\mu_n$ supported by $\Ii(h,n)\cap F$, we simply construct $\mu_{n+1}$ as following: for every $I_n\in\Ibb(n)$, we consider the definition of $\nu$ presented in Lemma~\ref{lemma:stable_usp3_tech0} using the parameters $\rho=\rho_n$, $\delta=\rho_{n+1}$, $h_\star=h_n$ and $h=h_{n+1}$. The measure $\mu_{n+1}$ is then defined on the interval $I_n$ by:
  \begin{align*}
    \mu_{n+1}(\dt a \cap I_n) = \frac{\mu_n(I_n)}{\nu(I_n)}\,\nu(\dt a\cap I_n).
  \end{align*}
  Note the construction by induction based on Lemmas~\ref{lemma:stable_usp3_tech0} and \ref{lemma:stable_usp3_tech3} is licit as the restriction of $\mu_n(\dt a)$ to $I_n$ is up to a multiplicative constant the deterministic measure $\mu_F(\dt a\cap I_n)$. In addition, since $\nu(I_n) \geq c_\nu\mu_F(I_n) = c\,\mu_n(I_n)$, $\nu(I_n)=0$ only if $\mu_n(I_n)=0$, therefore proving the consistency of the definition of $\mu_{n+1}(\dt a \cap I_n)$. To obtain a proper mass distribution, one needs to bound more precisely the renormalising constant:
  \begin{align*}
    \frac{\mu_n(I_n)}{\nu(I_n)} = \frac{\mu_n(I_n)}{\mu_F(I_n)}\cdot \frac{\mu_F(I_n)}{\nu(I_n)} \leq c_\nu^{-1} \frac{\mu_n(I_n)}{\mu_F(I_n)}.
  \end{align*}
  The latter term can be estimated based on the definition of $\nu$:
  \begin{align*}
    \frac{\mu_{n+1}(I_{n+1})}{\mu_F(I_{n+1})}
    = \frac{\mu_{n}(I_{n})}{\nu(I_n)} \,\rho_n^{-h_n} \rho_{n+1}^{\gamma h_{n+1}-1}
    \leq c_\nu^{-1} \frac{\mu_{n}(I_n)}{\mu_F(I_n)} \,\rho_n^{-h_n} \rho_{n+1}^{\gamma h_{n+1}-1}.
  \end{align*}
  Hence, by induction and owing the exponential convergence of $(\rho_n)_{n\in\N}$, there exist two constants $\eta>0$ and $c_1>0$ such that
  \begin{align*}
    \forall n\geq n_0;\quad \frac{\mu_{n}(I_n)}{\mu_F(I_n)} \leq c_1\,g(\rho_n)^{-\eta} \,\rho_{n}^{\gamma h_{n}-1}\quad\text{and}\quad \frac{\mu_n(I_n)}{\nu(I_n)} \leq c_1\,g(\rho_n)^{-\eta}\,\rho_{n}^{\gamma h_{n}-1}.
  \end{align*}
  Every probability measure $\mu_n$ is clearly supported by the set $\Ii(h,n)$. The Cantor structure of the latter and the Portmanteau theorem then ensure the convergence of the sequence $(\mu_n)_{n\in\N}$ to a unique measure $\mu_h$ supported by $\Ii(h)$.\vsp

  Finally, in the last part of the proof, let us prove $\mu_h$ satisfies a proper mass distribution principle. Let $a\in\Ii(h)$, $\epsilon>0$ and $r>0$ sufficiently small. There exists $n\in\N$ such that $r\in\ivfo{\rho_{n+1},\rho_{n}}$. In addition, without any loss of generality, we may suppose that $B(a,r)\subset I_{n}$, for some $I_{n}\in \Di_{n}(F)$ (otherwise, simply consider the intersection with the later). Then, the construction described in Lemma~\ref{lemma:stable_usp3_tech0} and the previous estimates entail
  \begin{align*}
    \mu_h(B(a,r)) = \frac{\mu_n(I_n)}{\nu(I_n)}\,\nu(B(a,r))
    &\leq c_1\,g(\rho_n)^{-\eta}\,\rho_{n}^{\gamma h_{n}-1} \cdot g(r)^{-\beta} r^{\gamma h_{n+1}-1+s-\epsilon} \rho_n^{-\gamma h_{n+1}+1} \\
    &\leq c_1\,g(r)^{-\eta-\beta}\,r^{\gamma h_{n+1}-1+s + \gamma(h_n-h_{n+1})-\epsilon}.
  \end{align*}
  Finally, since we know that $h_n\rightarrow h$ uniformly on the interval $\Hi$ and the previous inequality holds for any $\epsilon>0$, we obtain the desired bound on $\mu_h(B(a,r))$.
\end{proof}

We may now finally prove the lower bound of Theorem~\ref{th:levy_tree_upper_spectrum2}.
\begin{proof}[Proof of Theorem~\ref{th:levy_tree_upper_spectrum2} (lower bound)]
  Let us first observe that it is sufficient to prove for any $b>0$ that $\Nb_b$-a.e.
  \begin{align*}
    \forall h\in\ivoob{\tfrac{1-\dimH F_b}{\gamma}, \tfrac{1}{\gamma}};\quad \dimH \,\pthb{ E_\ell(h,\Ti)\cap \Ti(F_b) } \geq \gamma h - 1 + \dimH F_b
  \end{align*}
  where $F_b=F\cap\ivoo{0,b}$. Note that $F_b$ still satisfies a strong Frostman's lemma \eqref{eq:strong_frostman} if $F$ does. In addition, since the height function $h:\Ti\rightarrow\R_+$ is Lipschitz, the previous bound is a corollary of the following lower bound:
  \begin{align*}
    \forall h\in\ivoob{\tfrac{1-\dimH F_b}{\gamma}, \tfrac{1}{\gamma}};\quad \dimH \,\brcb{ a\in F_b : E_\ell(h,\Ti)\cap \Ti(F_b) \neq \vset } \geq \gamma h - 1 + \dimH F_b.
  \end{align*}
  Hence, let us set $b>0$, $s=\dimH F_b$, $\Hi\subset\ivoob{\tfrac{1-s}{\gamma}, \tfrac{1}{\gamma}}$ and $F_\star \subset F_b$ satisfying \emph{(i)-(iii)}. Then, using the notation introduced in the previous Lemma~\ref{lemma:usp_th3_cstr1} and the bound presented in Lemma~\ref{lemma:usp_spectrum_ub2}, we note that
  \begin{align*}
    \mu_h\pthbb{ \bigcup_{h'<h} \brcb{ a\in F_\star : E_\ell(h',\Ti)\cap \Ti(F_\star) \neq \vset } } = 0.
  \end{align*}
  Hence, setting
  \begin{align*}
    \widehat \Ii(h) = \Ii(h) \setminus \bigcup_{h'<h} \brcb{ a\in F_\star : E_\ell(h',\Ti)\cap \Ti(F_\star) \neq \vset },
  \end{align*}
  we observe that $\widehat\Ii(h) \subset \brcb{ a\in F_\star : E_\ell(h,\Ti)\cap \Ti(F_\star) \neq \vset }$ and $\mu_h\pthb{\widehat\Ii(h)} \in\ivoo{0,\infty}$. As a consequence, the mass distribution principle and Lemma~\ref{lemma:usp_th3_cstr1} entails that $\Nb_b$-a.e.
  \begin{align*}
    \forall h\in\ivoob{\tfrac{1-s}{\gamma}, \tfrac{1}{\gamma}};\quad \dimH \,\brcb{ a\in F_\star : E_\ell(h,\Ti)\cap \Ti(F_\star) \neq \vset } \geq \dimH \,\widehat \Ii(h) \geq \gamma h - 1 + \dimH F_\star,
  \end{align*}
  therefore proving the lower bound on the multifractal spectrum of the local time. The mass measure case is treated similarly using the property $F_\ell(h,\Ti)\subset F_\mb(h+1,\Ti)$.
\end{proof}

\subsubsection{Proof of Theorem~\ref{th:levy_tree_upper_spectrum3} (upper bound)}

The proof of the upper bound of Theorem~\ref{th:levy_tree_upper_spectrum3} is split into two technical lemmas and is mainly inspired by the work of \citet{Khoshnevisan.Peres.ea-2000}. To begin with, we investigate the case of a well-behaving compact set.
\begin{lemma}  \label{lemma:stable_usp2_tech0}
  Suppose $b>0$ and $F\subset\ivoo{0,b}$ is a compact set such that
  \begin{align*}
    \text{for all open sets $V$ s.t. }V\cap F\neq\vset\text{, }\dimBu (F\cap V) \geq s,
  \end{align*}
  for some $s>0$. Then, $F_\ell(\tfrac{1-s}{\gamma},\Ti)$ and $F_\mb(\tfrac{1+\gamma-s}{\gamma},\Ti)$ are $\Nb_b$-a.s. dense in $\Ti(F)$.
\end{lemma}
\begin{proof}
  For every $a>0$ and $n\in\N$, let us define the following collection of subtrees
  \begin{align*}
    \Tbb(a,\delta_n,\ell_n) = \brcB{ \Ti_\sigma\in\Tbb(a,\delta_n) : \inf_{u\in\ivff{\kappa\delta_n,\delta_n/\kappa}} \bk{\ell^u}(\Ti_\sigma) \geq \ell_n\delta_n^{1/\gamma}}.
  \end{align*}
  where $\kappa=\tfrac{1}{2}$ in the following and $(\ell_n)_{n\in\N}$ is a positive sequence depending on $F$ such that $\lim_{n\rightarrow\infty} \log(\ell_n) / \log(\delta_n) = -\tfrac{s}{\gamma}$. In addition, define the following subsets of $\Ti$:
  \begin{align*}
    \Ti(n) = \bigcup_{j\geq 1} \,\bigcup_{\Ti_\sigma\in{\Tbb}(j\delta_n,\delta_n,\ell_n)} \brcb{\sigma'\in\Ti_\sigma : h(\sigma',\Ti_\sigma)\in\ivoo{\delta_n,2\delta_n}} \quad\text{and}\quad \Ti^\star(n) = \bigcup_{k=n}^\infty \Ti(k).
  \end{align*}
  where $h(\sigma',\Ti_\sigma)$ denotes the height of $\sigma'$ in $\Ti_\sigma$.
  Briefly, $\Ti(n)$ gathers nodes in $\Ti$ which belong to subtrees $\Ti_\sigma\in {\Tbb}(j\delta_n,\delta_n,\ell_n)$, i.e. with a large local time at scale $\delta_n$. Since the height function is a continuous map, $\brc{\sigma'\in\Ti_\sigma : h(\sigma',\Ti_\sigma)\in\ivoo{\delta_n,2\delta_n}}$, and thus $\Ti(n)$ and $\Ti^\star(n)$, are clearly open sets. Moreover, the property satisfied by the sequence $(\ell_n)_{n\in\N}$ imply that $\cap_{n\in\N} \Ti^\star(n) \subset F_\ell(\tfrac{1-s}{\gamma},\Ti)$.

  We aim to prove that $\cap_{n\in\N}\Ti^\star(n)$ is dense in $\Ti(F)$. Due to Baire's category theorem, it is sufficient to prove that $\Nb_b$-a.e., $\Ti^\star(n)$ is dense in $\Ti(F)$ for every $n\in\N$. Hence, let $V\subset\Ti$ be an open set such that $V\cap\Ti(F)\neq\vset$. Without loss of generality, we may suppose that $V$ is a truncated subtree rooted at a level $a$ and of height $\delta>0$: $V=\tr(\Ti_\sigma,\delta)\setminus\brc{\sigma}$. According to the branching property, we know that given $Z(a,\delta)$, subtrees rooted at level $a$ are independently distributed following the measure $\Nb_\delta(\dt\Ti')$.
  As a consequence, it is sufficient to prove that if $F_a\cap\ivoo{0,\delta}\neq \vset$, where $F_a=F-a$, then for any $n\in\N$, $\Ti^\star(n)\cap\Ti(F_a\cap\ivoo{0,\delta})\neq\vset$ $\Nb_\delta(\dt\Ti)$-a.e.

  Let us set $a>0$ and $\delta>0$ such $F_a\cap\ivoo{0,\delta}\neq\vset$. $\eps>0$ can be chosen sufficiently small such that $F_a\cap\ivoo{\eps,\delta-\eps}\neq\vset$.
  Moreover, as $\dimBu (F_a\cap\ivoo{\eps,\delta-\eps}) \geq s$, there exists a subset $\N_E\subset \N$ such that
  \begin{align*}
    \lim_{n\in\N_E\rightarrow\infty} \frac{\log \Ni_n}{\log1/\delta_n} = s\quad\text{where }\ \Ni_n\eqdef\#\Di_n(F_a\cap\ivoo{\eps,\delta-\eps}).
  \end{align*}
  Since the local time on stable trees is càdlàg, $\brc{h(\Ti)>\delta} \subset \lim_{\ell_\star\rightarrow 0} \brcb{\inf_{u\in\ivff{\eps,\delta-\eps}} \bk{\ell^u}(\Ti) \geq \ell_\star }$. Therefore, let us set $\ell_\star>0$ and define the collection of events:
  \begin{align*}
    \forall u,v\in\ivff{\eps,\delta-\eps};\quad \Bi(u,v) = \brcB{\Ti : \inf_{w\in\ivff{u,v}} \bk{\ell^w}(\Ti) \geq \ell_\star }.
  \end{align*}
  Finally, for any $n\in\N$ and $k\geq 1$, let $Y(k,n) \eqdef \#\Tbb((k-1))\delta_n,\delta_n,\tfrac{1-s}{\gamma})$. Then, when $k\delta_n\in\ivoo{\eps,\delta-\eps}$, due to the branching property and Lemma~\ref{lemma:local_time_inf_tail0}, we observe that
  \begin{align*}
    \Nb_{(k-1)\delta_n}\pthcb{ Y(k,n)=0 \cap \Bi(\eps,(k+1)\delta_n) }{\Gi_{(k-1)\delta_n}}
    &\leq \exp\pthb{-c_0 \bk{\ell^{(k-1)\delta_n}} \ell_n^{-\gamma} } \indi_{\Bi(\eps,(k-1)\delta_n)} \\
    &\leq \exp\pthb{-c_0\ell_\star \ell_n^{-\gamma} } \indi_{\Bi(\eps,(k-1)\delta_n)}.
  \end{align*}
  Therefore, by induction, for any $i\in\brc{0,1}$
  \begin{align*}
    \Nb_\delta\pthbb{ \Bi(\eps,\delta-\eps) \cap \bigcap_{I_k\in\Di_{n,i}} Y(k,n)=0 } \leq c_1\exp\pthb{-c_0\ell_\star \#\Di_{n,i} \ell_n^{-\gamma} },
  \end{align*}
  where $\Di_{n,i} \eqdef \brc{I_k\in\Di_n(F_a\cap\ivoo{\eps,\delta-\eps}) : k\text{ mod }2 = i}$. Since $\Ni_n = \#\Di_{n,0}+\#\Di_{n,1}$, the latter bound entails
  \begin{align*}
    \Nb_\delta\pthbb{ \Bi(\eps,\delta-\eps) \cap \bigcap_{I_k\in\Di_n(F_a\cap\ivoo{\eps,\delta-\eps})} Y(k,n)=0 } \leq c_1\exp\pthb{-c_2 \Ni_n \ell_n^{-\gamma} }.
  \end{align*}
  We may now define precisely the sequence $\ell$: $\ell_n=g(\delta_n)\Ni_n^{1/\gamma} \geq 1$ for any $n\in\N_E$ and $\ell_n=\delta_n^{-s/\gamma}$ for any $n\in\N\setminus\N_E$. It clearly satisfies the condition $\lim_{n\rightarrow\infty} \log(\ell_n) / \log(\delta_n) = \tfrac{-s}{\gamma}$, proving that
  \begin{align*}
    \sum_{n\in\N_E} \Nb_\delta\pthbb{ \Bi(\eps,\delta-\eps) \cap \bigcap_{I_k\in\Di_n(E_a\cap\ivoo{\eps,\delta-\eps})} Y(k,n)=0 } \leq \sum_{n\in\N_E} c_1\exp\pthb{-c_2 g(\delta_n)^{-1/\gamma} } < \infty.
  \end{align*}
  Borel--Cantelli lemma therefore implies that on the event $\Bi(\eps,\delta-\eps)\cap\brc{h(\Ti)>\delta}$, for every $n\in\N$ sufficiently large, there exists $I_k\in\Di_n(F_a\cap\ivoo{\eps,\delta-\eps})$ such that $Y(k,n)\geq 1$. As a consequence, letting $\eps\rightarrow 0$, for every $n\in\N$ and every open set $V$ such that $V\cap\Ti(F)\neq\vset$, then $V\cap\Ti(F)\cap\Ti^\star(n)\neq\vset$ $\Nb_b$-a.e. The latter clearly shows that $\cap_{n\in\N}\Ti^\star(n)$, and thus $F_\ell(\tfrac{1-s}{\gamma},\Ti)$ is $\Nb_b$-a.e. dense in $\Ti(F)$. Finally, since $F_\ell(\tfrac{1-s}{\gamma},\Ti)\subset F_\mb(\tfrac{1+\gamma-s}{\gamma},\Ti)$, the same result also holds on $F_\mb(\tfrac{1+\gamma-s}{\gamma},\Ti)$.
\end{proof}

We may now obtain the complete upper bound of Theorem~\ref{th:levy_tree_upper_spectrum3}.
\begin{lemma}  \label{lemma:stable_usp2_tech1}
  Suppose $F\subset\ivoo{0,\infty}$ is an analytic set. Then, $\Nb\pth{ \dt \Ti }$-a.e.,
  \begin{align}
    \inf_{\sigma\in\Ti(F)} \alpha_\ell(\sigma,\Ti) \leq \frac{1-\dimP \,F_{|\Ti}}{\gamma}.
  \end{align}
  In addition, if $F$ is an analytic set such that for every $a>0$, $F\cap\ivoo{0,a}$ has positive packing measure or is empty, then $\Nb\pth{ \dt \Ti }$-a.e. the infimum is realized: $E_\ell(h,\Ti)\cap\Ti(F)\neq\vset$ where $h=\gamma^{-1}\pthb{1-\dimP \,F_{|\Ti}}$.

  Finally, the same two properties hold as well with the mass measure scaling exponent $\alpha_\mb(\sigma,\Ti)$.
\end{lemma}
\begin{proof}
  Let us set $b>0$ and  $s\in\ivoo{0,\dimP \,F\cap\ivoo{0,b}}$. As proved by \citet{Joyce.Preiss-1995}, there exists $F_\star\subset F$ such that for any open set $V$ intersecting $F_\star$, $\dimBu F_\star\cap V \geq s$. Then, due to Lemma~\ref{lemma:stable_usp2_tech0}, $\Nb_b$-a.e. $F_\ell(h,\Ti)\cap\Ti(F)\neq\vset$ where $h=\gamma^{-1}\pth{1-s}$. Hence, letting $s\rightarrow \dimP \,F\cap\ivoo{0,b}$, we obtain
  \begin{align*}
    \Nb_b\text{-a.e.}\quad \inf_{\sigma\in\Ti(F)} \alpha_\ell(\sigma,\Ti) \leq \frac{1-\dimP \,F\cap\ivoo{0,b}}{\gamma}.
  \end{align*}
  As a consequence, since the latter is satisfied for any $b\in\Q_+$ and $\lim_{b\rightarrow h(\Ti)} \dimP F\cap\ivoo{0,b} = \dimP F_{|\Ti}$, we get the desired upper bound
  \begin{align*}
    \Nb(\dt\Ti)\text{-a.e.}\quad \inf_{\sigma\in\Ti(F)} \alpha_\ell(\sigma,\Ti)  \leq \frac{1-\dimP \,F_{|\Ti} }{\gamma}.
  \end{align*}
  Moreover, using Proposition~\ref{prop:scaling_exponents}, we obtain an equivalent result with the mass measure scaling exponent.\vsp

  Let us now prove the second part of the lemma: suppose that for any $b>0$, $F\cap\ivoo{0,b}$ is empty or has positive packing measure. The first case is  trivial, hence, let us set $b>0$ such $\Pii^s(E\cap\ivoo{0,b})\in\ivoo{0,\infty}$, for some $s>0$ depending on $b$. Still according to the work of \citet{Joyce.Preiss-1995}, there exists a compact subset $F_\star\subset F$ which satisfies the assumption: for every open set $V$ intersecting $F_\star$, $\dimBu F_\star\cap V \geq s$. As a consequence, we may apply Lemma~\ref{lemma:stable_usp2_tech0} and obtain: $\Nb_b$-a.e., $F_\ell(\tfrac{1-s}{\gamma},\Ti)\cap\Ti(F)\neq\vset$. The common property of the packing dimension $\lim_{b\rightarrow h(\Ti)} \Pii^s(F\cap\ivoo{0,b}) = \Pii^s(F_{|\Ti})$ entails  that $\Nb$-a.e. $F_\ell(h,\Ti)\cap\Ti(F)\neq\vset$, where $h=\gamma^{-1}\pthb{1-\dimP \,F_{|\Ti}}$. Finally, according to Lemma~\ref{lemma:usp_spectrum_ub2}, $F_\ell(h',\Ti)\cap\Ti(F)=\vset$ for any $h'<h$, therefore proving the desired result. In addition, still using Proposition~\ref{prop:scaling_exponents}, we also get $F_\mb(h+1,\Ti)\cap\Ti(F)\neq\vset$.
\end{proof}


\section*{Acknowledgements}
The author would like to specially thank Leonid Mytnik for his inspiring discussions and comments on the geometry of random trees and superprocesses.

\renewcommand{\bibfont}{\normalfont\small}
\setlength{\bibsep}{2pt}

\end{document}